\documentclass[12pt,oneside]{amsart}
\usepackage[
paper=letterpaper,
headsep=20pt,text={137mm,206mm},centering,includehead
]{geometry}

\usepackage[utf8]{inputenc}
\usepackage[T1]{fontenc}
\usepackage{hyperref}

\usepackage{manfnt}

\usepackage{lmodern}

\usepackage{fnpct} 

\usepackage{xfrac}
\usepackage{float,array,multirow,calc,booktabs,extpfeil}

\usepackage{parskip}

\usepackage{scalerel}
\usepackage{mathtools}
\usepackage{bm}
\usepackage{xcolor}
\usepackage{listings}
\usepackage{mathrsfs}
\usepackage{alltt}

\usepackage[scaled=0.92]{beramono} 

\definecolor{outblue}{RGB}{0,70,200}
\definecolor{newtext}{RGB}{0,140,0} 

\usepackage{amssymb,amsxtra,bbm}
\usepackage[shortlabels]{enumitem}
\setlist[enumerate,1]{label=\textup{(\arabic*)}}
\usepackage{cancel}

\usepackage{graphicx,color}
\usepackage{faktor}

\usepackage{comment}

\usepackage{tikz}
\usepackage{tikz-cd}
\usetikzlibrary{
  cd,
  calc,
  positioning,
  arrows,
  decorations.pathreplacing,
  decorations.markings
}

\usepackage{placeins}
\usepackage{stackrel}
\newcommand{\tikzmark}[1]{\begin{tikzpicture}[overlay,remember picture]\coordinate(#1);\end{tikzpicture}}

\usepackage[normalem]{ulem}
\usepackage{soul} 

\makeatletter
\def\part{\@startsection{part}{0}
  \z@{2\linespacing\@plus\linespacing}{.5\linespacing}
  {\Large\bfseries\centering}}
\makeatother

\definecolor{todo-background-color}{gray}{0.95}
\usepackage[
  backgroundcolor=todo-background-color,
  bordercolor=black,
  linecolor=black,
  textsize=footnotesize
]{todonotes}

{}

\newtheorem{theorem}{Theorem}[section]
\newtheorem*{Realization of $n$-Surface Theorem}{Theorem}

\newtheorem*{theorem*}{Theorem}
\newtheorem{proposition}[theorem]{Proposition}
\newtheorem{corollary}[theorem]{Corollary}

\newtheorem{lemma}[theorem]{Lemma}

\theoremstyle{definition}
\newtheorem{definition}[theorem]{Definition}
\newtheorem{question}{Question}
\newtheorem*{question*}{Question}
\newtheorem*{note*}{Note}
\newtheorem{comment*}{Comment}
\newtheorem*{lemma*}{Lemma}
\newtheorem*{remark*}{Remark}
\newtheorem{example}[theorem]{Example}
\newtheorem{remark}[theorem]{Remark}

\newtheoremstyle{theorem-giventitle}
        {}{}              
        {\itshape}                      
        {}                              
        {\bfseries}                     
        {.}                             
        { }                             
        {\thmnote{\bfseries#3}}
\theoremstyle{theorem-giventitle}
\newtheorem{theorem-named}{}
\newtheorem{question-named}{}
\newtheoremstyle{definition-giventitle}
        {}{}              
        {}                      
        {}                              
        {\bfseries}                     
        {.}                             
        {.7em}                             
        {\thmnote{\bfseries#3}}

\def\Z{\mathbb{Z}}
\def\Q{\mathbb{Q}}
\def\R{\mathbb{R}}

\def\qz{\mathbb{Z}_{(2)}/\mathbb{Z}}

\def\Im{\operatorname{Im}}

\makeatletter
\def\tpmod#1{{\@displayfalse\pmod{#1}}}
\makeatother

\def\lk{\mathrm{lk}}

\def\setminus{\smallsetminus}

\def\isomto{\mathrel{\smash{\xrightarrow{\smash{\lower.4ex\hbox{$\scriptstyle{\cong}$}}}}}}
\def\isomfrom{\mathrel{\smash{\xleftarrow{\smash{\lower.4ex\hbox{$\scriptstyle{\cong}$}}}}}}

\def\defeq{\stackrel{\mathrm{def}}{=}}

\makeatletter
\def\l@subsection{\@tocline{2}{0pt}{2.5pc}{5pc}{}}
\makeatother

\begin{document}

\title[Extending Quotients of Knot Groups]
{Extending Quotients of Knot Groups\\ Over Surfaces in $B^4$}

\author{Alexandra Kjuchukova}
\address{
Department of Mathematics\\
University of Notre Dame\\
Notre Dame, Indiana 46656, 
USA
}
\email{akjuchuk@nd.edu}

\author{Kent E. Orr}
\address{Department of Mathematics\\
  Indiana University\\
  Bloomington, Indiana 47405
  \\USA}
\email{korr@iu.edu}

\def\subjclassname{\textup{2010} Mathematics Subject Classification}
\expandafter\let\csname subjclassname@1991\endcsname=\subjclassname
\expandafter\let\csname subjclassname@2000\endcsname=\subjclassname
\subjclass{%
}

\begin{abstract}
Let $K\subseteq S^3$ be a knot with exterior $E_K$, and denote by $\rho\colon \pi_1(E_K)\twoheadrightarrow G$ a quotient of its group. We give a sharp obstruction to the existence of a connected, oriented, smooth surface $F\subseteq B^4$ with $\partial F = K$ over whose exterior $\rho$ extends surjectively. Equivalently, we determine whether the cover of $S^3$ branched over $K$ and induced by $\rho$ bounds a connected cover of $B^4$ branched along such a surface. When $G$ is a dihedral group, we show the obstruction can be computed by evaluating the Seifert form of $K$ on a single curve, a so-called characteristic knot associated to $\rho$. When the dihedral obstruction vanishes, we construct the surface $F$ explicitly.
\end{abstract}
\maketitle
\thispagestyle{empty}

\vspace{1cm}

\renewcommand{\contentsname}{}
{\small \tableofcontents
\setcounter{tocdepth}{2}}

\newpage

\part*{ Introduction}
\section{The central question}

In Theorem~\ref{theorem:G-surface} and Corollary~\ref{thm:main}, we address the following: {\it Does a given connected cover $f: M^3\to S^3$ branched over a knot $K$ bound a connected cover of $B^4$ branched over a connected, orientable, smooth or locally flat surface $F\subseteq B^4$ with $\partial(F)=K$?} 

To resolve this question, we reformulate it as follows:

\begin{question}\label{q:main-v2}
    Given a knot $K\subseteq S^3$ and a quotient $\rho: \pi_1(E_K)\to G$, does there exist a surface $F\subseteq B^4$ (connected; orientable; smooth or locally flat) with $\partial(F)=K$ and a homomorphism $\overline{\rho}:\pi_1(E_F)\twoheadrightarrow G$ such that the following diagram commutes?
    \begin{equation}
\label{eq:comm-triangle-general}
\begin{tikzcd} 
\pi_1(E_K) \arrow[r,"\rho", two heads]\ar[d, "i_\ast"] & G.\\
\pi_1(E_F)\arrow[ur,dashed, two heads, "\overline{\rho}" ']
\end{tikzcd}
\end{equation}
\end{question}

In Part~I of the paper we develop a general obstruction theory for the existence of such an extension, valid for homomorphisms onto any based group~$G$. Given a knot $K$ and a based surjection $\rho\colon\pi_1(E_K)\twoheadrightarrow G$, we construct a space $KG$, of well-defined homotopy type; and we associate to $\rho$ a homotopy class $\Theta_\rho(K)\in\pi_3(KG)$. We prove, in Theorem~\ref{theorem:G-surface}, that the vanishing of
$\Theta_\rho(K)\in\pi_3(KG)$ is equivalent to the existence of a $G$-surface in $B^4$ (Definition~\ref{def:G-surface}) with $\partial F = K$ over whose exterior $\rho$ extends.  Being a $G$-surface is a priori more restrictive than the condition in Question~\ref{q:main-v2}; however, for a large class of groups we call {\it taut} (Definition~\ref{tautgroup}), the two conditions coincide, so the vanishing of $\Theta_\rho(K)$ answers Question~\ref{q:main-v2} without modification.

Part~II of the paper specializes to the case of dihedral groups, for which we show that our obstruction can be related to classical invariants and evaluated easily. Applying the general theory to dihedral groups yields the following computable criterion, which confirms a conjecture of the first author.

\begin{theorem}\label{thm:intro-dihedral}
Let $K\subseteq S^3$ be a knot equipped with a surjection $\rho\colon\pi_1(E_K)\twoheadrightarrow D_n$. Let $S$ be a Seifert surface for $K$, $V$ a Seifert matrix determined by $S$, and $\beta\subseteq\mathring{S}$ a mod~$n$ characteristic knot $($Definition~\ref{def:chark}$)$ corresponding to $\rho$. Then $\rho$ extends surjectively over the exterior of an orientable, locally flat surface $F\subseteq B^4$ with $\partial F = K$ if and only if
\[
[\beta]^T\cdot(V+V^T)\cdot[\beta] \equiv 0 \pmod{n^2}.
\]
\end{theorem}

Theorem~\ref{thm:intro-dihedral} gives one of several equivalent conditions we find for the existence of a surface $F$ as above. The other conditions are established in Section~\ref{sec:core-criterion}; see, in particular, Theorem~\ref{thm:equivalence}. 

We also note that, in the case $n=3$, the ``only if'' direction of the above theorem can be proved without the machinery developed here, as an application of the $\Xi_3$ formula from~\cite{kjuchukova2018dihedral}.  Proving the necessary condition for general~$n$ required developing the obstruction theory of Part~I. Our proof of the sufficient condition relies on the construction provided in Section~\ref{sec:n-bottle}.

\begin{remark*}
    Since $K$ is a knot, the existence of a surjective homomorphism $\rho$ as above implies that $n$ is odd, so we will work with odd $n$ whenever dihedral groups are concerned. This entails no loss of generality:  $H_1(E_K)\cong\mathbb{Z}$ while  $H_1(D_n)\cong\mathbb{Z}_2\times\mathbb{Z}_2$ for $n$ even. Thus, no surjection
$\pi_1(E_K)\twoheadrightarrow D_n$ can exist unless $n$ is odd. 
\end{remark*}

\subsection{Motivation}\label{sec:motivation}
Metacyclic invariants of knots and links are a classical subject stemming from seminal work by Casson-Gordon~\cite{cass-gor1986cobordism}, Cappell-Shaneson~\cite{CS1984linking} and others going back to Reidemeister~\cite{reidemeister1929knoten} and Fox~\cite{fox1962quick}. In their most general form, metacyclic knot invariants arise from a diagram of covering spaces (all except $\widehat{f}_p$ branched):

\begin{center} 
\begin{equation}\label{eq:cover-diagrams}
\noindent
\begin{minipage}{0.34\textwidth} 
\centering
\begin{tikzcd}[column sep=3pt] 
& \rho^{-1}(e) \ar[dl, hook', labels=above left] 
     \ar[dr, hook'] 
     \ar[dd, hook']
&
\\
\rho^{-1}(\mathbb{Z}_q) \ar[dr, hook']
& 
& \rho^{-1}(\mathbb{Z}_p) \ar[dl, hook']
\\
& \rho^{-1}(\mathbb{Z}_p\rtimes \mathbb{Z}_q)
&
\end{tikzcd}
\end{minipage}
\hspace{1.6cm}
\begin{minipage}{0.46\textwidth} 
\centering
\begin{tikzcd}[column sep=25pt]
& M_{pq}  
     \ar[dl, "\widehat{f}_q", labels=above left]  
     \ar[dr, "\widehat{f}_p"] 
     \ar[dd, "f_{pq}"]
&
\\
M_p \ar[dr, "f_p"']
& 
& \Sigma_q(K) \ar[dl, "f_q"]
\\
& S^3
&
\end{tikzcd}
\end{minipage}
\end{equation}
\end{center}

On the left is a diagram of subgroups of a metacyclic group $M_{pq}\cong\Z_p\rtimes_\varphi \Z_q$; on the right is the corresponding diagram of branched covering spaces of a knot $K\subseteq S^3$. As usual, $\Sigma_q(K)$ denotes the cyclic $q$-fold branched cover. The $p$-fold cover $M_p\to S^3$ is irregular, branched along $K$. Classical knot invariants --- Casson-Gordon signatures, Tristram-Levine signatures~\cite{levine1969invariants, tristram1969some}, and the $\Xi_p$ invariant\footnote{In the literature to date, the $\Xi_p$ invariant is restricted to the case where $q=2$ and the group $\Z_p\rtimes_\varphi\Z_q$ is dihedral. However, it exists in general~\cite{kjsh2026linear}.} introduced in~\cite{kjuchukova2018dihedral} --- can all be computed by extending the appropriate covers in~\eqref{eq:cover-diagrams} over $B^4$, branched along a surface $F$ with $\partial(F)=K$. This was one of the motivations for the extension problem posed in Question~\ref{q:main-v2} and answered in Theorem~\ref{theorem:G-surface} and Corollary~\ref{thm:main}.

When $f$ is a cyclic branched cover, an extension always exists, since any map $H_1(S^3\backslash K;\Z)\to \Z_n$ factors through $H_1(B^4\backslash F;\Z)\to \Z_n$ for any orientable surface $F\subseteq B^4$ with boundary $K$. For covers induced by metacyclic quotients, however, the answer is far from obvious. Prior to this work, as far as we are aware, the only known obstruction was in the case of $D_3,$ the dihedral group of order 6, as an application of~\cite{kjuchukova2018dihedral}; see Theorem~\ref{thm:n=3} and its proof. 

In the dihedral case, Cappell and Shaneson~\cite{CS1984linking, CS1975branched} give an explicit cobordism $W(K, \beta)$ between the dihedral branched cover of $K$ and a cyclic branched cover of a mod~$p$ characteristic knot $\beta$ for $K$ (Definition~\ref{def:chark}); moreover, $W(K, \beta)$ branch covers $S^3\times I$. This construction leads to a formula for the $\Xi_p$ invariant (Equation~\eqref{eq:Xi}; see Section~\ref{sec:xi}). However, the branching set in this construction is not a manifold, limiting the use of analytical tools such as the $G$-signature theorem. In Theorem~\ref{thm:equivalence}, we formulate a necessary and sufficient condition, in terms of the Seifert form of $K$, for when the singular branching set can be replaced by a manifold. When this is possible, we show how to construct such a surface explicitly; see Section~\ref{sec:n-bottle}.

\begin{remark}\label{rem:cone-point}
    Note that if we do not require that $F$ be locally flat, the existence question becomes trivial. Setting $F=c(K),$ the cone on $K,$ we have that $i_\ast: \pi_1(E_K)\xrightarrow{\cong} \pi_1(E_F),$ so $
    \bar{\rho}$ always exists. More generally, if we allow additional cone singularities on $F$, each cone point will contribute to the signature of its cover, as shown for dihedral covers in~\cite{kjuchukova2018dihedral}, which would hinder our effort to obtain invariants of $K$ from the intersection form of a cover branched along $F$. 
\end{remark}

\begin{remark} \label{rem:category}
    The existence of a surface over whose complement a given quotient $\rho: \pi_1(E_K)\to G$ extends is  {\em independent of category.} That is, if $\rho$ can be extended over a topologically locally flat surface,  then there is also a smooth surface which admits such an extension; see Corollary~\ref{cor:categories}.
\end{remark}

\begin{remark} \label{con:orientability}
    In the case where $G$ is a dihedral group, an orientable surface which fits into the commutative diagram~(\ref{eq:comm-triangle-general}) exists if and only if a non-orientable one does; this follows from Theorem~\ref{cor:nonor-variants}. We also show in Theorem~\ref{cor:disk-in-W} and that the same obstruction detects the existence of a surface in a general orientable four-manifold with $S^3$ boundary.
\end{remark}

Lastly, we emphasize that, while our initial interest in Question~\ref{q:main-v2} stemmed from a study of metabelian knot invariants, our Theorems~\ref{thm:main} and~\ref{theorem:G-surface} make no such restriction. It is classically known that every finitely generated group which has weight one is the quotient of a knot group~\cite{gonzalez1975homomorphs, johnson1980homomorphs}, and our results apply in this most general setting.

\subsection{Organization of the paper and summary of results} \label{sec:paper-summary}

In Section~\ref{sec:obstruction}, we define a space $KG$ whose homotopy type carries the obstruction to the existence of a surface $F$ as in Question~\ref{q:main-v2}, and we define the invariant $\Theta_\rho(K)\in\pi_3(KG)$.  In Section~\ref{sec:g-surf-thms} we prove the $G$-Surface Theorem~\ref{theorem:G-surface}: the vanishing of $\Theta_\rho(K)$ is equivalent to the existence of a $G$-surface (Definition~\ref{def:G-surface}) which fits into Diagram~\eqref{eq:comm-triangle-general}.  
Corollary~\ref{cor:categories} establishes that the answer is independent of category (smooth or locally flat).

Part~II applies the general obstruction theory to the case of dihedral groups. In Section~\ref{Sec:pairings}, we collect the algebraic and topological preliminaries needed. 
In Section~\ref{sec:dihedral-criteria} we specialize the obstruction $\Theta_\rho$ to the case of dihedral groups.  Theorem~\ref{thm:equivalence} gives several equivalent and computable conditions for the existence of an orientable dihedral surface in $B^4$. Theorem~\ref{cor:Dn-Dn-bottle} proves an analogous result for the closely related $n$-bottle groups (Definition~\ref{def:n-bottle}). Theorems~\ref{cor:nonor-variants} and~\ref{cor:disk-in-W} address the existence of non-orientable surfaces in $B^4$ and of surfaces in more general four-manifolds with boundary $S^3$.

In Section~\ref{sec:n-bottle} we build an explicit dihedral surface by hand whenever the Seifert matrix condition in Theorem~\ref{thm:intro-dihedral} is satisfied. Section~\ref{sec:proof-of-conjecture} is dedicated to applications and computations.  
Of independent interest is Theorem~\ref{thm:n=3}, which shows that 
when $3 \mid n$, the Seifert-matrix condition 
$[\beta]^T(V+V^T)[\beta] \equiv 0 \pmod{9}$ is necessary for the 
existence of \emph{any} surface for $K$ -- orientable 
or not, in any oriented $4$-manifold with $S^3$ boundary -- over which the dihedral quotient determined by $\beta$ extends. This 
result, which predates the present work and provided its initial 
impetus, follows from the $\Xi_n$ formula of~\cite{kjuchukova2018dihedral} 
without the obstruction theory of Part~I. For $n = 3$, the condition 
is also sufficient, yielding a complete characterization; see 
Section~\ref{sec:n3-thm}.

\part*{Part I. General obstruction theory}

\section{Based quotients of knot groups and the obstruction space}\label{sec:obstruction}

Given a knot $K\subseteq S^3$ and a based epimorphism $\rho\colon \pi_1(E_K)\twoheadrightarrow G$, we construct a space $KG$ and a homotopy class $\Theta_\rho(K)\in\pi_3(KG)$ that is the obstruction to extending $\rho$ over the exterior of a surface in $B^4$. This section defines the relevant objects: based groups, $G$-knots, the space $KG$, and the invariant $\Theta_\rho(K)$. 

The invariant $\Theta_{\rho}(K) \in \pi_3(KG)$ is defined homotopy-theoretically; the proof that it captures the geometry of  the extension problem relies on transversality and handle  decompositions. In Part~II, where we specialize to dihedral groups, we 
translate the obstruction into intersection-theoretic data that can be computed from a Seifert matrix; see Theorem~\ref{thm:equivalence}.

\subsection{Based groups and $G$-knots}\label{sec:based-groups}

Throughout, $K$ will always denote a knot in $S^3$ and $F$ a surface in a 4-manifold with $S^3$ boundary (usually, but not always, $B^4$) such that $\partial(F) = K$.  We require that $F$ be smooth or locally flat but rarely distinguish between the two: as noted in Remark~\ref{rem:category}, for our purposes the two cases are equivalent.  We use $E_K$ and $E_F$ to denote the exteriors of $K$ and $F$ in their ambient manifolds.

When we speak of the knot group, $\pi_1(E_K),$ we will always assume that the basepoint is contained on the boundary of $E_K$, so that the 0-framed longitude of $K$ and a preferred meridian (the boundary of the fiber of the normal $B^2$ bundle on $K$ containing the basepoint) unambiguously represent elements of the group. 

Throughout this paper we consider $G$-knots, that is, knots $K$ equipped with quotients $\rho: \pi_1(E_K)\twoheadrightarrow G$ as in Definition~\ref{def:G-Knot}. Knots and surfaces are oriented
unless otherwise indicated. $G$ will always denote a discrete group. Whenever $G$ is the quotient of a knot group, the abelianization of $G$ is cyclic.  Since $G$ is discrete, the classifying space, $BG$, is the Eilenberg-MacLane space, $K(G, 1)$. 

When $G$ is a quotient of a knot group, we will regard $G$ as a {\em based group}, see Definition~\ref{BasedDefn}. Knot groups will always be based by a choice of meridian, and, given an epimorphism $\rho: \pi_1(E_K)\twoheadrightarrow G$, the image of this meridian in $G$ will serve as a basing for the group $G$.

\begin{definition}\label{BasedDefn}
Let $G$ be a group with $H_1(G) \cong \Z$ or $H_1(G) \cong \Z_n$ and such that the abelianization homomorphism $G\to G^{ab}$ splits. Choose a  splitting homomorphism $\iota \colon G^{ab} \to G$ of the abelianization homomorphism on the group $G$. By assumption, $G^{ab}$ is cyclic, so the value of $\iota(1)$ determines the above splitting. We say $(G, \iota(1))$ is a {\em based group}.
\[
\begin{tikzcd}
G^{ab} \arrow[r,"\iota"] \ar[rr, bend right=50,"id"] & G \arrow[r,"ab"] & G^{ab}
\end{tikzcd}
\]
\end{definition}

\begin{example}
    If $G^{ab}\cong\Z,$ $G$ always admits a basing since $\Z$ is free, that is, the abelianization homomorphism always splits.
\end{example}

                                    \begin{example}
                                        If $G\cong D_n,$ a dihedral group with $n$ odd, the abelianization homomorphism $D_n\to \Z_2$ splits and a choice of basing is equivalent to a choice of reflection in $D_n$.
                                    \end{example}
                                    
                                    \begin{definition}\label{defn:basing}
                                    A {\em based homomorphism} between based groups  $(G, \iota_G(1))$ and $(H, \iota_H(1))$ is a group homomorphism $\rho \colon G \to H$ which preserves the basing, that is, $\rho(\iota_G(1))=\iota_H(1)$. 
                                    \end{definition}
                                    
                                    A composition of based homomorphisms is a based homomorphism.	
                                    
                                    \begin{example}
                                        Let $K\subseteq  S^3$ be a knot and $\mu\in \pi_1(E_K)$ a class represented by a meridian of $K$. Then $(\pi_1(E_K), \mu)$ is a based group. When we refer to a knot group as {\em based}, we will always mean that we have fixed a preferred meridian. 
                                    \end{example}

                                    \begin{definition}~\label{def:G-Knot}
                                    Let $G$ be a based group.  A {\em $G$-knot}, $(K, \rho)$,  is a knot $K\subseteq S^3$, together with a map $\rho\colon E_K \to BG$ satisfying the following conditions:

                                    \begin{enumerate}
                                    \item[i)]  $\rho_{\ast}: \pi_1(E_K)\twoheadrightarrow{G}$ is a based epimorphism.
                                    \item[ii)] The image of the $0$-framed longitude for $K$ is nullhomotopic in $BG$.
                                    \end{enumerate}
                                    When the second condition is satisfied, we will assume that the image of the $0$-framed longitude for $K$ is the basepoint in $BG$. 
                                    \end{definition}
                                    
                                    \begin{example}
                                    If $K$ is a knot with an epimorphism $\rho \colon \pi_1(E_K) \to \mathbb{D}_n$,  we call $K$ a $\mathbb{D}_n$-knot.  If the context demands more detail, we may say that $(K,\rho)$ is a $\mathbb{D}_n$ knot: condition {\it ii)} is satisfied automatically since the second derived subgroup of $\mathbb{D}_n$ is trivial. Similarly, a knot $K$ equipped with a surjective homomorphism $\rho \colon \pi_1(E_K) \to D_n$ is a $D_n$-knot. We might also call such a $K$ a dihedral knot, if the value of $n$ is either clear from the context or irrelevant.
                                    \end{example}

                                    \begin{example}
                                    Suppose $G$ is metabelian with $G^{ab}\cong\Z$.
The $0$-framed longitude for $K$ lies in the commutator subgroup $\pi'$ (since it bounds a Seifert surface) and commutes with the meridian $\mu$. Since $\mu - 1$ acts as an isomorphism on $\pi'/\pi''$ (see Lemma~\ref{lem:meta-is-taut}), the longitude in fact lies in $\pi''$. As $G$ is metabelian, $G'' = 1$, so the image of the longitude in $G$ is trivial. Thus, 
                                    given any choice of meridian for $K$ to use as a basing, the existence of an epimorphism $\pi_1(E_K) \to G$ automatically implies that $K$ is a $G$-knot.
                                    \end{example} 
                                    
                                    \begin{example}\label{eg:MK} Let $K\subseteq S^3$ be a knot and denote  $0$-surgery on $K$ by $M_K$. Choose a meridional element of the knot group.  We have that $K$, equipped with the inclusion-induced epimorphism 
                                    \[
                                    \pi_1(E_K) \to \pi_1(M_K)=: G,
                                    \]
                                    is a $G$-knot.
                                    \end{example}

                                    \subsection{The space $KG$ and the invariant $\Theta_\rho(K)$}\label{Gknots}
                                    
                                    Given a $G$-knot $K$, we now construct the space $KG$,  announced in the introduction, whose homotopy type carries the 
                                    obstruction to the existence of a $G$-surface bounded by $K$.
                                    
                                     As usual, we fix a meridian $\mu$ of $K$ and regard $\rho \colon \pi_1(E_K) \to G$ as a based epimorphism, by endowing $G$ with the basing $\iota(1): = \rho(\mu)$. 
                                    We denote the Eilenberg-MacLane space $K(G, 1)$ by $BG$, since $K(G,1)$ is the classifying space for the discrete group $G$. Let $S := \partial B^2$ denote the $1$-sphere with a fixed basepoint, and let $S\hookrightarrow B$  be a basepoint-preserving inclusion of CW complexes whose image is a loop that represents the basing $\iota(1) \in G$. (We can assume that $S \hookrightarrow BG$ is an inclusion of CW complexes by replacing $BG$ with the (based) mapping cylinder of a cellular map $S \hookleftarrow BG$.)
                                    
                                    Let $KG$ denote the push-out given by the following diagram of spaces.  Since inclusions of CW complexes are cofibrations, the pushout given below is well-defined up to homotopy type, and we may regard the 2-cell $B^2$ as embedded in $KG$. We will refer to $B^2 \subseteq  KG$ as the {\em distinguished 2-cell in $KG$}. 
                            
\begin{center}
  
\begin{tikzcd}[column sep=large]
S = \partial B^2 \ar[r, hook] \ar[d, hook', "i", shift left=.7em]	& BG \ar[d, dashed, hook']\\
B^2 \ar[r, dashed, hook] & BG \cup_{S} B^2 \defeq :KG
\end{tikzcd} 
  
\end{center}
                                    
This completes our construction of the space $KG$ where we will hang our $G$-knot invariant.  Here, $K$ is a $G$-knot with respect to the based epimorphism $\rho \colon \pi_1 (E_K )\to G$. That is, the image of the longitude of $K$ is trivial in $G$. We consider the composition%
                                    \[
                                    \begin{tikzcd}
                                    \pi_1(\partial E_K) \ar[r, "inc_\ast"] &  \pi_1(E_K)  \ar[r, "\rho"]   &   G. 
                                    \end{tikzcd}
                                    \] 
Recall that the basepoint for $\pi_1(E_K)$ lies in $\partial{E_K}$.  Parameterize  the boundary, $\partial E_K$, as $\lambda \times \mu$ where $\mu$ is the basing for $\pi_1(E_K)$ and $\lambda$ is the $0$-framed longitude of $K$. 
The above composition of homomorphisms is induced by a map
\[
\wp\colon \partial E_K \to BG.
\]
Up to homotopy, we may assume this map is given by projection onto the meridian, $\mu$, followed by a degree one map to the previously defined image of the based meridian, $S \subseteq BG$.  Since $BG$ is an Eilenberg-MacLane space, and since  
                                    \[
                                    \rho_{\vert \pi_1(\partial E_K)} = (\wp_\vert)_{\ast}
                                    \]
                                    the map $\wp$ determines an  extension  to a map $E_K \to BG$, unique up to homotopy rel boundary, and inducing the epimorphism $\rho$.

                                    Denote a regular neighborhood of $K$ in $S^3$ by $N(K)$. Also,
                                    denote by $\theta$ the map on pushouts induced by the following diagram, where the front-most square of maps is the square used to define~$KG$. Since $BG$ is a $K(\pi, 1)$ space, by a slight nonchalance with notation we will sometimes use $\rho$ to mean both the map on spaces and the induced homomorphism of fundamental groups.
                                    \begin{equation}\label{eqn:theta}
                                    \begin{tikzcd}
                                    \partial E_K \ar[r, hook,"inc"] \ar[d,hook',"inc"] & E_K 
                                    \arrow[hook', d, near start, "inc"] \ar[rrd,"\rho"] & \phantom{f}\\
                                    N(K) \arrow[r,"inc",hook] \ar[rrd, "proj_{B^2}" ']  & S^3 \ar[rrd, blue, near start, "\theta", dashed]  & S \ar[r,hook] \ar[from=llu , crossing over] &   BG \ar[d,"inc", hook']\\
                                    &  & B^2 \ar[r, "inc"',hook] \ar[from={u}, near start, "inc", crossing over, hook'] & KG
                                    \end{tikzcd}
                                    \end{equation}

                                    \begin{definition}\label{inv_univ} With the notation established above, define 
                                    \[
                                    \Theta_{\rho}(K): = [{\color{blue}\theta}] \in \pi_3(KG).
                                    \]	
                                    \end{definition}

                                    \noindent We observe that the knot $K$ and the based epimorphism $\rho \colon \pi_1 (E_K )\to G$ determine the map $\theta$ up to based homotopy, and thus, also determine the $G$-knot invariant $\Theta_{\rho}(K)$.
                                    
                                    \begin{remark}\label{rem:basing-stuff}
                                        By definition, $\Theta_{\rho}(K)$ is an invariant of a $G$-knot, $K$, equipped with a quotient $\rho \colon \pi_1(E_K) \to G$, where $G$ is a {\em based} group.  In other words, part of the structure of a $G$-knot is a choice of basing for the group $G$, which is determined by a choice of meridian, $\mu$, of $K$.  Therefore, strictly speaking, we have defined an invariant $\Theta_{\rho}$ of the triple $(K, \rho, \mu)$. It is clearly preferable to avoid specifying the auxiliary information that is the choice of $\mu$, and indeed we can avoid it, as explained next.

                                        Given a based epimorphism $\pi_1(E_K)\to G$, consider a different choice of meridian, $\mu'$, of $K$. Since any two meridians of $K$ are conjugate and $E_K$ is a $K(\pi, 1)$ space, there is a homotopy equivalence of pairs,  $(E_K, \partial{E_K}) \to (E_K, \partial{E_K})$, sending $\mu$ to $\mu'$ (and extending to a homotopy equivalence of $S^3$ to itself). 
                                        Moreover, since $\rho$ is surjective, there is an inner automorphism of $G$ sending $\rho(\mu)$ to $\rho(\mu')$. This induces a homotopy equivalence of $BG$, which in turn extends to a homotopy equivalence $KG \xrightarrow{_{\sim}}  KG$ sending the basing $\rho(\mu)$ to $\rho(\mu')$. This induces an automorphism of the group $\pi_3(KG)$. The above homotopy equivalences fit into the commutative diagram~\eqref{eqn:basing-indep}.

                                    \begin{equation}\label{eqn:basing-indep}
                                    \begin{tikzcd}
                                     E_K \ar[r, "\simeq"] \ar[d,hook',"inc"] & E_K 
                                    \arrow[hook', d, near start, "inc"] \ar[rrd,"\rho"] & \phantom{f}\\
                                    S^3 \arrow[r,"\simeq"] \ar[rrd,  blue, "\theta" ', dashed]  & S^3 \ar[rrd, blue, near start, "\theta'", dashed]  & BG \ar[r,"\simeq~"] \ar[from=llu , crossing over, near end, "\rho"] &   BG \ar[d,"inc", hook']\\
                                    &  & KG \ar[r, "\simeq"] \ar[from={u}, near start, "inc", crossing over, hook'] & KG
                                    \end{tikzcd}
                                    \end{equation}
                                     
                                        It follows that the induced isomorphism on $\pi_3(KG)$ sends the class $\theta$ of diagram~(\ref{eqn:theta}) determined by the meridian $\mu$ to the class $\theta'$ determined by the meridian $\mu'$. Therefore, the vanishing of $\Theta_{\rho}(K)$ is independent of the choice of basing for $G$ or, equivalently, of a choice of meridian for the knot $K$. Since we shall be primarily concerned with whether or not $\Theta_{\rho}(K)=0,$ we can, indeed, ignore the choice of meridian, and we will suppress $\mu$ from the notation.
                                    \end{remark}
                                    

                                    \section{$G$-surfaces and the general existence theorem}\label{sec:g-surf-thms}
                                    
                                    We now prove that $\Theta_\rho(K)$ is the obstruction to the existence of a $G$-surface extending $\rho$. We first gather the geometric preliminaries needed for the proof, then state and prove the $G$-Surface Theorem, and finally derive consequences for a class of groups we call taut.
                                    
                                    \subsection{Geometric preliminaries on surface exteriors}\label{sec:geometric-prelim}

                                    In order to keep the manuscript self-contained, we prove two well-known results about codimension-two embeddings of oriented manifolds in $\R^k$, $k = 3, 4$.   
                                    We use these results for the proof of the $G$-Surface Theorem.
                                    
                                    \begin{lemma}\label{Lemma:2_disk_bundle}
                                    Let $( F, K ) \subseteq (B^4, S^3)$ be an inclusion of pairs, where $F$ is an oriented surface with boundary $K$.  Then $F$ has a trivial normal bundle.  That is, the embedding $(F, K) \subseteq (B^4, S^3)$ extends to an embedding $(F, K) \times B^2 \subseteq (B^4, S^3)$.
                                    \end{lemma}
                                    
                                    \begin{proof} 
                                    The Euler class of the normal bundle of a codimension-two oriented surface in $B^4$ vanishes -- see, for example,~\cite[Corollary 11.4]{milnor1974characteristic}.  Therefore, this normal bundle has a non-vanishing section.  Since the bundle is 2-dimensional, the bundle is trivial.
                                    \end{proof}

                                    \begin{lemma}\label{lemma:0-framed-Section}
                                    Let $(F, K) \subseteq  (B^4, S^3)$ be a properly embedded, oriented surface and let $p \colon E_F \to S^1$ induce the abelianization epimorphism $p \colon \pi_1(E_F) \xrightarrow{ab} \Z$. Let $S(F) \cong F \times S^1$ be the normal sphere bundle of $F \subseteq B^4$. Then there is a section, $F \xrightarrow{s} S(F) \subseteq \partial{E_F}$, of the sphere bundle $S(F) \to F$ parameterizing $F \times S^1$ so that the composition
                                    \[
                                    F \xrightarrow{s} S(F) \xrightarrow{inc} E_F \xrightarrow{p} S^1
                                    \] 
                                    is nullhomotopic (where we denote by $p$ both the map on spaces and the induced map on the fundamental group).
                                    \end{lemma}
                                    
                                    \begin{definition}\label{def:longF} We call the image, $s(F) \subseteq S(F)$, of a section as in Lemma~\ref{lemma:0-framed-Section}, a {\em $0$-framed section} or a {\it longitude} of $F$ in $S(F) \subseteq E_F$.
                                    \end{definition}
                                    
                                    \begin{proof}[Proof of Lemma~\ref{lemma:0-framed-Section}]
                                    We show a section $s\colon F \to S(F)$ exists by making the above composition nullhomotopic.  
                                    By Lemma~\ref{Lemma:2_disk_bundle}, the sphere bundle $\pi \colon S(F) \to F$ is trivial.  Choose arbitrary bundle coordinates $F \times S^1  \xrightarrow{\iota} \partial{E_F}\cong S(F)$ for the total space of this bundle.  We have a commutative diagram of maps
                                    \begin{equation}\label{triangle}
                                    \begin{tikzcd}
                                    F \times S^1  \arrow[r, "\iota"]  \arrow{rd}[swap]{proj_{_F}}  & S(F)  \ar[d, "\pi", labels=left] \arrow[r, "inc"] & E_F   \ar[r, "p"] & S^1.\\
                                    & F
                                    \end{tikzcd}
                                    \end{equation}
                                    Let $(a, b)\in F\times S^1=S(F)$ denote a basepoint.
                                    Note that sections of the trivial bundle $F \times S^1 \xrightarrow{proj_{F}} F$ are continuous maps $s \colon F \to F \times S^1$ of the form $x \mapsto (x, \alpha(x))$. 
                                    
                                    Suppose the rank of $H_1(F) = n$.  Choose a minimal $1$-skeleton for $F$ consisting of $n$ loops $\{c_i\},$  $i = 1, \ldots, n,$ each based at the projection of the basepoint, $a\in F$. We can choose our section $x \mapsto (x, \alpha(x)) \in F \times S^1$ so that the degree of the following composition is zero on each $1$-cell $c_i$ of $F$, $i = 1, \ldots, n$.
                                    \begin{equation}\label{comptoKPi1}
                                    x \mapsto (x, \alpha(x)) \mapsto p \circ inc \circ \iota (x, \alpha(x)).
                                    \end{equation}
                                    We do so by choosing $\alpha \colon c_i \to S^1$ so that the 
                                    \[ 
                                    \text{deg}(\alpha) = -\text{deg}(p\circ inc \circ \iota{([c_i\times \{b\}])}).
                                    \] Since $S^1$ is a $K(\pi, 1)$ 
                                    space, and the displayed composition of functions, (\ref{comptoKPi1}), is null-homotopic on the $1$-skeleton, this section extends over $F$, producing the desired section, $s$.
\end{proof}

\begin{remark}
The section $s$ produced above can in fact be chosen so that $p\circ\mathrm{inc}\circ s$ is constant. Since $\phi:=p\circ\mathrm{inc}\circ\iota\circ s\colon F\to S^1$ is null-homotopic, it lifts to a map $\tilde{\phi}\colon F\to\mathbb{R}$. Because the bundle $S(F)\cong F\times S^1$ is trivial and $p$ restricts on each fiber $\{x\}\times S^1$ to a map of degree $\pm 1$, we may define a new section $s'(x)=(x,\,\alpha(x)-\tilde{\phi}(x))$, which is again a global section and satisfies $p\circ\mathrm{inc}\circ\iota\circ s'\equiv\mathrm{const}$. That said, since all applications of this lemma require only that $s$ induce the zero map on $\pi_1$, we omitted this condition.
\end{remark}

\subsection{The $G$-Surface Theorem}\label{sec:g-surf-thm}

\begin{definition}\label{def:G-surface}. Let $K\subseteq S^3$ be a $G$-knot with respect to a map $\rho: E_K\to BG$. A $G$-surface for $K$ is an oriented surface $F \subseteq B^4$ equipped with a map $\bar{\rho} \colon E_F \to BG$ satisfying the following$\colon$

\begin{enumerate}
\item[i)] $\bar{\rho}$ extends the epimorphism $\rho$, that is, $\bar{\rho}$ fits into the commutative diagram below.
\item[ii)] the image of the $0$-framed section  (Definition~\ref{def:longF}) of the circle bundle for $F$ in $E_F$, has  trivial image in $BG$. 
\end{enumerate}

\[
\begin{tikzcd}
\pi_1(E_K) \arrow[r, "\rho_{\ast}", two heads] \ar[d, "inc_*", labels=left] & G\\
\pi_1(E_F) \arrow[ru, "\overline{\rho}_{\ast}", two heads, labels=below right] &
\end{tikzcd}
\]
\end{definition}

In this section we prove the central theorem in this paper. Recall that, given a surface $F\subseteq  B^4$ with $\partial(F)= F \cap S^3 = K$ and $F \pitchfork \partial{S^3}$, and a group homomorphism $\rho: \pi_1(E_K)\to G$, we say that $\rho$ {\it extends over $F$} if there exists  a group homomorphism  $\bar{\rho}: \pi_1(E_F)\to G$ which fits into the commutative diagram

\begin{equation}
    \label{eq:triangle}
\begin{tikzcd}
\pi_1(E_K) \arrow[r, "\rho"] \ar[d, "i_*", labels=left] & G.\\
\pi_1(E_F) \arrow[ru, dashed, "\overline{\rho}", labels=below right] &
\end{tikzcd}
\end{equation}

In Theorem~\ref{theorem:G-surface} we prove that, for a $G$-knot $K$ (Definition~\ref{def:G-Knot}), a $G$-surface exists if and only if $\Theta_\rho(K)$ vanishes. In Corollary~\ref{thm:main}, we show that when $G$ is taut, the $G$-knot and $G$-surface conditions are automatically met, and the vanishing of $\Theta_\rho(K)$ is equivalent simply to the existence of a surface over which $\rho$ extends.

\begin{theorem}[The $G$-Surface Theorem]~\label{theorem:G-surface} A $G$-knot $K$ bounds a smooth $G$-surface in $B^4$ if and only if the homotopy class $\Theta_{\rho}(K)$ vanishes.
\end{theorem}

\begin{proof}[Proof of Theorem~\ref{theorem:G-surface}]
 $(\implies)$ Let $K$, equipped with a quotient $\rho: \pi_1(E_K)\to G$, be a $G$-knot bounding a $G$-surface $F$. By hypothesis, $F \subseteq B^4$ is a locally flat, oriented, connected, and properly embedded surface in $B^4$ such that $\partial F = K$.  By Lemma~\ref{Lemma:2_disk_bundle}, $F$ has a trivial regular neighborhood, whose exterior we denote, as usual, by $E_F \subseteq  B^4$. By hypothesis,  $\rho$ admits an extension to an epimorphism $\bar{\rho} \colon \pi_1(E_F) \to G$.  This extension determines the commutative diagram~\eqref{eq:triangle} of group homomorphisms.   (To simplify notation, we use $\rho$ and $\overline{\rho}$ to denote both the group homomorphisms in question and the maps inducing them.) 
 
Denote the 0-framed longitude of $K$ by $K_0$, so  $K_0\subseteq\partial E_K$ and $\partial E_K=K_0\times\mu$ where $\mu$ is the preferred meridian of $K$. Since  $\rho$ restricts on $\pi_1(\partial E_K)$ to projection to $\mu$, we may assume, by the homotopy extension theorem, and using that $BG \sim K(G, 1)$, that we have a commutative diagram of spaces 
\[
\begin{tikzcd}
E_K \arrow[r, "\rho"] \ar[d, "inc_*", labels=left] & BG\\
E_F \arrow[ru, dashed, "\overline{\rho}", labels=below right] &
\end{tikzcd}
\]
inducing the given diagram of homomorphisms on fundamental groups, and such that $\rho$ sends $K_0 \times \mu = \partial{E_K}$ to $\mu\cong S^1$ by projection. That is, $\rho$ maps the longitude $K_0$ to a point in BG.  By Lemma~\ref{lemma:0-framed-Section}, we may assume the sphere bundle of $F$ is parameterized so that the restriction of 
$\overline{\rho}$ to $F \times S^1$ is given by projection onto the $S^1$ factor. Therefore, the map 
$\overline{\rho}$ extends over the regular neighborhood of $F$ by projecting $F \times B^2$ to 
$B^2\subseteq  KG,$ where $B^2$ denotes the distinguished 2-cell whose interior equals $KG\backslash BG$. Thus $\theta \colon S^3 \to KG$ extends over $B^4$, 
implying that $\Theta_{\rho}(K) = 0$ as claimed.
\vspace{3mm}

 $(\impliedby)$
 This argument uses  transversality to construct a (possibly disconnected) surface, which we then modify, attaching handles, to produce a connected one.

By hypothesis, $\Theta_{\rho}(K) = 0 \in \pi_3(KG)$.  Thus,  a map  $\theta\colon S^3 \to KG$  representing $\Theta_{\rho}(K)$ extends to a map  $\bar{\theta} \colon B^4 \to KG$ which, after a homotopy rel boundary, we may assume is transverse to
\[
0 \in B^2 \subseteq BG \cup_{S^1} B^2 \defeq KG.
\]
Thus, ${\bar{\theta}}^{-1}(0)$ is a (possibly disconnected) oriented and properly embedded surface $F \subseteq B^4$. By construction, $\theta^{-1}(\{0\})=K$ and $\bar{\theta}|_{S^3}=\theta$. Hence 
\[\partial{F} = (\bar{\theta}|_{S^3})^{-1}(\{0\}) = {\theta}^{-1}(\{0\}) =K.
\]
Note also that $\bar{\theta}|_{B^4\backslash F}$ is a map from a surface complement to $KG\backslash\{0\}\simeq BG$ whose restriction to $S^3\backslash K$ induces $\rho$. In other words, if $F$ were connected, then the map $(\bar{\theta}|_{B^4\backslash F})_{\ast}:\pi_1(E_F)\to G$ would be the desired extension $\bar{\rho}$. $F$ satisfies all the claims in the stated theorem, except that $F$ may have additional closed components.

Assume that $\bar{\theta}^{-1}(0)$ is disconnected.  We will change the map $\bar{\theta} \colon B^4 \to KG$ by a homotopy rel boundary to a new map $\bar{\theta'}$ such that $\bar{\theta'}^{-1}(0)$ has one fewer components. Since we are changing $\bar{\theta}$ by a homotopy rel boundary, the new surface $\bar{\theta}'^{-1}(0)$ still has boundary $K$ and satisfies the conclusions of the theorem. Since $\bar{\theta}^{-1}(0)$ is compact, it suffices to show that it's possible to do this once.  

We reduce the number of components in two steps. We will first attach a 5-dimensional 1-handle to  $B^4\times\{1\}\subseteq  B^4\times [0,1]$. We use this handle to connect two components of $F$ and then extend the map $\bar{\theta}$ over the new surface exterior. We will then attach a 2-handle to cancel the 1-handle, once again extending our map. The result will be a homotopy rel $B^4\times\{0\}$ from $\bar{\theta}$ to a map $\bar{\theta}': B^4\to KG$ with the property that $\bar{\theta}'^{-1}(0)$ has one fewer components than $\bar{\theta}^{-1}(0)$.

To achieve the above, we begin with the product homotopy, 
\[
\bar{\theta} \times id \colon B^4 \times [ 0, 1 ] \to KG.
\]
Choose two points, denoted $p_{\pm}$, that lie in distinct components,  denoted $F_{\pm} \times \{1\}$, of
\[
(\theta \times\{1\})^{-1}(\{0\}) \subseteq  B^4 \times \{1\}.
\]
As before,  $0$ denotes the center of the distinguished 2-cell $B^2 \subseteq KG$.

The points $p_{\pm}$ lie in small $2$-disks, $d^2_{\pm}$, one in each component $F_{\pm} \times \{1\} \subseteq B^4 \times \{1\}$ of our surface $F$. Let $\nu_{\pm}$ be the $4$-disks, $d^2_{\pm} \times B^2$, which are the restriction of the tubular neighborhood of $F$ to each of the two $2$-disks $d^2_{\pm}$:
\[
\nu^4_{\pm}:= (d^2_{\pm} \times B^2) \subseteq (B^4 \times \{1\}).
\]
Since $F = \bar{\theta}^{-1}(\{0\})$, it follows that for each point $x\in d_\pm \subseteq B^4 \times 1$, the restriction $\bar{\theta}_{|{x\times B^2}}$ is projection to the 2-cell $B^2\subseteq BG \cup_{S^1} B^2$.

We now attach a 5-dimensional $1$-handle,
\[
H \cong (B^2 \times B^2) \times [ -1, 1 ]
\]
to $\nu_{\pm} \subseteq B^4 \times \{1\}$ by identifying $(B^2\times B^2)\times\{\pm1\}$ with $\nu_\pm$ to obtain a new oriented $5$-manifold 
\[
W: = (B^4 \times [0,1]) \cup_{\nu_{\pm}} H.
\]
By construction, the restriction of $\bar{\theta}$ to the image of each foot of the 1-handle projects the first $B^2$ factor to $0$ and identifies the second $B^2$ factor with $B^2\subseteq KG$. Thus, we may extend the map 
\[
\bar{\theta} \times id \colon B^4 \times [ 0, 1 ] \to KG.
\]
over our newly attached $1$-handle by projecting, for each $t \in [-1,1]$, the 4-disk 
\[
B^2\times B^2\times \{t\} \subseteq H
\]
to the second $B^2$ factor and then to the $2$-disk  
\[
B^2 \subseteq BG \cup_{S^1} B^2 = KG.
\]
Since $\bar{\theta}$ maps each component of the attaching region, $d_\pm\times B^2\times \pm 1$, of $H$ by projecting onto the second $B^2$ factor and then identifying that with the corresponding 2-cell in $KG$, the above extension is well-defined and continuous.

Let $(x,y)\in d_+\times \partial(B^2)$ be a point such that $\bar{\theta}(x,y)=b,$ the basepoint of $\pi_1(KG)$. Let $\gamma:=(x, y)\times [-1,1]$, that is, $\gamma$ is a saddle of the newly attached $1$-handle, and regard $\gamma$ as a path from $(x, y)\times \{1\}$ to $(x, y)\times \{-1\}$. The image of $\gamma$ under the extension of $\bar{\theta}$ defined above is a point $b$ in the boundary of the distinguished 2-cell $B^2\subseteq KG$. Now let $\alpha$ be any path connecting $(x,y)\times\{-1\}$ to $(x,y)\times\{1\}$ and whose interior is contained in the complement of the attaching region of the 1-handle $H$. 
Since $\pi_1(E_F) \twoheadrightarrow G$ is onto, we can find a second path, $\beta$, this time from $(x,y)\times\{1\}$ to itself, 
such that the image $\bar{\theta}_\ast(\alpha\ast\beta)$ of the concatenation of paths is null-homotopic in $BG$. It follows that $\bar{\theta}_\ast(\gamma\ast\alpha\ast\beta)$ is trivial as well, since the image of $\gamma$ is a constant loop at $b$. After a small isotopy of $\gamma\ast\alpha\ast\beta$ and a homotopy of $\bar{\theta}$, we may assume that the path $\gamma\ast\alpha\ast\beta$ is embedded and has a tubular neighborhood mapped to a point. 

Attaching a $2$-handle to this saddle loop cancels the $1$-handle, creating a $5$-ball, $B^4\times [0, 1]$. Extending our map via a constant map over this $2$-handle provides the desired homotopy  from our initial map $\bar{\theta}$, on $B^4\times\{0\}$, to a new map $\bar{\theta'}: B^4\to KG$, whose domain is the newly constructed $B^4 \times \{1\}$, with the property that $\bar{\theta}'^{-1}(0)$ has one fewer components than $\bar{\theta}^{-1}(0)$. This proves the last claim in our theorem.  
\end{proof}

\subsection{Taut groups and consequences}\label{sec:taut-groups}

\begin{definition}\label{tautgroup}
A based group $(G, \iota(1))$ is \emph{taut} if whenever an element $g \in [G, G]$ satisfies $\iota(1)\, g\, \iota(1)^{-1} = g$, then $g = e$.
\end{definition}
 
\begin{example}\label{ex:dihedral-taut}
An easy direct calculation shows that for $n$ odd the dihedral group $D_n$, based by the reflection $u$, and the $n$-bottle group $\mathbb{D}_n$ (Definition~\ref{def:n-bottle}), based by the generator $u$, are taut. This also follows from the upcoming Lemma~\ref{lem:meta-is-taut}.
\end{example}

Note that for $K$ a non-trivial knot, the group $\pi_1(E_K)$ is {\em not} taut. Indeed, the 0-framed longitude of $K$ is a non-trivial element in the commutator subgroup of $\pi_1(E_K)$ and commutes with the basing.

The following Lemma shows that taut quotients of knot groups abound. Indeed, every knot with non-trivial determinant admits a dihedral quotient and then, by Corollary~\ref{cor:char-knots-induce-both}, also a based metabelian quotient.

\begin{lemma}\label{lem:meta-is-taut}
Let $K\subseteq S^3$ be a knot and $\pi: = \pi_1(E_K)$. Any based metabelian  quotient of $\pi$ is taut.
\end{lemma}

\begin{proof}
Suppose we are given an epimorphism $\rho: \pi\twoheadrightarrow G$ where $G$ is a metabelian group, $\pi$ is based by a meridian $\mu,$ and its image, $\rho(\mu)=:m$, is a basing for $G$. 

Let $\pi':=[\pi, \pi]$ and $G':=[G, G].$ We have $\pi\cong \pi'\rtimes \mathbb{Z}$ and  $G\cong G'\rtimes \mathbb{Z}$. Since $\rho$ is a based homomorphism, it is a map $\pi'\rtimes \mathbb{Z} \to G'\rtimes \mathbb{Z}$, that is, it respects the semi-direct product structure and the action on $\pi'$ by $\mu$ descends to an action on $G'$ by $m$. 

Let $\widetilde{E_K}$ and $\widetilde{BG}$ denote the $\mathbb{Z}$ covers induced by the abelianization homomorphisms of $\pi$ and $G,$ respectively. Denoting by $\rho$ a map $E_K\to BG$ inducing the given quotient $\rho:\pi\to G$, we obtain a commutative diagram:
\[
\begin{tikzcd}
\widetilde{E_K} \arrow[r, dashed, "\tilde{\rho}"] \ar[d, "p" left ] & \widetilde{BG} \ar[d]\\
E_K \arrow[r, "\rho"] & BG
\end{tikzcd}
\]
where the map $\tilde{\rho}$ is the lift of the composition $\widetilde{E_K}\to BG$. (This exists since the fundamental group of each covering map has image the respective commutator subgroup. Moreover, surjectivity of $\rho: \pi\twoheadrightarrow G$ implies that $\tilde{\rho}$ induces a surjection on first homology.) 

Note that 
\[H_1(E_K; \mathbb{Z}[\mu, \mu^{-1}])\cong H_1(\widetilde{E_K};\mathbb{Z})\cong H_1(\pi'; \mathbb{Z})\cong \pi'/[\pi',\pi'], 
\]
so $H_1(\widetilde{E_K};\mathbb{Z})$, the Alexander module of $K$, is a $\mathbb{Z}[\mu, \mu^{-1}]$ module on which $\mu$ acts by translation. This action has the property, surely well known, that if an element in $\pi'$ is fixed under conjugation by $\mu$, its lifts to $H_1(\widetilde{E_K};\mathbb{Z})$ are in the kernel of $\mu-1$. We recall this below.  

Let $t\in E_K$ be the basepoint and denote its lifts to $\widetilde{E_K}$ by $t_i, i\in \mathbb{Z}$, of course labeled so that $\mu\cdot t_i=t_{i+1}.$  Given any loop based at $t$ which represents an element $x\in \pi'$, the lift $\tilde{x}_i$ of $x$ based at $t_i$ is also a loop and thus represents a homology class. By definition, $\mu\cdot\tilde{x}_i=\tilde{x}_{i+1}$. Now suppose $x\in \pi'$ is fixed under conjugation by $\mu,$ that is, $\mu^{-1} x\mu=x$. The lifts of $x$ and $\mu^{-1} x\mu$ based at $t_i$ are thus homotopic and therefore homologous. But the latter can be written as $\tilde{\mu}^{-1}_{i}\tilde{x}_{i+1}\tilde{\mu}_i$, where concatenation of paths is to be read left to right and $\tilde{\mu}_i$ of course denotes the lift of $\mu$ based at $t_i$. We derive the following equality in $H_1(\widetilde{E_K};\mathbb{Z})$:
\[
[\tilde{x}_i]=[\tilde{\mu}^{-1}_{i}\tilde{x}_{i+1}\tilde{\mu}_i]=[\tilde{x}_{i+1}]=\mu\cdot[\tilde{x}_i],
\]
from which it follows that $[\tilde{x}_i]$ is in the kernel of $\mu-1$, as we set out to show.  

The analogous statements hold for the covering map $\widetilde{BG}\to BG$. That is, the conjugation action of $m$ on $G'$ lifts to an action by $m$ on $H_1(\widetilde{BG},\mathbb{Z}),$ a $\mathbb{Z}[m, m^{-1}]$ module;
and,  if $y\in G'$ is fixed by conjugation by $m$, then $m$ fixes the homology class represented by $\tilde{y}$, where $\tilde{y}$ is any lift of $y$.

Now suppose $y\in G'$ has the property that $mym^{-1}=y$. Our aim is to show that $G$ is taut, in other words, that $y=e$.  Let $\tilde{y}\in \pi_1(\widetilde{BG})\cong G'$ be a lift of $y$. By the above, we have that $m\tilde{y}=\tilde{y}$ or $(m-1)\tilde{y}=0$. 

If $K \subseteq S^3$ is any knot and $A := H_1(E_K; \Z[\mu, \mu^{-1}])$ denotes its Alexander module, multiplication by $\mu-1$ is an isomorphism on $A$~\cite[Proposition~{1.2}]{MR461518}. Since $A$ is a finitely generated module over $\Z[\mu, \mu^{-1}]$, it follows from~\cite[Corollary~4.4 (a)]{eisenbud1995}, that multiplication by $(\mu-1)$ descends to an isomorphism on any quotient module of $A$. In other words, $(\mu-1)\cdot\tilde{y}=0$ implies $\tilde{y}=0$. Thus, $y=e$ as claimed.
\end{proof}

The hypotheses of the $G$-Surface Theorem require $K$ to be a $G$-knot and the surface $F$ to be a $G$-surface (Definitions~\ref{def:G-Knot} and~\ref{def:G-surface}). Both conditions ask that certain longitudes map trivially under $\rho$ and $\bar{\rho}$. When $G$ is taut, these conditions are satisfied automatically, as shown by the following lemma and proposition.

\begin{lemma}\label{lem:taut-G-knot}
Let $G$ be a taut group and $\rho \colon \pi_1(E_K) \twoheadrightarrow G$ a based epimorphism. Then the $0$-framed longitude of $K$ maps trivially under $\rho$. In particular, $K$ is a $G$-knot.
\end{lemma}

\begin{proof}
The $0$-framed longitude of $K$ lies in $[\pi_1(E_K), \pi_1(E_K)]$, since it bounds a Seifert surface. Its image under $\rho$ therefore lies in $[G, G]$. Moreover, the longitude commutes with any meridian, so its image commutes with the basing $m = \rho(\mu)$. Since $G$ is taut, any element of $[G, G]$ fixed by conjugation by $m$ is trivial.
\end{proof}

\begin{proposition}\label{prop:dihedral-surface-section}
Let $(F, K) \subseteq (B^4, S^3)$ be a properly embedded, connected, oriented surface and let $\bar{\rho} \colon \pi_1(E_F) \twoheadrightarrow G$ be a based surjection to a taut group $G$. Then $F$ is a $G$-surface; that is, $F$ admits a section of its normal circle bundle whose image under $\bar{\rho}$ is trivial.
\end{proposition}

\begin{proof}
By Lemma~\ref{lemma:0-framed-Section}, there exists a $0$-framed section $\sigma \colon F \to S(F) \subseteq \partial E_F$, meaning the composition $F \xrightarrow{\sigma} S(F) \xrightarrow{\mathrm{inc}} E_F \xrightarrow{p} S^1$ is constant, where $p$ induces the abelianization $\pi_1(E_F) \to \Z$. We show that $(\bar{\rho} \circ \sigma)_*$ is trivial on $\pi_1(F)$.

Let $\{a_i\}$ be generators for $\pi_1(F)$. Since $F$ is orientable, $S(F) \cong F \times S^1$, and $\sigma_*(a_i)$ commutes with the meridian $\mu$ in $\pi_1(F \times S^1) = \pi_1(F) \times \Z$. This commutativity is preserved under $\mathrm{inc}_* \colon \pi_1(S(F)) \to \pi_1(E_F)$. Hence $\bar{\rho}_*(\sigma_*(a_i))$ commutes with the basing $m = \bar{\rho}_*(\mu)$ in $G$.

Since $\sigma$ is $0$-framed, $\sigma_*(a_i)$ lies in the kernel of the abelianization $\pi_1(E_F) \to \Z$, so $\bar{\rho}_*(\sigma_*(a_i)) \in [G, G]$. We thus have an element of $[G,G]$ that is fixed by conjugation by $m$. Since $G$ is taut, this element is trivial.

Since $BG$ is an Eilenberg--MacLane space and $F$ is $2$-dimensional, $\bar{\rho} \circ \sigma$ is nullhomotopic, so $F$ is a $G$-surface.
\end{proof}

\begin{remark}\label{rem:nonor-G-surface}
The non-orientable analogue of this result does not follow from tautness alone. Lemma~\ref{lemma:nonor-section} establishes the corresponding statement for $G = D_n$ ($n$ odd), but uses the specific structure of the dihedral group (the centralizer of a reflection is $\{1, s\}$ and $s^2 = 1$); see also Remark~\ref{rem:nonor-taut}.
\end{remark}

\begin{corollary}[Main Extension Result]\label{thm:main}
Suppose $\rho\colon \pi_1(E_K) \to G$ is a based epimorphism to a {\em taut} group $G$.  Then, there is an oriented smooth, connected surface $F \subseteq B^4$ with $\partial F = F \cap S^3 = K$ for which there exists an extension of $\rho$ to an epimorphism $\bar{\rho} \colon \pi_1(E_F) \to G$ if and only if  
\[
0 = \Theta_{\rho}(K) \in \pi_3(KG).
\]
\end{corollary}

\begin{proof}
By Lemma~\ref{lem:taut-G-knot}, $K$ is a $G$-knot. If $\Theta_\rho(K) = 0$, the $G$-Surface Theorem~\ref{theorem:G-surface} produces a $G$-surface $F$; in particular, $\rho$ extends over $E_F$. Conversely, if $\rho$ extends over some oriented surface $F$, Proposition~\ref{prop:dihedral-surface-section} implies $F$ is a $G$-surface, and the $G$-Surface Theorem gives $\Theta_\rho(K) = 0$.
\end{proof}

\begin{example}
Let $G=\mathbb{Z}$ and let $\rho: \pi_1(E_K)\to \Z$ be the abelianization. The space $KG$, that is $K\Z$ in this case, has the homotopy type of a point, consistent with the fact that for any knot $K \subseteq  S^3$ the abelianization homomorphism extends, for instance over a pushed-in Seifert surface. 
\end{example}

Recall that, by Lemma~\ref{lem:meta-is-taut}, all based metabelian quotients of knot groups are taut. Therefore, for any quotient $\rho$ of a knot group onto a metabelian group, our result gives a sharp obstruction to the existence of an extension of $\rho$ over the exterior of a smooth, connected oriented surface $F \subseteq B^4$. In Theorem~\ref{thm:equivalence} we show that the vanishing of $\Theta_\rho(K)$ also determines precisely when a dihedral quotient of $\pi_1(E_K)$ extends over a surface in $B^4$. It follows, for instance, that a 3-colorable twist knot $K_n$ bounds such a surface if and only if $n\equiv 2\mod{9},$ see Example~\ref{ex:twist}.

As noted in Remark~\ref{rem:category}, the word ``smooth'' in the statements of the $G$-Surface Theorem~\ref{theorem:G-surface} and Corollary~\ref{thm:main} can be replaced by ``locally flat''.

\begin{corollary} \label{cor:categories}
    A $G$-knot $K$ bounds a locally flat $G$-Surface if and only if $K$ bounds a smooth $G$-Surface.
\end{corollary}

\begin{proof} Assume a $G$-knot $K$ bounds a topologically locally flat $G$-surface $F\subseteq B^4$. Since locally flat surfaces admit tubular neighborhoods, by the same argument as in proof of Theorem~\ref{theorem:G-surface}, the class $\Theta_\rho(K)$ vanishes. Hence, again by Theorem~\ref{theorem:G-surface}, a smooth $G$-surface exists as well.
\end{proof}

\begin{remark}\label{rem:gompf}
    Bob Gompf~\cite{gompf-email} pointed out an alternative proof of the above corollary, relying on~\cite{freedman1982topology}: a locally flat surface can be topologically isotoped to have only isolated singularities, then locally smoothed in small neighborhoods of these singularities, while preserving the existence of $\bar{\rho}$.
\end{remark}

If a $G$-knot bounds a $G$-surface in $B^4$, the following corollary shows that it in fact bounds a $G$-disk in some oriented $4$-manifold.
\begin{corollary}\label{cor:surgery}
  Let $K$ be a $G$-knot bounding a $G$-surface $F \subseteq B^4$. Then there exists a compact, oriented $4$-manifold $W$ with $\partial W=S^3$ such that $K$ bounds a smooth $G$-disk in $W$.
\end{corollary}

\begin{proof}
Since $F$ is a $G$-surface, the $0$-framed section $s\colon F\to S(F) \subseteq \partial E_F$ satisfies $\bar{\rho}\ \circ\ \mathrm{inc \circ \, s}_*=0$ on $\pi_1(F)$. Choose a half-symplectic basis $\{a_1, \ldots, a_g\}$ for $H_1(F)$; these are represented by simple closed curves in $F^{\circ}$ whose images in $\pi_1(E_F)$ (via the section) are mapped trivially by $\bar{\rho}$.
Perform ambient surgery on $F$ along each $a_i$: for each $a_i$, remove a tubular neighborhood $S^1\times D^3$ of $a_i$ in $B^4$ and glue in $D^2\times S^2$. This replaces $B^4$ by a new manifold and reduces the genus of $F$ by one. Since each $a_i$ maps trivially under $\bar{\rho}$, the extension of $\rho$ over the surface exterior is preserved at each step. After $g$ surgeries, $F$ has become a disk in $W$.
\end{proof}

\begin{remark}
For $G = D_n$ or $G = \mathbb{D}_n$, the converse also holds: if $K$ is a $G$-knot which bounds a $G$-disk (or indeed any $G$-surface, orientable or not) in some oriented $W$ with $\partial W = S^3$, then $K$ bounds a $G$-surface in $B^4$. This is shown in Section~\ref{sec:dihedral-criteria}.
\end{remark}

We end this section by proving that the obstruction space $KG$ is no bigger than it needs to be.

\begin{theorem} \label{thm:subgroup}
Suppose $G$ is a based non-abelian group which abelianizes to $\Z$ and $x \in \pi_3(KG)$ is any non-trivial element. Then there exists a based inclusion of based groups $H \xhookrightarrow{\iota} G$, where $H$ is non-abelian, a knot $K \subseteq S^3$, and a based epimorphism $\rho \colon \pi_1(E_K) \to H$ such that $\Theta_\rho(K) \not = 0 \in \pi_3(KH)$ and $\iota_{\ast}(\Theta_{\rho}(K)) = x \in \pi_3(KG)$.
\end{theorem}

\begin{proof} We note that in what follows, $\rho$ denoted both a map $\pi_1(E_K)\to KH$ and its composition with the given inclusion $\iota: H\to G$. 

The proof is similar to that of the $G$-surface Theorem.  Given $x$ as above, choose a map $\varphi \colon S^3 \to KG$ representing the given element $x \in \pi_3(KG)$ and transverse to 
\[
\{0\} \in B^2 \subseteq KG \defeq BG \cup_S B^2.
\]  
We can assume, after a homotopy, that $\varphi$ maps onto $B^2 \subseteq KG$.  By an argument analogous to the one used in the proof of Theorem~\ref{theorem:G-surface}, we can connect components of $\varphi^{-1}(\{0\})$ without changing the relevant homotopy class, so that, $\varphi^{-1}(\{0\})$ is a knot $K \subseteq S^3$.

The restriction of $\varphi$ to $E_K$ induces a homomorphism $\rho\colon \pi_1(E_K) \to G$. Let $H \subseteq G$ denote its image.  Furthermore, since $K= \varphi^{-1}(\{0\})$, we may parametrize the regular neighborhood of  $K$ in $S^3$ so that $\varphi|_{K\times B^2}$ maps  to $B^2 \subseteq KG$ by projecting each fiber to $B^2 \subseteq KG$. Therefore, the subgroup $H$ contains a basing given by the image of a meridian of $K$.

Observe that $H$ cannot be abelian.  Indeed, since $H$ is based, it admits  $\Z$ as a quotient, so, in particular, it is infinite. Therefore, since $H$ is a quotient of a knot group,  if $H$ is abelian  then $H$ is an infinite quotient of $\pi_1(E_K)^{ab}$, so in fact $H\cong \Z$. But this would imply $\pi_3(KH) \cong 0$, since 
\[
K\Z\simeq S^1\cup_S B^2 \simeq \{*\}, 
\]  
contradicting the assumption that  $x \neq 0$. 

By construction, the group $H$ has all the properties stated in the theorem. In particular, we see that $\iota_{\ast}(\Theta_{\rho}(K)) = x \in \pi_3(KG)$ by regarding $KH$ as a subcomplex of $KG$.
\end{proof}

We end this section with a remark on potential further applications for our 
$G$-Surface Theorem, independent of our goals in this paper. 
Homology cobordism plays a significant role in the classification 
problem in low-dimensional manifold theory. The $G$-Surface 
obstruction may suggest new paths in homology cobordism theory and 
knot concordance. Recall that $G$ can be any quotient 
of a knot group, as long as the 0-framed longitude of $K$ is mapped trivially. When our 
obstruction vanishes, transversality and handlebody theory allow us 
to construct $G$-surfaces for knots in $B^4$. Moreover, 
intersection theory and duality provide computable criteria for the 
vanishing of the obstruction~(see, for instance, 
Theorem~\ref{thm:equivalence}). For solvable groups $G$, 
$L^2$-signature theory may prove useful in extracting  new concordance 
invariants from $G$-surfaces.

\newpage

 \part*{ Part II. The case of dihedral groups}

We now specialize the general obstruction theory to the case of dihedral quotients of knot groups. In this context, we translate the invariant $\Theta_\rho(K)$ into classical data about the double branched cover and, consequently, about the Seifert form of the knot. Furthermore, we present an explicit construction of a dihedral surface when one exists. We conclude the paper with an array of examples and computations.

\section{Dihedral machinery}\label{Sec:pairings}

In this section, we specialize to the case where the target group is the dihedral group $D_n$ or its larger cousin the $n$-bottle group $\mathbb{D}_{n}$, defined shortly (Definition~\ref{def:n-bottle}). We first collect the algebraic properties of $D_n$ and $\mathbb{D}_n$ (\ref{section:linking}), then recall the Cappell--Shaneson theory of characteristic knots (\ref{sec:dih-knots}--\ref{sec:beta}), which provides the bridge between dihedral quotients and the Seifert form. We then embed our invariant in $\pi_3(K\mathbb{D}_{(2)})$ and establish the key bijection with $H_1(\Sigma_2 K)$ (\ref{section:theta}--\ref{twocovers}). The $\sfrac{\Z_{(2)}}{\Z}$-linking pairing and its Seifert-matrix formulas are reviewed in \ref{linkingpairing}--\ref{torlinking}.

Much of this section should be familiar to experts, and we provide it to keep the paper self-contained. We also note that, throughout, colimits are indexed over the directed set of {positive, odd} integers. We will always use $n$ to denote an odd integer.

\subsection{Properties of dihedral groups.}\label{section:linking}

Consider the following two families of groups, indexed by positive {odd integers}, $n$.  These are based groups with a basing denoted $u$.

\[
\mathbb{D}_n = \langle t, u \mid t^n, \ utu^{-1}=t^{-1}\rangle \cong \Z_n \rtimes \Z
\text{\hspace{.25in} and \hspace{.25in}} D_n=\mathbb{D}_n/\langle u^2\rangle \cong \Z_n \rtimes \Z_2
\]

\begin{definition}\label{def:n-bottle}
The group $\mathbb{D}_{n}$ defined above will be called the \emph{$n$-bottle group}. 
\end{definition}
We note that the \emph{Klein bottle group} is $B \coloneqq \langle t,u \mid utu^{-1}=t^{-1}\rangle,$ so, as the name captures, each $n$-bottle group is a quotient of $B$: \(\mathbb{D}_{n} \cong B/\langle t^n\rangle\).

Recall that the groups $D_n$ and $\mathbb{D}_{n}$ are taut, so we may apply Corollary~\ref{thm:main} to epimorphisms from a knot group to these groups.

The first family of groups abelianizes to $\Z$, generated by $u$, which clearly splits.  The second family consists of dihedral groups of order $2n$, $n$ odd. Each group in this family splits its abelianization,  $\Z_2$, again generated by $u$.  

As before, we let $BG$ represent the classifying space of a discrete group $G$. Then $BG \sim K(G, 1)$, an Eilenberg MacLane space and in each case we build the space $KG$ like we did in Section~\ref{Gknots}, by attaching a 2-cell along a 1-cell representing $u$.

In what follows, we denote the generator of the commutator subgroup $\Z_n \subseteq D_n$ by the number $1$, and the generator of the commutator subgroup of $\mathbb{D}_n$, that is, $\Z_n \subseteq \mathbb{D}_n$, by $t \in \mathbb{D}_n$.

\begin{lemma}\label{H3(G)}
Consider the homomorphism $\iota \colon \Z_n \to \mathbb{D}_n$ given by $\iota(1) = t$, with $n$ odd.  Then
\begin{enumerate}
\item[(i)] $
\Z_n \cong H_3(\Z_n) \xrightarrow{\iota_{\ast}}  H_3(\mathbb{D}_n)$
is an isomorphism, and
\item[(ii)] $H_2(\mathbb{D}_n) \cong 0$.
\end{enumerate}
\end{lemma}

\begin{proof}
Consider the  fibration 
\[
1 \to B\Z_n \to B\mathbb{D}_n \xrightarrow{\pi} B\Z \to 1.
\]
Here $B\Z \sim S^1$.  The Serre Spectral sequence determines a long exact sequence, often called the Wang Sequence in this context~\cite{MR28580}, as follows:
\[
\xrightarrow{} H_3(\Z_n) \xrightarrow{(c_u - id)= 0} H_3(\Z_n) \xrightarrow{\text{inc}_{\ast}} H_3(\mathbb{D}_n) \xrightarrow{\partial}  \cancelto{0}{H_2(\Z_n)}\to
\]
\[ 
\xrightarrow{c_u - id} \cancelto{0}{H_2(\Z_n)} \xrightarrow{\text{\hspace{3mm}inc}_{\ast}} H_2(\mathbb{D}_n)\vspace{3mm}\xrightarrow{\partial}
 H_1(\Z_n) \xrightarrow[c_u - id]{\times{(-2)}} H_1(\Z_n) \to 
\]
Here, $c_u$ is the automorphism induced by the conjugation action of $u$ on $\Z_n$ in the semi-direct product $\mathbb{D}_n \cong \Z_n \rtimes \Z \cong \langle u, t \mid t^n, \ utu^{-1}=t^{-1}\rangle$, $n$ odd.
We have $H_2(\Z_n) \cong 0$, and $H_3(\Z_n) \cong \Z_n$. Conjugation by $u$, denoted $c_u$, induces the trivial action on $H_3(\Z_n)$ by an elementary modification of an analogous argument for the dihedral group. (See, for instance,~\cite[Example 6.7.10]{MR1269324}.)
 Thus, the first homomorphism in the first row is zero, that is, $c_u - id = 0$, as displayed above, implying statement~$(i)$.  Since $H_1(\Z_n) \cong \Z_n$ and $n$ is odd, multiplication by $-2$ in the second row is an isomorphism. The Lemma follows.
\end{proof}

\subsection{Dihedral knots}\label{sec:dih-knots}

We use the following notation for dihedral groups
\begin{equation}  \label{eq:Gn-def}
D_n = \langle t, u \mid t^n, u^2, \ utu^{-1}=t^{-1}\rangle \cong \Z_n \rtimes \Z_2.
\end{equation}

Recall that a $G$-knot (Definition~\ref{def:G-Knot}) is a knot $K$ equipped with a quotient homomorphism $\rho \colon \pi_1(E_K) \to G$ with the property that the $0$-framed longitude of $K$ is mapped trivially under $\rho$. That is, we say $(K, \rho)$ is a $G$-knot, where $\rho\colon \pi_1(E_K) \twoheadrightarrow G$. 

Note that for any knot $K$, the longitude of $K$ bounds a Seifert Surface, and thus lies in the second derived subgroup of $\pi_1(E_K)$.  Thus, any knot $K$ which admits an epimorphism $\rho\colon \pi_1(E_K) \to D_n$ is a $D_n$ knot, and similarly for $\mathbb{D}_n$ knots. (This is a special case of Lemma~\ref{lem:taut-G-knot}, since $D_n$ and $\mathbb{D}_n$ are taut.)

Dihedral quotients of knot groups go back to Reidemeister and were studied extensively by Fox, who introduced Fox colorings. There is a combinatorial way of describing such quotients (and checking whether they exist for a given knot). It is well known that  homomorphisms $\rho: \pi_1(E_K) \to D_n$ correspond to characters $\chi: H_1(\Sigma_2 K; \Z)\to \Z/n\Z$. Therefore, for $n$ square-free, a knot $K$ is a $D_n$-knot if and only if $n$ divides the determinant of $K,$ since $|\Delta_K(-1)|=|H_1(\Sigma_2 K; \Z)|$. 

The existence of a dihedral quotient can also be determined using the Seifert form of the knot. This point of view turns out to be particularly useful for us, and we recall it next.

\subsection{Cappell-Shaneson characteristic knots}\label{sec:beta}

The following definition is used for detecting the existence of dihedral quotients for a knot $K$; evaluating our obstruction $\Theta_\rho(K)$; computing the $\Xi_n$ invariant of dihedral covers (see Section~\ref{sec:xi}); and building dihedral surfaces (see Section~\ref{sec:n-bottle}). Throughout, $n$ is odd.

\begin{definition} (\cite{CS1984linking}\label{def:chark})
Let $K \subseteq S^3$ be a knot. A knot $\beta \subseteq S^3$ is a {\em mod-$n$ characteristic knot for $K$} if there is a Seifert surface $S$ for $K$, containing $\beta$ in its interior, such that $\beta$ represents a nonzero (primitive) element in $H_1(S)$ and $(V+V^T)\cdot [\beta] \equiv 0 \mod n,$
where $V$ denotes a Seifert matrix associated to $S$.
\end{definition}

As explained in the proof of~\cite[Proposition~1.1]{CS1984linking}, a mod~$n$ characteristic knot $\beta$ for $K$ determines a dihedral quotient of the knot group,
\begin{equation}
   \rho_\beta: \pi_1(E_K)\twoheadrightarrow \langle u, t \mid u^2, \ t^n, \ utu^{-1}=t^{-1}\rangle ,
\end{equation}
as in Equation~\ref{eq:char-knot-quotient}. Recall that we have an isomorphism 
\[
\pi_1(E_K)\cong (\pi_1(E_S)\ast  \langle m\rangle)/\langle\langle m w^+ m^{-1}=w^-, w\in\pi_1(S) \rangle \rangle,
\]
where $m$ denotes the class represented by a fixed meridian of $K$ and $w^\pm$ are the positive and negative push-offs of $w$ determined by $S$. We then set

\begin{equation} \label{eq:char-knot-quotient}
\rho_\beta(x)=
\begin{cases}
u & \text{ if } x=m; \\
t^{\text{lk}(\beta,\gamma)} & \text{ if } x=\gamma,\ \gamma \subseteq E_S .
\end{cases}
\end{equation}

The condition that $\beta$ is a mod-$n$ characteristic knot precisely guarantees that the above map, which as written is defined on $\pi_1(E_S) \ast \langle m\rangle$, descends to a quotient of $\pi_1(E_K)$. To verify this, we need to check that, for any $w\in\pi_1(S)$, 

\begin{equation}\label{eq:homomorphism-check}
\rho_\beta(m w^+ m^{-1})=\rho_\beta(w^-).    
\end{equation}

By Equation~\ref{eq:char-knot-quotient}, we have:
\[
\rho_\beta(m w^+ m^{-1})=u t^{\text{lk}(\beta,w^+)} u^{-1} = t^{-\text{lk}(\beta,w^+)} =t^{\text{lk}(\beta,w^-)}=\rho_\beta(w^-),
\]
where the second equality is a relation in $D_n$ and the penultimate one uses the defining property of a characteristic knot: 
\[
\text{lk}(\beta,w^+) + \text{lk}(\beta,w^-) \equiv0\mod n.
\]

Furthermore, it is shown in~\cite[Proposition~1.2]{CS1984linking} that any dihedral quotient of $\pi_1(E_K)$ arises in this way. We thus obtain the following

\begin{corollary}\label{cor:char-knots-induce-both}
    A knot $K\subseteq S^3$ admits a mod~$n$ characteristic knot if and only if the group of $K$ admits a quotient to the $n$-bottle group $\mathbb{D}_n$,
    \[
     \tilde{\rho}_\beta: \pi_1(E_K)\twoheadrightarrow \langle u, t \mid t^n, \ utu^{-1}=t^{-1}\rangle .
    \]
    In particular, $ \pi_1(E_K)$ surjects to $D_n$ if and only if it surjects to $\mathbb{D}_n$.
\end{corollary}

\begin{proof}
    If $\pi_1(E_K)$ admits $\mathbb{D}_n$ as a quotient, by composing with the natural map $\mathbb{D}_n\twoheadrightarrow{D_n}$, we obtain a quotient $\pi_1(E_K)\twoheadrightarrow{D_n}$. By~\cite[Proposition~1.2]{CS1984linking}, $K$ admits a mod~$n$ characteristic knot.
    
    Conversely, assume that $K$ has a mod~$n$ characteristic knot $\beta\subseteq S$, where $S$ is a Seifert surface for $K$. We may define the desired quotient $\tilde{\rho}_\beta$ exactly as in Equation~\ref{eq:char-knot-quotient}. The proof of Equation~\ref{eq:homomorphism-check}, repeated verbatim, shows that the relation $\tilde{\rho}_\beta(m w^+ m^{-1})=\tilde{\rho}_\beta(w^-)$ holds for any $w\in\pi_1(S)$. Hence, this produces a group homomorphism, as desired. Since $\rho_\beta$ is surjective (and, in particular, its image contains the generators $u$ and $t$), so is $\tilde{\rho}_\beta$.
\end{proof}

Thus, {\it we may regard a mod~$n$ characteristic knot for $K$ as inducing $D_n$ and $\mathbb{D}_n$ quotients of the group of $K$}. Lastly, we note that two mod~$n$ characteristic knots $\beta_1$ and $\beta_2$ for $K$ are {\it equivalent} if they determine the same dihedral (or $n$-bottle) quotient of $\pi_1(E_K)$. When $\beta_1$ and $\beta_2$ lie on the same Seifert surface for $K,$ they are equivalent if and only if $[\beta_1]-[\beta_2]\equiv0\mod{n}$. In particular, it follows that for any dihedral (or $n$-bottle) quotient of a knot group, there is an infinite family of associated characteristic knots in any Seifert surface for $K$.

\subsection{The lifted invariant $\Theta_\rho(K)$ in $\pi_3(K\mathbb{D}_{(2)})$}\label{section:theta}

Recall that the sequence of groups and group homomorphisms
\begin{equation}\label{Eq:colimit}
\Z_3 \xrightarrow{\times 5} \Z_{15} \xrightarrow{\times 7} \Z_{105} \xrightarrow{\times 9} \Z_{945} \xrightarrow{\times 11}  \cdots
\end{equation}
has as colimit the group $\sfrac{\Z_{(2)}}{\Z}$, the subgroup of $\Q/\Z$ represented by the elements below, with $n$ odd.
\[
\sfrac{\Z_{(2)}}{\Z} := \left\{ \frac{m}{n} + \Z \in\Q/\Z \mid  n \text{ odd }
 \right\}.
\]  
Let $\Z$ act on $\sfrac{\Z_{(2)}}{\Z}$ by negation, and define   
\[
\mathbb{D}_{(2)}:= \text{colim}\ \mathbb{D}_n = (\sfrac{\Z_{(2)}}{\Z}) \rtimes \Z.
\]

Now let
\[
K\mathbb{D}_{(2)}:= B\mathbb{D}_{(2)}\cup_S B^2,
\]
where, as usual, $S = S^1$ denotes an embedded circle which represents the basing $u \in \mathbb{D}_{(2)}$.  Note that the inclusion of spaces $B\mathbb{D}_n \hookrightarrow B\mathbb{D}_{(2)}$ induces an inclusion $K\mathbb{\mathbb{D}}_n \hookrightarrow K\mathbb{D}_{(2)}$.

\noindent
{\it Technical note.} 
    In this paper we repeatedly make use of the fact that {homology commutes with colimits}.  A more precise statement can be found in work of John Milnor~\cite{MR0159327}.  In particular,  the homology of CW complexes commutes with colimits of cellular maps; and, of course, up to homotopy, maps between CW complexes can be assumed to be cellular inclusions by taking mapping cylinders of cellular maps.

\begin{lemma}\label{lemma:1-1}
 The injection, $\iota \colon \Z_n \to \frac{\Z_{(2)}}{\Z}$ determines a map
 \[
 \iota \rtimes id \colon K\mathbb{D}_n \to K\mathbb{D}_{(2)}
 \]
 which induces an injection $H_3(K\mathbb{D}_n)\to H_3(K\mathbb{D}_{(2)})$.
 \end{lemma}
\begin{proof}
 Recall that both $K\mathbb{D}_n$ and $K\mathbb{D}_{(2)}$ are constructed by attaching a 2-cell to an Eilenberg-MacLane space  along a 1-cell, $S$, which represents the basing of the corresponding group, $\mathbb{D}_n$ or $\mathbb{D}_{(2)}$. The class represented by $S$ normally generates $\pi_1(B\mathbb{D}_n)$ and $\pi_1(B\mathbb{D}_{(2)})$. Thus, the space $K\mathbb{D}_n$ is simply-connected. Furthermore,  by Lemma~\ref{H3(G)}, $H_2(\mathbb{D}_n) \cong 0$ for all $n$. Since homology commutes with colimits, $H_2(\mathbb{D}_{(2)}) \cong 0$ as well. Therefore, the spaces $K\mathbb{D}_n$ and $K\mathbb{D}_{(2)}$ are 2-connected. By the Hurewicz Theorem, we conclude that $\pi_3(K\mathbb{D}_{(2)}) \cong H_3(K\mathbb{D}_{(2)})$ and $\pi_3(K\mathbb{D}_n)\cong H_3(K\mathbb{D}_n)$ for all $n$.   
 Our lemma now follows from Lemma~\ref{H3(G)} since homology commutes with colimits.  
\end{proof}

\begin{lemma}\label{lemma:homologyKG}
The inclusion-induced homomorphisms $\pi_3(K\mathbb{D}_n) \to \pi_3(K\mathbb{D}_{(2)})$ are injective for all (odd) $n$.	
\end{lemma}

\begin{proof}
As shown in Lemma~\ref{lemma:1-1}, both $K\mathbb {D}_{(2)}$ and $K\mathbb{D}_n$ are $2$-connected for every odd positive integer, $n$. The Hurewicz Theorem implies the homomorphisms labeled $h$ in the commutative diagram  below are isomorphisms. 

\begin{center}  
\begin{tikzcd}
\pi_3(K\mathbb{D}_n) \ar[r,"h","\cong"'] \ar[d,"inc_{\ast}"] & H_3(K\mathbb{D}_n) \ar[d, "inc_{\ast}"]  & H_3(\Z_n) \ar[l, "\cong"] \ar[d,"inc_{\ast}"] \\ 
\pi_3(K\mathbb{D}_{(2)}) \ar[r,"h","\cong"'] & H_3(K\mathbb{D}_{(2)})  & H_3(\sfrac{\Z_{(2)}}{\Z}) \ar[l, "\cong"]
\end{tikzcd}
\end{center} 
The middle vertical arrow is an injection by Lemma~\ref{lemma:1-1}. Therefore, the leftmost vertical arrow is an injection, as claimed.
\end{proof}

\begin{remark}
The above lemma can also be proved using the righthand side of the above diagram, by showing that the horizontal homomorphisms on the right, $H_3(\Z_n)\to H_3(\mathbb{D}_n)$ and $H_3(\Z_{(2)}/\Z)\to H_3(K\mathbb{D}_{(2)})$, are isomorphisms. We encourage the reader to check these computations.
\end{remark}

Summarizing, for any $n$ we have a commutative diagram of {\em based homomorphisms}~(see Definition~\ref{BasedDefn}) where each vertical isomorphism is induced by a splitting, which is determined by a basing in $\mathbb{D}_{(2)}$ -- a basing inherited from a basing for $\mathbb{D}_n$. 
\begin{center}
\begin{equation}
\begin{tikzcd}\label{diagram{iota}}
& \mathbb{D}_n \arrow[r,"\iota_n"]\ar[d,"\cong"] & \mathbb{D}_{(2)} \ar[d, "\cong"]\\
& \Z_n \rtimes \Z \arrow[r, "j_n \rtimes \text{ id}"] & \sfrac{\Z_{(2)}}{\Z} \rtimes \Z
\end{tikzcd}
\end{equation}
\end{center}
Here, $j_n(1) = 1/n$ for $1 \in \Z_n$. As noted earlier, the based inclusion homomorphism $\mathbb{D}_n \to \mathbb{D}_{(2)}$ induces a map (well-defined up to homotopy) 
\[
K\mathbb{D}_n = B\mathbb{D}_n\cup_S B^2 \to B\mathbb{D}_{(2)}\cup_S B^2 = K\mathbb{D}_{(2)},
\]
and an injection (by Lemma~\ref{lemma:homologyKG}) of homotopy groups 
\[
\pi_3(K\mathbb{D}_n) \xrightarrow{\iota_n} \pi_3(K\mathbb{D}_{(2)}).
\]
Note that we are using $\iota_n$ to represent an inclusion of groups as in Diagram~(\ref{diagram{iota}}), as well as a map of spaces, and the induced homomorphism  on homotopy groups.

Using the injection $\iota_n$ above, we may view $\Theta_{\rho}(K)\in \pi_3(K\mathbb{D}_n)$ as an element in $\pi_3(K\mathbb{D}_{(2)})$.  We make this formal in the following definition.

\begin{definition}\label{theta-in-k}
Let $K \subseteq S^3$ be a knot equipped with a based epimorphism $\rho \colon \pi_1(E_K) \to \mathbb{D}_n$, and thus determining a map $\bar{\rho}\colon S^3 \to K\mathbb{D}_n$ which sends a meridian for $K$ to the basing in $\mathbb{D}_n$. Recall from Definition~\ref{inv_univ} that
\[
\Theta_{\rho}(K) \defeq [\bar{\rho}] \in \pi_3(K\mathbb{D}_n).
\]
We now set
\begin{equation}\label{eq:limit-theta}
    \widetilde{\Theta}_{\rho}(K) \defeq \iota_n (\Theta_{\rho}(K)) \in \pi_3( K\mathbb{D}_{(2)}).
\end{equation}
\end{definition}
As we see shortly, to determine whether the original obstruction, $\Theta_{\rho}(K)$, vanishes, it suffices to evaluate $\widetilde{\Theta}_{\rho}(K)$. This will be used in Theorem~\ref{thm:equivalence}, which reveals the connection between these invariants and well-explored linking forms on knots. Specifically,  we will be using Definition \ref{theta-in-k} to relate these obstructions to the $\sfrac{\Z_{(2)}}{\Z}$-linking pairing on the two-fold branched cover of the knot, $K$.

\begin{corollary}~\label{cor:n-surface} 
Suppose $K$ is a $\mathbb{D}_n$-knot. Then  
\[
\widetilde{\Theta}_{\rho}(K) = 0 \in \pi_3(K\mathbb{D}_{(2)}) \iff \Theta_{\rho}(K) = 0 \in \pi_3(K\mathbb{D}_n).
\]
\end{corollary}

\begin{proof} This is an immediate consequence of Lemma~\ref{lemma:homologyKG}.  
\end{proof}

For the following Theorem, the reader may wish to review Definitions~\ref{def:G-Knot} and~\ref{def:G-surface}. 
\begin{theorem}\label{Dn-Surface}
 Suppose $K \subseteq S^3$ is a $\mathbb{D}_n$-knot with epimorphism
 \[
\begin{tikzcd}
\pi_1(E_K) \arrow[r, "\rho", two heads] & \mathbb{D}_n.
\end{tikzcd}
\]
Then $K$ bounds a $\mathbb{D}_k$-surface in $B^4$ with $k > n$ if and only if $K$ bounds a $\mathbb{D}_n$-surface.
\end{theorem}

\begin{proof}
 Since the group $\mathbb{D}_n$ is taut, we know by Corollary~\ref{thm:main} that $K$ bounds a $\mathbb{D}_n$-surface if and only if the invariant $\Theta_\rho(K)$ vanishes. The result then follows from Corollary~\ref{cor:n-surface}. 
\end{proof}

We note that the above Theorem does not imply that every $\mathbb{D}_k$ surface, $k > n$, is a $\mathbb{D}_n$ surface.  It states that if $\rho: \pi_1(E_K)\to \mathbb{D}_n$ is onto, then $K$ bounds a $\mathbb{D}_k$-surface, extending the composition of $\rho$ with the inclusion $\mathbb{D}_n\hookrightarrow\mathbb{D}_k$, if and only if $K$ bounds a (possibly different) $\mathbb{D}_n$ surface.

\subsection{Two-fold branched covers.}\label{twocovers}

Let $\Sigma_2 K$ be the $2$-fold branched cover of $S^3$ branched over $K$. The first homology of this cover is a module over $\Z[\Z_2]$, the group ring of $\Z_2$, where $\Z_2$ acts on $\Sigma_2 K$ via deck transformations.  In particular, the action of the generator $u \in \Z_2$ on $H_1(\Sigma_2 K)$ has order two.

Given a based group $G$ with abelianization homomorphism $G \rightarrow \Z$, let
\[
H := \ker\{G \to \Z\}.
\]
Observe that $H = [G, G]$ is the commutator subgroup of $G$. The generator of the quotient, $\Z$, acts on $H$ by conjugation, and the group $G$ splits, 
\[
G \cong H \rtimes \Z.
\]

As previously observed, given a knot $K$ with a fixed meridian, and a quotient group $G$ of $\pi_1(E_K)$, the meridian determines a choice of splitting, $\iota\colon \Z \to G$, of the abelianization epimorphism. We denote the set of {\em based epimorphisms} $\pi_1(E_K)\twoheadrightarrow G$ by 
\[
\hom^\flat(\pi_1(E_K), G).
\]
The following is an easy exercise using the semi-direct product splittings established in Section~\ref{section:linking}.

\begin{lemma}\label{lemma:lifting}
If the conjugation action of the generator of $\Z$, acting on $H \subseteq H \rtimes \Z$, has finite order $k$, then there is a one-to-one correspondence
\[
\hom^\flat(\pi_1(E_K), H \rtimes \Z) \longleftrightarrow \hom^\flat(\pi_1(E_K), H \rtimes \Z_k) 
\]
with the left to right map determined by composing with the quotient homomorphism $\Z \to \Z_k$.
\end{lemma}

\begin{proposition}\label{prop:bijection}
There is a one-to-one bijective correspondence
\begin{equation}\label{leftrightiso}
{hom^\flat(\pi_1(E_K), \mathbb{D}_{(2)})\hspace{10pt}\tikzmark{A}\hspace{50pt}\tikzmark{B}\hspace{10pt}\hom(H_1(\Sigma_2 K), \sfrac{\Z_{(2)}}{\Z})},
\begin{tikzpicture}[overlay,remember picture]
\draw[->] ([shift={(0,0.13)}]A)--([shift={(0,0.13)}]B) node[midway, above]{tr};\draw[<-] ([shift={(0,0)}]A)--([shift={(0,0)}]B) node[midway, below]{\tiny{$\rho$}};
\end{tikzpicture}	
\end{equation}
where $\Sigma_2 K$ is the $2$-fold cover of $S^3$ branched along the knot $K$.
\end{proposition}

\begin{proof} The morphism left-to-right is a type of transfer function, arising from a $2$-fold branched cover. The morphism on the right has a $\Z_2$ action, free away from the  branching set $K$, and leaving the knot pointwise invariant. The commutative diagram below should help the reader follow the ensuing argument.

\begin{figure}[H]
\begin{tikzcd}
& B(\sfrac{\Z_{(2)}}{\Z}) \arrow[]{d}{\rho}\ar[dr,"id"]\\
\Sigma_2 K \ar[r,"\theta'"]\ar[d,"p"]  
		& B(\sfrac{\Z_{(2)}}{\Z} \times \Z) \cup_{\tilde{S}} \widetilde{B}^2 \ar[r,"proj"'] \ar[d,"q"]& B(\sfrac{\Z_{(2)}}{\Z}) \\
S^3 \ar[r,"\theta"] & B\mathbb{D}_{(2)}\cup_{S} B^2 \defeq K\mathbb{D}_{(2)}
\end{tikzcd}
\caption{\label{fig:trans} $\Sigma_2 K$ denotes the $2$-fold cover of $S^3$ branched along the knot $K$. The $2$-disk $\widetilde{B}^2$ is the $2$-fold cover of $B^2$ branched along~$\{0\}$ and $\widetilde{S}=\partial \widetilde{B}^2$.}
\end{figure}

Recall that the group $\mathbb{D}_{(2)} \cong (\sfrac{\Z_{(2)}}{\Z}) \rtimes \Z$ projects onto $\Z$, where $\Z$ is generated by $u \in \mathbb{D}_{(2)}$ and the conjugation action of $u\in \Z$ on the split normal subgroup, $\sfrac{\Z_{(2)}}{\Z}$, is given by negation.  Thus, the unique index two (normal) subgroup of $\mathbb{D}_{(2)}$ is the abelian group $(\sfrac{\Z_{(2)}}{\Z}) \times 2\Z \subseteq (\sfrac{\Z_{(2)}}{\Z}) \rtimes \Z$.

Every based homomorphism $\rho\in \hom^\flat(\pi_1(E_K), \mathbb{D}_{(2)})$ induces a map
\[\theta\colon S^3\to  B\mathbb{D}_{(2)}\cup_{S} B^2
\]
as in Diagram~(\ref{eqn:theta}) from Section~\ref{Gknots}. Since the group $\mathbb{D}_{(2)}$ has a unique homomorphism to $\Z_2$ and $H_1(\mathbb{D}_{(2)}) \cong \Z$, we can pull back the two-fold {\em branched cover} of $B\mathbb{D}_{(2)} \cup_S B^2$, branched over $0 \in B^2$, to obtain the second row of Figure~\ref{fig:trans}. Here, $\tilde{S}$ is the $2$-fold cover of the curve, $S = \partial{B^2} \subseteq K\mathbb{D}_{(2)}$.

Now consider $\tilde{\theta} \colon \Sigma_2 K \to B(\sfrac{\Z_{(2)}}{\Z})$, defined by $\tilde{\theta} = proj \circ \theta'$, where $proj$ is projection to the first factor.
Since $\sfrac{\Z_{(2)}}{\Z}$ is an abelian group, the map $\Sigma_2 K \to B(\sfrac{\Z_{(2)}}{\Z})$ induces a homomorphism $H_1(\Sigma_2 K) \to \sfrac{\Z_{(2)}}{\Z}$. 

This provides one direction of the desired bijection~\eqref{leftrightiso}.

We describe an inverse of the above bijection to complete our proof. Let $\widetilde{\mathcal{N}(K)} \subseteq \Sigma_2 K$ be the transverse pre-image under the map $p$ of an open tubular neighborhood of $K\subseteq S^3$. This is the same as the $2$-fold branched cover, with branched set $K$, of the regular neighborhood of $K \subseteq S^3$.  The $2$-fold cover of $E_K \subseteq S^3$ is $\widetilde{E_K} = \Sigma_2 K \setminus int(\widetilde{\mathcal{N}(K)}).$
The map restricting the domain, $proj\circ \theta'|_{\widetilde{E_K}}$, induces the homomorphism 
\[
proj\circ \theta'|_{\widetilde{E_K}_{\ast}} \colon H_1(\widetilde{E_K}) \to \sfrac{\Z_{(2)}}{\Z}.
\]

Recall that $\pi_1(E_K)$ and $\mathbb{D}_{(2)}$ are based groups~(see Definition~\ref{BasedDefn}), and their basings determine splittings, 
\[
\pi_1(E_K) \cong \pi_1(\widetilde{E_K}) \rtimes \Z, \text{ and } \mathbb{D}_{(2)} \cong (\sfrac{\Z_{(2)}}{\Z}) \rtimes \Z.
\]
We can now define a homomorphism
\[
\rho \colon \pi_1(\widetilde{E_K}) \rtimes \Z \to (\sfrac{\Z_{(2)}}{\Z}) \rtimes \Z \cong \mathbb{D}_{(2)}
\]
by $\rho([x], k) = ((\tilde{\theta}|_{{\widetilde{E_K}}})_{\ast}([x])\ , k)$. Composing with the given splitting isomorphisms we have constructed a homomorphism, also denoted by $\rho$,
\[
\rho \colon \pi_1(E_K) \to \mathbb{D}_{(2)}.
\]

We have now described both maps in the bijection~\eqref{leftrightiso}. These two maps arose by taking a branched $\Z_2$ cover in one direction, and modding out by the $\Z_2$ action in the other.  These are clearly inverse functions: the second map deconstructed the first map, step by step. This establishes our claimed bijection.  
\end{proof}

\subsection{The $\sfrac{\Z_{(2)}}{\Z}$ \-linking form on a $\Z_{(2)}$-homology sphere}\label{linkingpairing}
We review a well-known invariant: the torsion linking form on a $\Z_{(2)}$-homology sphere (a special case of the $\Q/\Z$-linking form on a rational homology sphere).  We begin by recalling the definition and well-know properties of this form.  

We do not claim any new results in this section. The linking form discussed here is well covered in the literature, although most references focus solely on the $\Q/\Z$-linking form for a rational homology sphere. 

The two-fold branched cover of a knot $K\subseteq S^3$, here denoted $\Sigma_2 K$, is a closed, oriented $\Z_{(2)}$-homology $3$-sphere.  Such a manifold comes with a {\em non-singular} $\sfrac{\Z_{(2)}}{\Z}$ valued linking pairing, as stated in the following well-known lemma.  We recall the result here, and derive an explicit Seifert-matrix formula for the linking form, which we use in the proof of Theorem~\ref{thm:equivalence}.

\begin{lemma}\label{lemma:linkingpairing}
Let $\Sigma_2 K$ be a closed, oriented $\Z_{(2)}$-homology $3$-sphere. Then there is an isomorphism
\begin{equation}\label{nonsingular-pairing}
H_1(\Sigma_2 K) \xrightarrow{\cong}  \hom(H_1(\Sigma_2 K); \ \sfrac{\Z_{(2)}}{\Z}) \quad \text{ given by }\quad  [c] \mapsto \lk([c], -).
\end{equation}
\end{lemma}

This isomorphism, which we will derive shortly, determines a non-singular $\qz$-linking form on the $\Z_{(2)}$-homology sphere, $\Sigma_2 K$.  Since $H_1(\Sigma_2 K)$ is finite, for any $[y] \in H_1(\Sigma_2 K)$ there is an odd integer $k$ such that $k\cdot y = \partial(d)$ for some $2$-chain $d$. Then
\[
\lk([x], [y]) = \frac{1}{k}(x \cdot d) \in \frac{\Z_{(2)}}{\Z}.
\]
where $x \cdot d$ is the algebraic intersection number of the $1$ and $2$-chains $x$ and $d$, respectively.

Together with Proposition~\ref{prop:bijection}, this proves the following corollary.

\begin{corollary}\label{cor:correspondence}
Recall that $\mathbb{D}_{(2)} \cong (\sfrac{\Z_{(2)}}{\Z}) \rtimes \Z.$ There is a bijective correspondence
\[
H_1(\Sigma_2K) \xleftrightarrow{\hspace{1mm}[c] \hspace{1mm} \longleftrightarrow  \hspace{1mm} \rho_{\ast} \hspace{1mm}} \hom^{\flat}(\pi_1(E_K), \mathbb{D}_{(2)}),
\]
where $\rho_{\ast}(x) = \lk([c], \ x)$.
\end{corollary}

 Given an element of $\hom(H_1(\Sigma_2 K); \ \sfrac{\Z_{(2)}}{\Z})$, the uniqueness of the corresponding element $[c]$ as in Lemma~\ref{lemma:linkingpairing}, as well as an explicit formula for a linking pairing in this lemma, plays a significant role in the results that follow.  Formulae and proofs for this Lemma can be found throughout the literature.  But, given its importance in  the proof of Theorem~\ref{thm:equivalence}, we derive a careful construction of this linking form, along with an explicit formula for the correspondence stated in Lemma~\ref{lemma:linkingpairing}.\\

\subsection{Computing the linking pairing of Lemma~\ref{lemma:linkingpairing}.}~\label{Subsection:Seifert}
In what follows, fix a knot $K$ and, as always, let $\Sigma_2 K$ denote the $2$-fold cover of $S^3$ branched along $K$.  A standard argument shows this $2$-fold cover of $K \subseteq S^3$ is a $\Z_{(2)}$-homology sphere. See, for example,~\cite[Chapter~8D]{rolfsen1976knots}.

Let $E_S \subseteq S^3$ denote the exterior of a Seifert surface $S$ for $K$.  Let $V$ be the Seifert matrix associated to the Seifert surface $S$ with a symplectic basis,  $e_1, \ldots, e_{2g}$,  where $g$ is the genus of $S$.  Alexander duality provides a dual basis of curves $f_1, \ldots, f_{2g} \in H_1(E_S)$.  That is, if $\lk_{S^3}$ denotes the linking number in $S^3$, then for all $1 \leq i, j \leq g$, we have
\[
\lk_{S^3}(f_i, e_j) = \delta_{ij}.
\]
To an element $f_i \in H_1(E_S)$ we associate an element $f_i^* \in \hom(H_1(S), \Z)$ by setting $f_i^*(e_j) = \lk_{S^3}(f_i, e_j) = \delta_{ij}$.

From the above, we observe the following formula:
\begin{equation}\label{eqn:inner_product}
\lk_{S^3}\left(\sum_{i=1}^{2g} a_i f_i,  \sum_{j=1}^{2g} b_j e_j\right) = \sum_{i=1}^{2g}\sum_{j=1}^{2g} a_i b_j f_i^*(e_j) = \sum_{i =1}^{2g} a_ib_i f_i^*(e_i) = \sum_{i=1}^{2g}\ a_ib_i \in \Z.
\end{equation}
This is the standard inner product if we represent two vectors by their coordinates with respect to the given bases.

\subsubsection{The homology of a 2-fold branched cover.}\label{2branches} We recall the well-known presentation for the homology of a $2$-fold branched cover, $\Sigma_2 K$, of a knot $K$.  We include this computation to make the text self-contained.

Consider the usual Mayer-Vietoris sequence to compute the Alexander module over $\Z[t, t^{-1}]$ of $K$. (For instance, see \cite{kervaire1978MR0521731}.)
\[
0 \to H_1(S)\otimes \Z[\Z]  \xrightarrow{V - t V^{T}} H_1(E_S)\otimes \Z[\Z] \to H_1(E_K; \Z[\Z]) \to \Z \to 0.
\]
Here we use the traditional notation, multiplication by $t$, to represent the conjugate action of the meridian of the knot $K$, which is the generator of $\Z$ and acts on these $\Z$ covers as the generator of the group of the deck transformations.  Tensoring with
\[
\Z[\Z_2] \cong \Z[t, t^{-1}]/(t^2 -1),
\]
we get an exact sequence of $\Z[\Z_2]$-modules as follows.
\[
0 \to H_1(S) \otimes_\Z \Z[\Z_2] \xrightarrow{V - t V^{T}} H_1(E_S)\otimes_\Z \Z[\Z_2] \to H_1(E_K; \Z[\Z_2]) \to \Z \to 0.
\]
Here, the Alexander polynomial is $p(t) = det(V-tV^T)$. Since $H_1(E_K) \cong \Z$, we know the det$(V-V^T) = 1$, and $V - V^T$ is invertible.

The meridian for $K$ splits the last homomorphism in this sequence, and appropriately attaching $S^1 \times B^2$ to the two-fold cover of $E_K$  along the bounding torus, we obtain a presentation of the $2$-fold branched cover over $K$ as follows:
\begin{equation}\label{Eq:presentation}
0 \to H_1(S) \otimes_\Z \Z[\Z_2] \xrightarrow{V - t V^{T}} H_1(E_S)\otimes_\Z \Z[\Z_2] \to H_1(\Sigma_2 K) \to  0.
\end{equation}

Finally, we have an equality of ideals
\begin{equation}\label{eq:T+1}
(t+1, p(t)) = (t^2-1, p(t)) \subseteq \Z[t, t^{-1}],
\end{equation}
where $p(t)$ is the Alexander polynomial for $K$. To see this, let $\mathscr{I} = (t^2 - 1, p(t))$. It clearly suffices to show $t+1 \in \mathscr{I}$ since $(t^2-1)\in (t+1)$.  Modulo a multiple of $t^2-1$, we know $p(t) = at + b$ for some integers $a, b$. We also know that $p(1)  = 1$ since $H_1(E_K, S^1) \cong 0$, where $S^1$ represents the meridian of $K$.  Therefore, $a+b=1$ and, again up to a multiple of $t^2-1$, we have $p(t) = at + (1-a)$.
 Thus, using that $at^2-a$ is in $\mathscr{I}$, we see that
\[
at + (1-a) \in \mathscr{I} \mbox{ and } a + (1-a)t \in \mathscr{I}.
\]
The second equality above is obtained by multiplying the first by the invertible element, $t$. Now adding these two elements of $\mathscr{I}$ we obtain:
\[
1 + t = at+(1 - a) +  a + (1-a)t \in \mathscr{I}.
\] 

We can tensor the displayed exact sequence~(\ref{Eq:presentation}) over $\Z[t, t^{-1}]$ with 
\[
\Z^t:= \Z[t, t^{-1}]/(t+1).
\]
Here $\Z^t$ is the integers as a $\Z[t, t^{-1}]$ module where $t$ acts by negation.   Since $t=-1$, the matrix $V - t V^T$ becomes $A: = V + V^T$, where $V^T$ is the transpose of $V$.  That is, we have an exact sequence as follows:
\begin{equation}\label{presentation2}
\{0\} \to H_1(S) \xrightarrow{A} H_1(E_S) \to H_1(\Sigma_2 K) \to \{0\} 
\end{equation}

\subsection{The linking pairing on $\Sigma_2 K$ and the Seifert matrix for $K$}\label{linkingM2K}
Recall that the $\sfrac{\Z_{(2)}}{\Z}$-linking pairing 
$$\lk \colon H_1(\Sigma_2 K) \to \Z_{(2)}/{\Z}$$ 
is defined as follows:
Let $x, y$ be loops in $E_S$ lifting  elements $[x], [y] \in H_1(\Sigma_2 K)$ from the sequence above. Then the value of the $\sfrac{\Z_{(2)}}{\Z}$-linking pairing on $H_1(\Sigma_2 K)$ is given by the following, where on the right side of the equality we are taking the homology classes of the lifts of $x, y$ in $H_1(S)$.

\begin{equation}\label{Eq:characteristic}
\lk([x], [y]) = [x]^T A^{-1}[y]\in  {\Z_{(2)}}/{\Z}.
\end{equation}

\noindent As noted previously, $A = V + V^T$, where $V$ is the Seifert matrix for $K$.  The vector $[x]^T$ is the transpose of an element in $H_1(E_S)$ with coordinates in the basis $\{f_1, \ldots f_{2g}\}$.  We know that $H_1(\Sigma_2 K)$ has odd order (it is the 2-fold branched cover of a sphere over a knot), or, equivalently, $\det(A)$ is odd.  So  $A^{-1}([y])$ is a vector in $H_1(S)\otimes \Z_{(2)}$ with coordinates in the basis $\{e_1, \ldots, e_{2g}\}$. As in equation~\eqref{eqn:inner_product} we have
\[
\lk_{S^3}([x], [y]) = \lk \left(\sum a_i f_i, \sum b_j e_j\right) = \sum a_ib_i \in \Z.
\]
That is, this is the standard inner product (mod $\Z$) where the vectors are written with coordinates in the given bases.

As before, we denote the Alexander polynomial of $K$ by $p(t)$, and $t$ also denotes the generator of action of the meridian of $K$ on the $\Z$ cover of $K$. That is, we regard the variable $t$ as denoting the generator of $H_1(E_K;\mathbb{Z})$. The action of $t$ on $H_2(\Sigma_2 K)$, as a deck transformation, is given by multiplication by $-1$ since $t+1 = 0 \in \Z^t:= \Z[t, t^{-1}]/(t+1)$. Thus, $\text{det}(A) = p(-1)$, which is an odd number since $p(t)$ is symmetric, and $p(1) = 1$.  Let $m = p(-1)$.  Then
\[
[x]^T A^{-1}[y] \in \Z\left[1/m\right]/\Z \subseteq \sfrac{\Z_{(2)}}{\Z},
\]
since $A^{-1}$ is a matrix with coordinates in $\Z[1/m]$.  

Suppose we are given two elements $[x], [y] \in H_1(\Sigma_2 K)$.  Since $\Sigma_2 K$ is a $\Z_{(2)}$-homology sphere, there is an odd integer $k$ such that $k[y] = 0$.  Then $ky = \partial w$ where $w$ is a $2$-chain in $\Sigma_2 K$.  It follows from Alexander Duality that  
\[
\lk_{\Sigma_2 K}([x], [y]) = (1/k) x\cdot w \in \sfrac{\Z_{(2)}}{\Z}.
\]
Here, $x \cdot w$ is the signed intersection number of $x$ and $w$.  

To clarify the reference to Alexander Duality, it is an exercise to see that the above formula coincides with the previously given formula $[x]A^{-1}[y]$ for $[x],[y] \in H_1(\Sigma_2 K)$.  To see this, one needs to push $x$ off the Seifert surface $S$ to get an element in $H_1(E_S) \cong \hom(H_1(S), \Z)$ where this isomorphism applied to these basis elements is $f_j \mapsto f_j^*$.

The following known result plays a key role in Theorem~\ref{thm:equivalence}, so we review it.

\begin{lemma}\label{lemma:CupFormula}
Suppose $[a] \in H_1(\Sigma_2 K)$ is an element of order $k$, and $[b] \in H_2(\Sigma_2 K; \Z_k)$ such that $\partial b=k\cdot a$.  Let $\alpha \in C^2(\Sigma_2 K; \Z_k)$ be defined by $\alpha (x) = a \cdot x \in \Z_k$, where $a \cdot x$ is the algebraic intersection number of $a$ and $x\mod k$. Here, $\alpha$ is a cocycle such that $[\alpha] \in H^2(\Sigma_2 K;\Z_k)$ is the Poincar\'e dual to $[a] \in H_1(\Sigma_2 K)$.   Similarly, $[\beta] = PD([b]) \in H^1(\Sigma_2 K; \Z_k)$ satisfies $\beta(y) = b \cdot y \mod k$, where $b\cdot y$ is the algebraic intersection number of $b$  and $y$. 

Then
\begin{equation}\label{linking:cuproduct}
\lk([a], [b]) = \iota(\langle [\alpha] \cup [\beta], [\Sigma_2 K] \rangle) \in \sfrac{\Z_{(2)}}{\Z}.
\end{equation}
Here, $\iota \colon \Z_k \to \sfrac{\Z_{(2)}}{\Z}$ is defined by $\iota(1) = \frac{1}{k} +\Z \in \sfrac{\Z_{(2)}}{\Z}$.
\end{lemma}

\begin{proof}  This employs the well-known connection between intersection theory and Poincar\'e Duality.

By hypothesis, $[a] \in H_1(\Sigma_2 K)$ is an element of order $k$.  Then $k\cdot a = \partial b$.  This implies $b \in C_2(\Sigma_2 K)$ is represented by an immersed surface in $\Sigma_2 K$ with boundary $k\cdot a$.  In particular, this implies 
$[a], [b]\in H_*(\Sigma_2 K; \Z_k)$.

There is a $2$-cocycle $\alpha \colon  C_2(\Sigma_2 K) \to  \Z_k$ defined by $\alpha(x) = a \cdot x \in  \Z_k$, where $a \cdot x$ is the signed intersection number modulo $k$ of the cycles 
$a$ and $x$.  Note that $[\alpha]$ is the Poincar\'e dual of $[a]$, that is, 
\[
[\alpha] = PD([a]) \in H^2(\Sigma_2 K; \Z_k). 
\]

Similarly, let $[\beta] = PD([b]) \in H^1(\Sigma_2 K; \Z_k)$ be the Poincar\'e dual of $[b]$. That is, 
\[
\beta \colon C_1 (\Sigma_2 K) \to \Z _k \text{ is defined by } \beta (y) = b \cdot y \mod k.  
\]
where $b\cdot y$ is the algebraic intersection number of $b$ and $y$.

Since $H_3(\Sigma_2 K; \Z_k) \cong \Z_k$,  generated by the fundamental class $[\Sigma_2 K]$, the signed intersection number satisfies the formula
\begin{equation}\label{cupprod}
a \cdot b = \langle \alpha \cup \beta, [\Sigma_2 K] \rangle \in \Z_k \xhookrightarrow{\iota} \qz.
\end{equation}
Here, the brackets, $\langle -, - \rangle$, represent the Kronecker product, that is, the evaluation homomorphism to $\Z_k$, and $\iota(1) = \frac{1}{k}$.
\end{proof}

\subsection{Torsion linking forms and maps between $\Z_{(2)}$-homology spheres.}\label{torlinking}

We now review how to geometrically interpret a map between $\Z_{(2)}$-homology spheres as it relates to their $\sfrac{\Z_{(2)}}{\Z}$ linking forms.

Suppose
\[
\rho : M_1 \to M_2
\]
is an orientation preserving map between oriented $\Z_{(2)}$-homology $3$-spheres $M_1$ and $M_2$. Given an embedded loop 
\[
a \in C_1(M_2;\Z_k),
\]
we can assume, after a homotopy, that $\rho$ is transverse to $a$.  Then 
\begin{equation}\label{EQ:PD}
PD\circ\rho^{\ast}\circ PD^{-1}([a]) = [\rho^{-1}(a)].
\end{equation}

Now suppose we have an epimorphism $\rho_* \colon \pi_1(\Sigma_2 K) \to \Z_k$, and after abelianization, an epimorphism that we will also denote by $\rho_*$,
\[
\rho_* \colon H_1(\Sigma_2 K) \to \Z_k.
\]
This is the induced homomorphism of a map $\rho \colon \Sigma_2 K \to B\Z_k$.  Construct $B\Z_k$ by adding cells of dimension $\geq 4$ to the lens space $L(k, 1)$, where $L(k, 1)$ has the standard cellular structure with one cell in each dimension $\leq 3$.  After a homotopy we can assume $\rho$ is cellular, and, by restricting to the three-skeleton of the codomain, we have a map  $\Sigma_2 K \to L(k,1)$ which we still denote by $\rho$. 

Summarizing so far, we may assume that the given map $\rho_* \colon H_1(\Sigma_2 K) \to \Z_k$ is induced by a map
\[
\rho \colon \Sigma_2 K \to L(k,1).
\]

One easily checks that if $a$ is the embedded loop represented by the $1$-skeleton of $L(k,1)$, then $\lk([a], [a]) = 1/k$.  So if 
\[
\rho([x]) = m\cdot [a] \in H_1(L(k,1))
\]
then 
\[
\lk \colon H_1(L(k, 1))\times H_1(L(k, 1)) \to \qz
\]
satisfies 
\begin{equation}\label{Eqn:phi(x)}
\lk([a], \rho([x])) = m\cdot \lk([a], [a]) = \frac{m}{k} \in \qz.
\end{equation}

The next lemma is well-known; we could not locate a reference so we include the proof. By non-singularity of the torsion linking form, every homomorphism obtained as the following composition
\[
H_1(\Sigma_2 K) \to \Z_k \xrightarrow{1 \mapsto 1/k} \qz
\]
is given by linking with $[c] \in H_1(\Sigma_2 K)$ for some {\em unique} choice of $[c]$.  The lemma below identifies that element $[c] \in H_1(\Sigma_2 K)$.

\begin{lemma}\label{lemma:linking-correspondence} Suppose 
\[
\rho_{\ast} \colon H_1(\Sigma_2 K) \to \Z_k,
\]
$k$ odd, is an epimorphism. As noted previously, the above epimorphism is induced by a map 
\[
\rho \colon \Sigma_2 K \to L(k, 1)
\]
which is transverse to the single $1$- and $2$-cells in the standard cell structure on the lens space, $L(k,1)$.  

Let $a$ denote the 
$1$-skeleton of $L(k,1)$, so that $[a] \in H_1(L(k,1))$ is a generator.  Let $c = \rho^{-1}(a)$.
Then $[c] \in H_1(\Sigma_2 K)$ is the unique element with the property that for all $[x] \in H_1(\Sigma_2 K)$,
\[
\iota(\rho_*([x])) = \lk([c], [x]) \in \sfrac{\Z_{(2)}}{\Z},
\]
where, as before, $\iota$ is the homomorphism 
$\Z_k \xrightarrow{\iota} \sfrac{\Z_{(2)}}{\Z} \text{ given by } \iota(1) = 1/k.$

\end{lemma}

\begin{proof}
Let $[a], [c]$ be as in the theorem, and let $b$ represent the $2$-cell of $L(k,1)$. Recall that we assume $\rho$ is transverse to $b$. Let $d = \rho^{-1}(b)$.  Observe that
\[
\partial d = \partial \rho^{-1}(b) = \rho^{-1}(\partial b) = \rho^{-1}(k\cdot a) = k \cdot \rho^{-1}(a) = k\cdot c.
\]

Let  $[a], [b] \in H_*(L(k,1),\Z_k)$ have Poincar\'e duals $[\alpha], [\beta] \in H^*(L(k,1); \Z_k)$. Let $[c], [d] \in H_*(\Sigma_2 K; \Z_k)$ have Poincar\'e duals $[\gamma], [\delta] \in H^*(\Sigma_2 K; \Z_k)$.

Suppose  $[x] \in H_1(\Sigma_2 K; \Z_k)$ and $\rho_*([x]) = m\cdot[a]$.  Then we have the following, where the penultimate equality below follows from equation~\eqref{Eqn:phi(x)}:
\begin{equation*}
\begin{split}
\lk([c], [x]) = \frac{1}{k}(d \cdot x) =  \frac{1}{k}\langle [\delta], [x] \rangle= \frac{1}{k}\langle \rho^*[\beta], [x] \rangle  =  \frac{1}{k} \langle [\beta], \rho_*[x] \rangle \\\\
=  \frac{1}{k} (b \cdot \rho(x))  =\lk([a], \rho_*[x]) = \lk([a], m\cdot[a]) = \frac{m}{k}= \iota(\rho_*(x)). \qedhere
\end{split} 
\end{equation*}    \end{proof}


\section{Dihedral extension criteria}\label{sec:dihedral-criteria}

In Theorem~\ref{thm:equivalence} we give several equivalent criteria for the existence of a surface over which a given dihedral quotient $\rho:\pi_1(E_K)\to D_n$ extends, answering the motivating question of this work. Theorem~\ref{cor:Dn-Dn-bottle} establishes an analogous result for $n$-bottle quotients of knot groups; and Theorem~\ref{cor:nonor-variants} addresses the existence of extensions of such quotients over non-orientable surfaces and surfaces in other ambient 4-manifolds, which turn out to be detected by the same obstruction. 

Before we state and prove these key results, we need to establish a geometric criterion for a given (connected, properly embedded, orientable, smooth) surface in $B^4$ to admit a dihedral quotient of its group. This is analogous to Cappell and Shaneson's characteristic knots in $S^3$ (Definition~\ref{def:chark}).

\subsection{Characteristic surfaces in $B^4$}\label{sec:characteristic-surfaces}

Throughout, $n$ is odd. Let $F \subseteq B^4$ be a properly embedded, connected, oriented surface with connected boundary $\partial F = K \subseteq S^3$. We will show that dihedral quotients $\pi_1(E_K)\twoheadrightarrow D_n$ are induced by mod~$n$ characteristic surfaces for $F$ (Definition~\ref{def:char-surface}) in a manner exactly analogous to the 3-dimensional picture described in the previous section. 

As usual, we will use $\Sigma_2F$ to denote the two-fold cover of $B^4$ branched along $F$. Recall the classical fact that $H_1(\Sigma_2 F)$ is finite; see, for example,~\cite{kauffman1976signature}.

\begin{definition}\label{def:seifert-solid}
Let $F$ be as above and $S\subseteq S^3$ a compatibly oriented Seifert surface for $K=\partial F$. A \emph{Seifert solid} for $F$ is a compact, connected, oriented $3$-manifold $U \subseteq B^4$ such that 
\[
\partial U = F \cup (-S).
\]
\end{definition}

A Seifert solid $U$ has a tubular neighborhood in $B^4$ denoted $N(U) \cong U \times [-1, 1]$ and such that $U = U \times \{0\}$. Let $E := \overline{B^4 \setminus (int N(U))}$ be the exterior of $U$. Its boundary, $\partial E$, contains two copies of $U$, denoted $U_\pm := U \times \{\pm 1\}$. Let $i_\pm \colon U \hookrightarrow E$ denote the natural inclusions. Given $\gamma$ a curve (or, for that matter, other subset) in $U$, denote its images under $i_\pm$ by $\gamma^\pm,$ respectively.

\begin{definition}\label{def:char-surface}
Let $K$ be a knot in $S^3$ and $S$ a Seifert surface for $K$. Let $F\subseteq B^4$ be a connected, oriented, properly embedded surface with boundary $K$. Let $U$ be an oriented Seifert solid for $F$ with $\partial U = F \cup (-S)$. 
A properly embedded oriented surface $G \subseteq U$ is a \emph{mod $n$ characteristic surface for $F$} 
if:
\begin{enumerate}
     \item the class $[G, \partial G]\in H_2(U, \partial U; \Z)$ is primitive;
    \item For every class $x \in H_1(U; \mathbb{Z})$ represented by a curve $\gamma \subseteq \text{int}(U)$, the following condition holds
    \begin{equation}\label{eq:char-surface-condition-n}
    \text{lk}_{B^4}(\gamma^+, G) + \text{lk}_{B^4}(\gamma^-, G) \equiv 0 \pmod n.
    \end{equation}
\end{enumerate}
\end{definition}

\begin{remark}
We may sometimes require that the boundary of a characteristic surface for $F$ be a particular characteristic knot for  $K$. We will do this when we wish to ensure that the dihedral quotient $\rho \colon \pi_1(E_F) \twoheadrightarrow D_n$ associated to a characteristic surface as in the next theorem realizes a given quotient $\pi_1(E_K) \twoheadrightarrow D_n$ when restricted to the boundary.
\end{remark}

\begin{theorem} \label{thm:Dn-iff-char-surface} 
Let $F \subseteq B^4$ be a properly embedded, connected, oriented surface with connected boundary $\partial F = K \subseteq S^3$. 
There exists a surjection $\rho \colon \pi_1(E_F) \twoheadrightarrow D_n$ if and only if $F$ admits a mod $n$ characteristic surface $G \subseteq U$.
\end{theorem}

The proof of this theorem follows the same steps as the analogous result in dimension 3, which we provided in the previous section. We review the necessary preliminaries, some of which were discussed earlier in the paper while others are classically known. The following lemma generalizes the analogous classical result for double branched covers of knots in $S^3$. The proof is similar in both dimensions; see for example~\cite{kauffman1976signature}.

\begin{lemma}\label{lem:cut-exact-seq}
Given $F\subseteq B^4$ a connected, oriented, properly embedded surface
with connected boundary $\partial F = K \subseteq S^3$, and $U$ a Seifert solid for $F$, there is an exact sequence:
\begin{equation}\label{eq:surface-presentation-final}
0 \to H_1(U; \mathbb{Z}) \xrightarrow{A_U} H_1(E_U; \mathbb{Z})
  \to H_1(\Sigma_2F; \mathbb{Z}) \to 0,
\end{equation}
where the map is defined by $A_U([x]) = [x^+] + [x^-]$.
\end{lemma}

We will be using the correspondence, reviewed in the introduction,
\begin{equation}\label{eq:quotients=characters}
\{
\text{ quotients } \pi_1(E_F)\twoheadrightarrow {D}_n
\}
\;\longleftrightarrow\;
\{
\text{ characters } H_1(\Sigma_2F)\twoheadrightarrow \Z_n
\}.
\end{equation}

We also remind the reader that, given a surface $F$ as above and a Seifert solid $U$ for $F$, the group $\pi_1(E_F)$ admits an HNN presentation as follows: 
\begin{equation}\label{eq:HNN-surface}
    \pi_1(E_F) \cong \left( \pi_1(E_U) \ast \langle m \rangle \right) / \langle\langle m w^+ m^{-1} (w^-)^{-1} \mid w \in \pi_1(U) \rangle\rangle,
\end{equation}

where $m$ denotes a fixed meridian of $F$ and, as before, $w^\pm$ denote the positive and negative pushoffs of $w$ determined by $U$.

Finally, we recall our notation for dihedral and $n$-bottle groups:
 \[
       D_n= \langle u, t \mid u^2, \ t^n, \ utu^{-1}=t^{-1}\rangle~~\text{ and }~~ 
    \mathbb{D}_n= \langle \tilde{u}, \tilde{t} \mid \tilde{t}^n, \ \tilde{u}\tilde{t}\tilde{u}^{-1}=\tilde{t}^{-1}\rangle.
\]

\begin{proof}[Proof of Theorem~\ref{thm:Dn-iff-char-surface}]
$(\Leftarrow)$ Let $G\subseteq U$ be a mod~$n$ characteristic surface for~$F$. We set 
\begin{equation} \label{eq:char-surf-quotient}
\rho_G(x)=
\begin{cases}
u & \text{ if } x=m; \\
t^{\text{lk}(G,x)} & \text{ if } x=\gamma,\ \gamma \subseteq E_U .
\end{cases}
\end{equation}
As before, we must verify that the HNN relation $\rho_G(m(w^+) m^{-1}) = \rho_G(w^-)$ is satisfied for any $w\in \pi_1(U)$. This holds since
\[
u t^{\text{lk}(G, (i_+)_*w)} u^{-1} = t^{-\text{lk}(G, (i_+)_*w)} = t^{\text{lk}(G, (i_-)_*w)},
\]
where the last equality uses Equation~\eqref{eq:char-surface-condition-n}. Thus, $\rho_G$ descends to a homomorphism (called by the same name) $\rho_G:\pi_1(E_F)\to D_n$. 

It remains to show that $\rho_G$ is surjective. By definition, its image contains the generating reflection $u$. Secondly, since the class $[G, \partial G]\in H_2(U, \partial U;\Z)$ is primitive, there is an element $\gamma\in H_1(E_U)$ linking $G$ once. By definition, $\rho_G(\gamma)=t,$ the generating rotation. Thus, $\rho_G$ is surjective.

$(\Rightarrow)$ 

Given a surjection to $D_n$, per~\eqref{eq:quotients=characters}, this is equivalent to the existence of a surjective character $\chi \colon H_1(\Sigma_2 F) \twoheadrightarrow \mathbb{Z}/n$. Let $q \colon H_1(E_U) \to H_1(\Sigma_2 F)$ be the quotient map from the exact sequence in Lemma \ref{lem:cut-exact-seq}. Set:
\begin{equation}
\psi := \chi \circ q \colon H_1(E_U) \to \mathbb{Z}/n.
\end{equation}
By the exactness of \eqref{eq:surface-presentation-final}, $\text{Im}(A_U) = \ker(q)$, which implies that the character vanishes on the image of the symmetrized inclusion: 
\begin{equation}\label{eq:psi-circ-Av}
    \psi(A_U(x)) = 0
\end{equation}
for all $x \in H_1(U)$.

We will construct a surface $G$, properly embedded in $U$, such that the character $\psi$ is realized by linking with $G$. 

The integral linking pairing in $B^4$:
\begin{equation}\label{eq:linking-pairing}
\mathcal{L} \colon H_1(E_U; \mathbb{Z}) \times H_2(U, \partial U; \mathbb{Z}) \to \mathbb{Z}.
\end{equation}
vanishes on the torsion subgroups of both homology groups in the domain. Consequently, there is an isomorphism 
\[
(H_2(U, \partial U; \mathbb{Z}))/\text{Torsion} \cong \text{Hom}(H_1(E_U; \mathbb{Z}), \mathbb{Z}).
\]

Recall that our character is defined as a homomorphism $\psi \colon H_1(E_U; \mathbb{Z}) \to \mathbb{Z}/n$. To represent this character geometrically via the linking pairing, we first seek an integral lift $\tilde{\psi} \colon H_1(E_U; \mathbb{Z}) \to \mathbb{Z}$ such that $\tilde{\psi} \pmod n = \psi$. Such a lift exists if and only if $\psi$ vanishes on the torsion subgroup of $H_1(E_U; \mathbb{Z}) \cong H_1(U, \partial U; \mathbb{Z})$.

The long exact sequence in homology for the pair $(U, \partial U)$:
\begin{equation}
\cdots \to H_1(U; \mathbb{Z}) \xrightarrow{j_*} H_1(U, \partial U; \mathbb{Z}) \xrightarrow{\partial_*} \tilde{H}_0(\partial U; \mathbb{Z}) \to \cdots
\end{equation}
implies that the torsion subgroup of $H_1(U, \partial U; \mathbb{Z})$ is entirely contained in the image of the natural inclusion $j_*$. Consider the commutative triangle:
\begin{equation}\label{eq:duality-triangle}
\begin{tikzcd}[row sep=large, column sep=large]
H_1(U; \mathbb{Z}) \arrow[r, "A_U"] \arrow[dr, "j_*"'] & H_1(E_U; \mathbb{Z}) \arrow[d, "\mathcal{D}"', "\cong"] \\
 & H_1(U, \partial U; \mathbb{Z})
\end{tikzcd}
\end{equation}
where the vertical isomorphism is Alexander duality. This implies that all torsion in $H_1(E_U; \mathbb{Z})$ is contained in the image of the map $A_U \colon H_1(U; \mathbb{Z}) \to H_1(E_U; \mathbb{Z})$, defined by $x \mapsto x^+ + x^-$.

Since we have established that $\psi(A_U(x)) = 0$ for all $x \in H_1(U; \mathbb{Z})$, it follows that $\psi$ evaluates to zero on all torsion elements of $H_1(E_U; \mathbb{Z})$. This implies that an integral lift $\tilde\psi_0\colon H_1(E_U;\Z)\to\Z$ with $\tilde\psi_0\bmod n=\psi$ exists; fix one. 
Let $d>0$ generate the image $\Im(\tilde\psi_0)=d\Z$. Since $\psi$ is surjective onto $\Z/n$, we have $\gcd(d,n)=1$.
Choose $a\in\Z$ with $ad\equiv 1\pmod n$, replace $\chi$ by $a\chi$, and denote the resulting character again by $\chi$
(and likewise denote the resulting $\psi$ again by $\psi$). Then
\[
\tilde\psi:=\tilde\psi_0/d\colon H_1(E_U;\Z)\to\Z
\]
is an integral lift of $\psi$ and satisfies $\Im(\tilde\psi)=\Z$.

By non-singularity of $\mathcal L$ on free parts, there exists $\xi\in H_2(U,\partial U;\Z)$ such that
\[
\mathcal L(y,\xi)=\tilde\psi(y)\qquad\text{for all }y\in H_1(E_U;\Z),
\]
and we represent $\xi$ by a properly embedded oriented surface $G\subseteq U$. We make the following observation about the boundary of $G$.

{\bf Claim.} After modifying $G$ by surgeries supported in a collar neighborhood of $\partial U$, we may assume
\[
\partial G\cap K=\emptyset,
\]
so that $\partial G$ decomposes as a disjoint union
\[
\partial G=\partial_F G\sqcup \partial_S G
\qquad\text{with}\qquad
\partial_F G\subseteq \mathrm{int}(F),\ \ \partial_S G\subseteq \mathrm{int}(S).
\]
Moreover, the homology class of $\partial_F G$ in $H_1(F;\Z)$ is divisible by $n$; that is,
\[
[\partial_F G]=n\cdot a\in H_1(F;\Z)
\]
for some $a\in H_1(F;\Z)$.

\begin{proof}[Proof of Claim]
First we ensure that no component of $\partial G$ intersects $K$ (equivalently, intersects both $F$ and $S$).
By transversality, we may assume the boundary $1$--manifold $\partial G$ intersects the curve
$K=\partial F=\partial S$ in a finite set of points. Let $\alpha$ be a component of $\partial G$ which intersects both
$F$ and $S$ nontrivially. Regard $\alpha$ as subdivided into consecutive arcs by the finite set of (transverse)
intersections $\alpha\cap K$. Because $K$ (after rounding corners) separates $F$ and $S$ in $\partial U$, consecutive arcs of $\alpha\setminus K$ are alternately
contained in $F$ and $S$. Let $\gamma_i$ be one component of $K\setminus \alpha$ with $\partial \gamma_i=\{a,b\}$.
Let $\gamma_i\times[-\epsilon,\epsilon]$ be a small neighborhood of $\alpha$ in $F\cup S$, and assume
$\{a,b\}\times[-\epsilon,\epsilon]=\alpha\times[-\epsilon,\epsilon]$. Now surger $\alpha$ by replacing
$\{a,b\}\times[-\epsilon,\epsilon]$ with $\gamma_i\times\{-\epsilon,\epsilon\}$. Pushing the rectangle
$\gamma_i\times[-\epsilon,\epsilon]$ into the interior of $U$ along a collar of the boundary allows us to surger $G$
in the obvious way, without changing the relative homology class $[G,\partial G]\in H_2(U,\partial U;\Z)$
(since the rectangle and its pushoff sweep out a $3$--chain). Performing this operation for each component of
$K\setminus \partial G$ ensures that $K\cap \partial G=\emptyset$, i.e.\ every component of $\partial G$ is contained
in either $\mathrm{int}(F)$ or $\mathrm{int}(S)$. Thus we may write
\[
\partial G=\partial_F G\sqcup \partial_S G,
\qquad
\partial_F G\subseteq \mathrm{int}(F),\ \ \partial_S G\subseteq \mathrm{int}(S).
\]

Next, we will show that $[\partial_F G]\in H_1(F;\Z)$ is divisible by $n$. 

We first record a consequence of \eqref{eq:psi-circ-Av} for curves on $F$.
Let $c\subseteq \mathrm{int}(F)\subseteq \partial U$ be an oriented embedded curve and let $c^\pm\subseteq U_\pm\subseteq \partial E_U$
denote its push--offs. The cylinder $c\times[-1,1]\subseteq \partial U\times[-1,1]\subseteq \partial N(U)\subseteq E_U$
is an annulus with boundary $c^+\sqcup(-c^-)$, hence $[c^+]=[c^-]\in H_1(E_U;\Z)$.
Define $p_F([c]):=[c^+]=[c^-]$. Viewing $[c]\in H_1(F;\Z)$ as an element of $H_1(U;\Z)$ via the inclusion $F\subseteq \partial U\subseteq U$, we have $A_U([c])=[c^+]+[c^-]=2\,p_F([c])$. Thus, \eqref{eq:psi-circ-Av} gives
$2\,\psi(p_F([c]))=0\in \Z/n$. Since $n$ is odd, $\psi(p_F([c]))=0$.

For any $[c]\in H_1(F;\Z)$, the standard boundary/intersection interpretation of linking gives
\[
c\cdot_F (\partial_F G)
=\lk_{B^4}(p_F(c),\,G)
=\mathcal L\big(p_F([c]),[G,\partial G]\big)
=\tilde\psi\big(p_F([c])\big),
\]
where $c\cdot_F(\partial_F G)$ is algebraic intersection on the surface $F$.
On the other hand, since $\psi(p_F([c]))=0\in \Z/n$ for all $[c]\in H_1(F;\Z)$, we have
$\tilde\psi(p_F([c]))\equiv 0\pmod n$, hence
\[
c\cdot_F(\partial_F G)\equiv 0\pmod n
\qquad\text{for all }[c]\in H_1(F;\Z).
\]
Because $F$ has one boundary component, its intersection pairing on $H_1(F;\Z)$ is unimodular. Therefore the condition
$c\cdot_F(\partial_F G)\equiv 0\pmod n$ for all $[c]\in H_1(F;\Z)$ is equivalent to
\[
[\partial_F G]\in n\,H_1(F;\Z),
\]
i.e.\ there exists $a\in H_1(F;\Z)$ with $[\partial_F G]=n\cdot a$.
\end{proof}

To show that $G$ is a mod~$n$ characteristic surface, note that the algebraic condition $\psi(x^+ + x^-) = 0$ (Equation~\eqref{eq:psi-circ-Av}) translates directly to the geometric linking condition:
\begin{equation}
\text{lk}_{B^4}(x^+, G) + \text{lk}_{B^4}(x^-, G) \equiv 0 \pmod n.
\end{equation}
This satisfies condition (2) of Definition \ref{def:char-surface}.

    Finally, since $\Im(\tilde\psi)=\Z$, the class $\xi\in H_2(U,\partial U;\Z)$ determined by $\mathcal L(-,\xi)=\tilde\psi(-)$
is primitive. Hence so is $[G,\partial G]$, completing the proof.
\end{proof}

\begin{corollary}\label{char-surf-induce-both}
   Let $F\subseteq B^4$ be a connected, oriented, properly embedded surface with boundary. The group $\pi_1(E_F)$ admits a quotient to the $n$-bottle group $\mathbb{D}_n$, that is, there exists a surjective homomorphism
    \[
      \tilde{\rho}_\beta: \pi_1(E_F)\twoheadrightarrow \langle u, t \mid t^n, \ utu^{-1}=t^{-1}\rangle ,
    \]
    if and only if $F$ has a mod~$n$ characteristic surface.
\end{corollary}

\begin{proof}
    The proof is virtually identical to that of Corollary~\ref{cor:char-knots-induce-both}, appealing to Theorem~\ref{thm:Dn-iff-char-surface} in the place of~\cite[Proposition~1.2]{CS1984linking}.
\end{proof}

\begin{corollary}\label{cor:dihedral-n-bottle-same-same}
Let $F \subseteq B^4$ be a properly embedded, connected, oriented surface with connected boundary $\partial F = K \subseteq S^3$. The group $\pi_1(E_F)$ surjects to $D_n$ if and only if it surjects to $\mathbb{D}_n$.
\end{corollary}

\begin{proof}
    By Theorem~\ref{thm:Dn-iff-char-surface} and Corollary~\ref{char-surf-induce-both}, the existence of either quotient is equivalent to the existence of a mod~$n$ characteristic surface for $F$.
\end{proof}

\begin{proposition}\label{prop:quotients-are-equivalent}
    Let $K\subseteq S^3$ be a knot which admits a mod~$n$ characteristic knot $\beta$ and let $\rho$ and $\tilde{\rho}$ be the dihedral and $n$-bottle quotients of $\pi_1(E_K)$ induced by $\beta$ in the sense of Section~\ref{sec:beta}. Denote by $F\subseteq B^4$ an orientable, smooth or locally flat surface with $\partial(F)=K$. The homomorphism $\rho$ extends over $\pi_1(E_F)$ if and only if $\tilde{\rho}$ does.  
\end{proposition}

\begin{proof}
The reverse implication is clear: if $\tilde{\rho}$ extends over 
$\pi_1(E_F)$, composing with $\mathbb{D}_n \to D_n$ gives an 
extension of $\rho$.

For the forward implication, the argument is analogous to the one used in the proof of Corollary~\ref{cor:char-knots-induce-both}. Suppose $\bar{\rho}\colon \pi_1(E_F) \to D_n$ 
extends $\rho$. Using the HNN presentation~\eqref{eq:HNN-surface}, 
$\bar{\rho}$ sends the meridian $m \mapsto u$ and elements 
$\gamma \in \pi_1(E_U)$ into $\langle t \rangle \cong \Z_n$, 
since $\langle t \rangle = [D_n, D_n] = \ker(D_n \to \Z_2)$. 
Define $\tilde{\rho}\colon \pi_1(E_F) \to \mathbb{D}_n$ by 
$\tilde{\rho}(m) = u$ and 
$\tilde{\rho}(\gamma) = \bar{\rho}(\gamma) \in \langle t \rangle$, 
identifying $\langle t \rangle \cong \Z_n$ as a subgroup of both 
$D_n$ and $\mathbb{D}_n$. The HNN relation 
$\tilde{\rho}(m w^+ m^{-1}) = \tilde{\rho}(w^-)$ holds because it 
uses only the relation $u t u^{-1} = t^{-1}$, which is present in 
$\mathbb{D}_n$. The relation $u^2 = 1$ is never invoked, so 
$\tilde{\rho}$ is a well-defined homomorphism extending 
$\tilde{\rho}|_{\pi_1(E_K)}$.
\end{proof}

\subsection{The core dihedral criterion}\label{sec:core-criterion}

We now specialize the $G$-Surface Theorem to the cases $G\cong\mathbb{D}_n$ and $G\cong D_n$. In Theorem~\ref{thm:equivalence} and Theorem~\ref{cor:Dn-Dn-bottle}, we give computable necessary and sufficient conditions for the existence of dihedral and $n$-bottle surfaces in $B^4$. In Theorem~\ref{cor:nonor-variants} we prove that the same obstruction detects the existence of extensions of $G$ quotients over non-orientable surfaces in $B^4$ and surfaces in other orientable 4-manifolds with $S^3$ boundary. 

We refer the reader to Section~\ref{section:linking} for the definitions and notation for $D_n$, $\mathbb{D}_n$, and $\mathbb{D}_{(2)}$; and to Section~\ref{sec:beta} for the theory of characteristic knots. It may also be helpful to review Corollary~\ref{cor:correspondence}, which describes, precisely, the bijection
\begin{equation}\label{bijection}
\hom^{\flat}(\pi_1(E_K), \mathbb{D}_{(2)}) \xleftrightarrow{\rho_{\ast} \leftrightarrow [c]} H_1(\Sigma_2 K).
\end{equation}

Recall that $\Theta_{\rho}{(K)}=[\theta] \in \pi_3(K\mathbb{D}_{(2)})$, as given in Definition~\ref{theta-in-k}, while $\theta$ is determined by $\rho \in \hom^{\flat}(\pi_1(E_K), \mathbb{D}_{(2)})$, and defined in Diagram~\eqref{eqn:theta}.  The {\em characteristic knot}, first  defined and studied by S. Cappell and J. L. Shaneson~\cite{CS1984linking}, is reviewed in  Definition~\ref{def:chark}.

The following is the central result of this section. It gives four equivalent conditions --- one geometric, one homotopy-theoretic, and two computable --- for the existence of an orientable dihedral surface in $B^4$ extending a given dihedral quotient. The proof is almost entirely contained Section~\ref{sec:proof-equivalences}; the constructive implication~\ref{thmpart:beta}$\implies$\ref{thmpart:extension} is established independently in Section~\ref{sec:n-bottle}.

\begin{theorem}~\label{thm:equivalence} 
Let $K \subseteq S^3$ be a $D_n$ knot, that is, a knot equipped with a surjection 
$\bar{\rho}\colon \pi_1(E_K) \twoheadrightarrow D_n$, $n$ odd. 
Let $[c] \in H_1(\Sigma_2 K)$ be the element corresponding to 
$\bar{\rho}$ under the bijection~\eqref{bijection}, 
let $S$ be a Seifert surface for $K$ with Seifert matrix~$V$, 
and let $\beta \subseteq \mathring{S}$ be a mod~$n$ characteristic 
knot (Definition~\ref{def:chark}) corresponding to $\bar{\rho}$. 
The following are equivalent:

\begin{enumerate}[(i), itemsep=3mm]

\item\label{part:dih-surf} The quotient $\bar{\rho}$ extends over the 
exterior of an orientable, locally flat surface $F \subseteq B^4$ with 
$\partial F = K$; that is, $\bar{\rho}$ fits into 
Diagram~\eqref{eq:comm-triangle-general}.

\item\label{part:dihedral-invariant} The invariant 
$\Theta_{\bar{\rho}}(K) = 0 \in \pi_3(KD_n)$.

\item\label{upstairs-beta} The self-linking of $[c]$ vanishes:
\[
\lk_{\Sigma_2 K}([c],[c]) = 0 \in \qz.
\]

\item\label{thmpart:beta} The Seifert form evaluates trivially 
mod~$n^2$ on the characteristic knot:
\[
[\beta]^T \cdot (V + V^T) \cdot [\beta] \equiv 0 \pmod{n^2}.
\]

\end{enumerate}
\end{theorem}

\begin{remark} 
Julius L. Shaneson pointed out an alternative proof of the fundamental result~\ref{thmpart:extension}$\implies$\ref{thmpart:beta} of the above theorem, using characteristic classes~\cite{julius-email}.
\end{remark} 

\begin{remark} 
Before producing the proof, we comment on the practical significance of the theorem, as stated in~\ref{part:dih-surf}. The homotopy theoretic invariant $\Theta_\rho(K)$ informs whether a $\mathbb{D}_n$-knot $K$ bounds a $\mathbb{D}_n$-surface. Theorem~\ref{thm:equivalence} allows us to determine whether this obstruction vanishes using classical tools such as the $\qz$ linking pairing on $\Sigma_2 K$.  One easily computes the $\qz$ linking pairing for $K$ starting from a Seifert matrix for $K$; we give a detailed example in Section~\ref{sec:QZcomp}. In Example~\ref{ex:twist} we also show how to evaluate the obstruction using Condition~\ref{thmpart:beta}.
\end{remark}

\begin{theorem}[$\mathbb{D}_n$-surface equivalences]
\label{cor:Dn-Dn-bottle}
Let $K$ be a $D_n$ knot and adopt the notation of Theorem~\ref{thm:equivalence}. Let $\rho \colon \pi_1(E_K) \twoheadrightarrow \mathbb{D}_n$ be the 
lift of $\bar{\rho}$ given by 
Corollary~\ref{cor:char-knots-induce-both}. 
Conditions~\ref{part:dih-surf}--\ref{thmpart:beta} are also equivalent 
to each of the following:

\begin{enumerate}[(i), itemsep=3mm, resume]

\item\label{thmpart:extension} The $\mathbb{D}_n$-knot $(K, \rho)$ 
bounds an orientable $\mathbb{D}_n$-surface $F \subseteq B^4$.

\item\label{part:inv-not lifted} The invariant 
$\Theta_{\rho}(K) = 0 \in \pi_3(K\mathbb{D}_n)$.

\item\label{part:inv}  The stable invariant 
$\widetilde{\Theta}_{\rho}(K) = 0 \in \pi_3(K\mathbb{D}_{(2)})$.

(Recall that $\mathbb{D}_{(2)} = \varinjlim\, \mathbb{D}_n$ is the colimit 
defined in Section~\ref{section:theta} and 
$\widetilde{\Theta}_\rho(K)$ is the image of $\Theta_\rho(K)$ under 
the injection 
$\pi_3(K\mathbb{D}_n) \hookrightarrow \pi_3(K\mathbb{D}_{(2)})$ 
of Lemma~\ref{lemma:homologyKG}.)

\end{enumerate}
\end{theorem}

\begin{theorem}[Disk in a $4$-manifold]\label{cor:disk-in-W}
Let $K$ be a knot equipped with a quotient $\rho:\pi_1(E_K)\to \mathbb{D}_n$. Conditions~\ref{thm:equivalence}~\ref{thmpart:extension}--\ref{thmpart:beta} and \ref{cor:Dn-Dn-bottle}~\ref{thmpart:extension}--\ref{part:inv} are also equivalent to:
\begin{enumerate}[(i), itemsep=3mm, resume]
\item\label{part:disk} There is an orientable $4$-manifold $W$ with $\partial{W} = S^3$, and a locally flat disk $B^2$ properly embedded in $W$ with $\partial{B^2} = K \subseteq S^3$, such that $\rho$ extends over $\pi_1(E_{B^2})$, where $E_{B^2}$ denotes the exterior of the $2$-disk $B^2$ in $W$.

\begin{equation}~\label{diagramEB2}
\begin{tikzcd}
\pi_1(E_K) \arrow[r, "\rho"] \ar[d, "i_*", labels=left] & \mathbb{D}_{n}.\\
\pi_1(E_{B^2}) \arrow[ru, dashed, "\overline{\rho}", labels=below right] &
\end{tikzcd}
\end{equation}
\end{enumerate}
\end{theorem}

\begin{corollary}[Fundamental-class pushforward]\label{cor:fund-class}
Under the hypotheses of Theorem~\ref{thm:equivalence}, conditions~\ref{thm:equivalence}~\ref{thmpart:extension}--\ref{thmpart:beta} are also equivalent to:
\begin{enumerate}[(i), itemsep=3mm, resume]
\item\label{test}  The $2$-fold branched cover $\Sigma_2 K$ of the knot $K \subseteq S^3$, under the homomorphism 
\[
\tilde{\theta} \colon H_1(\Sigma_2 K) \to \qz,
\]
satisfies
\[
\tilde{\theta}_{\ast}([\Sigma_2 K]) = 0 \in H_3(\qz) \cong \qz.
\]
\end{enumerate}
\end{corollary}

\subsection{Proof of the computable equivalences}\label{sec:proof-equivalences}

\begin{proof}[Proof of Theorem~\ref{thm:equivalence}]
 
\ref{part:dih-surf} $\iff$ \ref{part:dihedral-invariant}: 
This is immediate from Corollary~\ref{thm:main}, since $D_n$ is taut. 
\qed
 
\medskip
\noindent \ref{part:dih-surf} $\implies$ \ref{upstairs-beta}: 
Suppose $\bar{\rho}$ extends over the exterior of an orientable surface 
$F \subseteq B^4$ with $\partial F = K$. By 
Corollary~\ref{cor:char-knots-induce-both}, $\bar{\rho}$ lifts to a 
based epimorphism $\rho \colon \pi_1(E_K) \twoheadrightarrow 
\mathbb{D}_n$, and by 
Proposition~\ref{prop:quotients-are-equivalent}, $\rho$ also extends 
over~$E_F$. The $2$-fold transfer construction 
(analogous to Proposition~\ref{prop:bijection}), applied with branch set 
$F \subseteq B^4$, produces a commutative diagram
\[
\begin{tikzcd}
\Sigma_2 K = \partial{(\Sigma_2 F)} \arrow[r, "\tilde{\theta}"] 
  \ar[d, "i_*", labels=left] & B\Z_n \arrow[r] & B\qz.\\
\Sigma_2 F \arrow[ru] &
\end{tikzcd}
\]
Since $\Sigma_2 K = \partial(\Sigma_2 F)$, the homomorphism 
$i_* \colon H_3(\Sigma_2 K) \to H_3(\Sigma_2 F)$ is the zero map, 
so $\tilde{\theta}_*([\Sigma_2 K]) = 0 \in H_3(\Z_n)$.
 
It remains to show that $\tilde{\theta}_*([\Sigma_2 K]) = 0$ implies 
$\lk([c],[c]) = 0$. 
Suppose $\tilde{\theta}_*([\Sigma_2 K]) = m \cdot [L(n,1)]$. 
By the cup product calculation of 
Lemma~\ref{lemma:CupFormula},
\[
\lk([c],[c]) = \frac{m}{n} \in \qz.
\]
Since $\tilde{\theta}_*([\Sigma_2 K]) = 0$, we have $n \mid m$, 
so $\lk([c],[c]) = 0$. \qed
 
\medskip
\noindent \ref{upstairs-beta} $\iff$ \ref{thmpart:beta}: 
 
Recall from Equation~\eqref{Eq:characteristic} that the $\qz$-linking form on $H_1(\Sigma_2 K) \cong \Z^{2g}/A\Z^{2g}$ is given by $\lk([x],[y]) = [x]^T A^{-1} [y] \in \qz$, where $A = V + V^T$.

The class $[c] \in H_1(\Sigma_2 K)$ corresponding to $\bar{\rho}$ under the bijection~\eqref{bijection} is related to the characteristic knot $\beta$ as follows. By Definition~\ref{def:chark}, $A[\beta] \equiv 0 \pmod{n}$, so we may set $\mathbf{w} := \frac{1}{n}A[\beta] \in \Z^{2g}$. The non-singularity of the linking form (Lemma~\ref{lemma:linkingpairing}) and the correspondence of Lemma~\ref{lemma:linking-correspondence} identify $[c]$ with the class represented by $\mathbf{w}$ in $\Z^{2g}/A\Z^{2g}$. Indeed, the character determined by $\beta$ satisfies 
\[
\chi([x]) = \frac{1}{n}[\beta]^T [x] \in \qz
\] 
for all $[x] \in H_1(\Sigma_2 K)$; since $\chi = \lk([c], -)$ by Lemma~\ref{lemma:linking-correspondence}, we have $A^{-1}[c] = \frac{1}{n}[\beta]$ in $\Q^{2g}/\Z^{2g}$, that is, $[c] = [\mathbf{w}]$. Therefore,
\[
\lk([c],[c]) = \mathbf{w}^T A^{-1} \mathbf{w} = \tfrac{1}{n^2}[\beta]^T A \cdot A^{-1} \cdot A[\beta] = \tfrac{1}{n^2}\,[\beta]^T A\, [\beta] \in \qz.
\]
This vanishes in $\qz$ if and only if $[\beta]^T (V + V^T) [\beta] \equiv 0 \pmod{n^2}$.
\qed
 
\medskip
\noindent \ref{thmpart:beta} $\implies$ \ref{part:dih-surf}: 
This is proved constructively in Section~\ref{sec:n-bottle}, where we 
build an explicit dihedral surface whenever 
condition~\ref{thmpart:beta} is satisfied.
\end{proof}

\begin{proof}[Proof of Theorem~\ref{cor:Dn-Dn-bottle}]
These results have essentially been established earlier in the paper. For the convenience of the reader, we specify precisely where.

Theorem~\ref{thm:equivalence}\ref{part:dih-surf} $\iff$ 
Theorem~\ref{cor:Dn-Dn-bottle}\ref{thmpart:extension}: 
This is the content of Proposition~\ref{prop:quotients-are-equivalent}.

Theorem~\ref{cor:Dn-Dn-bottle}\ref{thmpart:extension} $\iff$ 
Theorem~\ref{cor:Dn-Dn-bottle}\ref{part:inv-not lifted}: 
This follows from Corollary~\ref{thm:main}, since $\mathbb{D}_n$ is taut.

Theorem~\ref{cor:Dn-Dn-bottle}\ref{part:inv-not lifted} $\iff$ 
Theorem~\ref{cor:Dn-Dn-bottle}\ref{part:inv}: 
This is Corollary~\ref{cor:n-surface}.
\end{proof}

\begin{proof}[Proof of Theorem~\ref{cor:disk-in-W}]
Theorem~\ref{cor:Dn-Dn-bottle}\ref{part:inv-not lifted}$\iff$\ref{part:disk}:  Since $H_2(B\mathbb{D}_{n}) \cong 0$, the space $K\mathbb{D}_{n}$ is $2$-connected. Using the Hurewicz Theorem for the first isomorphism, and the Atiyah-Hurzebruch spectral sequence (which in this case degenerates in low degrees) for the second isomorphism,  we observe that
\[
\pi_3(K\mathbb{D}_{n}) \cong H_3(K\mathbb{D}_{n}) \cong \Omega_3^{SO}(\mathbb{D}_{n}),
\]
where, $\Omega_3^{SO}\mathbb{D}_{n}$ denotes the oriented bordism group over $\mathbb{D}_{n}$.  Since \ref{thmpart:extension} $\iff$ \ref{part:inv-not lifted},  the obstruction $\Theta_\rho(K) = 0$  $\iff$ $K$ bounds a $\mathbb{D}_{n}$ surface, $F \subseteq B^4$, with trivial normal bundle $\nu(F) \cong F \times B^2 \subseteq B^4$. In this case, we have a commutative diagram as follows, where $\partial{F} = K$:
\[
\begin{tikzcd} 
F \times S^1  \ar[r, "proj"] \ar[d, "inc"] & S^1 \arrow[d,"inc"] \\
E_F \ar[r,dashed, "\overline{\rho}"] &  B\mathbb{D}_{n}\\
\end{tikzcd}
\]
Since a surface $F$ with $\partial F = S^1$ is orientation preserving cobordant (rel boundary) to a disk, $\rho$ extends over a surface exterior in $W$ if and only if it extends over the exterior of a disk in some (possibly different) 
oriented $4$-manifold $W'$ with $\partial W' = S^3$; that is, if and 
only if Diagram~\eqref{diagramEB2} commutes. 

To show that~\ref{part:disk} implies Theorem~\ref{cor:Dn-Dn-bottle}\ref{part:inv-not lifted}, note that Diagram~\eqref{diagramEB2} determines a map 
$W \to K\mathbb{D}_n$ extending $\theta \colon S^3 \to K\mathbb{D}_n$, 
giving a null-bordism of $\theta$ over $K\mathbb{D}_n$. Using the isomorphism $\pi_3(K\mathbb{D}_n) \cong \Omega_3^{SO}(K\mathbb{D}_n)$, 
we may conclude that $\Theta_\rho(K) = 0$, since the Hurewicz map sends $\Theta_\rho(K) = [\theta] \in \pi_3(K\mathbb{D}_n)$ to the bordism class $[S^3, \theta] \in \Omega_3^{SO}(K\mathbb{D}_n)$. Thus, a null-bordism of $\theta$ directly witnesses the vanishing of $\Theta_\rho(K)$.
\end{proof}

\begin{proof}[Proof of Corollary~\ref{cor:fund-class}]
In the proof of Theorem~\ref{thm:equivalence}, the implication~\ref{thm:equivalence}\ref{part:dih-surf}$\implies$~\ref{thm:equivalence}\ref{upstairs-beta} factors as \ref{part:dih-surf}$\implies$~\ref{cor:fund-class}\ref{test}$\implies$\ref{upstairs-beta}. Since conditions~\ref{part:dih-surf}--\ref{thmpart:beta} of Theorem~\ref{thm:equivalence} are equivalent, condition~\ref{test} is equivalent to them as well.
\end{proof}

\subsection{Non-orientable and ambient-4-manifold variants}\label{sec:nonor-variants}

The following two lemmas are the non-orientable counterparts of Lemma~\ref{lemma:0-framed-Section}. For non-orientable $F$, the normal bundle is non-trivial, so the proof of Lemma~\ref{lemma:0-framed-Section} does not apply. In the case $G = D_n$, however, the order-$2$ property of reflections allows us to provide a replacement.

\begin{lemma}\label{lemma:nonor-section}
Let $(F, K) \subseteq (B^4, S^3)$ be a properly embedded, compact, locally flat, non-orientable surface with $\partial F = K$. Let $f \colon E_F \to BD_n$ ($n$ odd) be a continuous map such that $f_*(\mu)$ is a reflection $s \in D_n$, where $\mu$ denotes a meridian of $F$. Then there exists a global section $\sigma' \colon F \to S(F)$ of the normal circle bundle such that $(f \circ \mathrm{inc} \circ\, \sigma')_* \colon \pi_1(F) \to D_n$ is trivial.
\end{lemma}

\begin{proof}
Since $F$ is compact with $\partial F \neq \emptyset$, we have $H^2(F; \mathbb{Z}^{w_1}) = 0$, where $\mathbb{Z}^{w_1}$ is the coefficient system twisted by the first Stiefel-Whitney class of the normal bundle. Therefore, a global section $\sigma \colon F \to S(F)$ exists.

Give $F$ a CW structure with one $0$-cell, $(g+1)$ one-cells $a_1, \ldots, a_g, \partial$ (where $\partial$ represents the boundary $K$ and $a_1, \ldots, a_g$ are interior generators), and one $2$-cell attached along $a_1^2 \cdots a_g^2 \cdot \partial^{-1}$.

cWe determine the image of $\sigma$ on the interior $1$-skeleton. The normal circle bundle restricted over each $a_i$ forms a Klein bottle, since $a_i$ is orientation-reversing. In this restricted bundle, $\sigma(a_i)$ and the meridian $\mu$ satisfy $\sigma(a_i)\, \mu\, \sigma(a_i)^{-1} = \mu^{-1}$. Applying $f_*$ gives $f_*(\sigma(a_i)) \cdot s \cdot f_*(\sigma(a_i))^{-1} = s$, so $f_*(\sigma(a_i))$ centralizes $s$. Since $n$ is odd, the centralizer of any reflection in $D_n$ is $\{1, s\}$. Hence $(f \circ \sigma)_*(a_i) \in \{1, s\}$ for each $i$.

Define a modified section $\sigma'$ by setting $\sigma'(a_i) = \sigma(a_i) \cdot \mu^{\epsilon_i}$ with $\epsilon_i \in \{0, 1\}$ chosen so that $(f \circ \sigma')_*(a_i) = 1$, and leaving $\sigma'$ unchanged on $\partial$.

We verify that $\sigma'$ extends over the $2$-cell. Because conjugation by $\sigma(a_i)$ inverts the meridian fiber, the added twist cancels:
\[
\sigma'(a_i)^2 = \sigma(a_i)\, \mu^{\epsilon_i}\, \sigma(a_i)\, \mu^{\epsilon_i} = \sigma(a_i)^2 \cdot \mu^{-\epsilon_i} \cdot \mu^{\epsilon_i} = \sigma(a_i)^2.
\]
Therefore $\prod_{i=1}^g \sigma'(a_i)^2 \cdot \sigma'(\partial)^{-1} = \prod_{i=1}^g \sigma(a_i)^2 \cdot \sigma(\partial)^{-1} = 1$, the last equality holding because $\sigma$ is a global section. Hence $\sigma'$ extends continuously over the $2$-cell.

Finally, $(f \circ \sigma')_*$ is trivial on all of $\pi_1(F)$: it vanishes on each $a_i$ by construction, and on $\partial$ by the attaching relation $(f \circ \sigma')_*(a_1^2 \cdots a_g^2) = (f \circ \sigma')_*(\partial)$.
\end{proof}

\begin{lemma}\label{lemma:nonor-extension}
Under the hypotheses of Lemma~\ref{lemma:nonor-section}, let $\nu(F)$ denote the tubular neighborhood of $F$ in $B^4$ (a disk bundle over $F$). Then the map $f \colon E_F \to BD_n \subseteq KD_n$ extends to a continuous map $\tilde{f} \colon \nu(F) \to KD_n$.
\end{lemma}

\begin{proof}
By Lemma~\ref{lemma:nonor-section}, there exists a global section $\sigma' \colon F \to S(F)$ such that $(f \circ \sigma')_*$ is trivial on $\pi_1(F)$. Since $BD_n$ is an Eilenberg--MacLane space and $F$ is $2$-dimensional, the map $f \circ \sigma'$ is nullhomotopic. By the Homotopy Extension Property, we may assume (after a homotopy of $f$) that $f$ maps $\sigma'(F) \subseteq S(F)$ to the basepoint $\ast \in BD_n \subseteq KD_n$.

Recall that $KD_n = BD_n \cup_s e^2$ is formed by attaching a $2$-cell along the loop representing $s$. Let $\Phi \colon D^2 \to e^2$ be the characteristic map of this $2$-cell, with $\Phi|_{S^1}$ tracing $s$. Since $s = s^{-1}$ in $D_n$, we may choose $\Phi$ to be equivariant under the reflection $(r, \theta) \mapsto (r, -\theta)$.

We extend $f$ over $\nu(F)$ fiber-wise: for each $x \in F$, parameterize the disk fiber $D^2_x$ using polar coordinates $(r, \theta)$ with $\theta = 0$ aligned to $\sigma'(x)$. Define $\tilde{f}(x, r, \theta) = \Phi(r, \theta)$. The transition functions of the non-orientable disk bundle act by fiber reflection $(r, \theta) \mapsto (r, -\theta)$ along orientation-reversing loops. By the equivariance of $\Phi$, the fiber-wise extensions are compatible across transition functions, yielding a continuous map $\tilde{f} \colon \nu(F) \to KD_n$.
\end{proof}

\begin{remark}\label{rem:nonor-taut}
Lemma~\ref{lemma:nonor-section} relies on two properties of $D_n$ ($n$ odd): the centralizer of a reflection is $\{1, s\}$, and $s^2 = 1$. This raises the question of whether an analogous result holds for all taut groups $G$: is it true that whenever a non-orientable surface $F \subseteq B^4$ admits a quotient $\pi_1(E_F) \twoheadrightarrow G$, there necessarily exists a section of the normal circle bundle that maps trivially to $BG$?
\end{remark}

Using the above lemmas, we can extend Theorem~\ref{thm:equivalence} to non-orientable surfaces and to arbitrary orientable ambient $4$-manifolds with boundary $S^3$.

\begin{theorem}[Non-orientable and ambient variants]\label{cor:nonor-variants}
Under the hypotheses of Theorem~\ref{thm:equivalence}, conditions~\ref{thm:equivalence}~\ref{thmpart:extension}--\ref{thmpart:beta} are also equivalent to each of the following:
\begin{enumerate}[(i), itemsep=3mm, resume]
\item\label{part:nonor-surf} There exists a locally flat, connected, properly embedded, \emph{non-orientable} surface $F \subseteq B^4$ with $\partial F = K$, such that $\bar{\rho}$ extends to a homomorphism $\pi_1(E_F) \to D_n$.

\item\label{part:nonor-in-W} There exists an orientable $4$-manifold $W$ with $\partial W = S^3$ and a locally flat, connected, properly embedded, non-orientable surface $F \subseteq W$ with $\partial F = K$, such that $\bar{\rho}$ extends to a homomorphism $\pi_1(E_F) \to D_n$.
\end{enumerate}
\end{theorem}

\begin{proof}
\ref{part:nonor-surf}$\implies$\ref{part:dihedral-invariant}: By hypothesis, $\bar{\rho}$ extends to a map $f \colon E_F \to BD_n$ where $F \subseteq B^4$ is a non-orientable surface with $\partial F = K$. By Lemma~\ref{lemma:nonor-extension}, $f$ extends to a continuous map $\tilde{f} \colon \nu(F) \to KD_n$. Since $B^4 = E_F \cup_{S(F)} \nu(F)$, it remains to check that $f$ and $\tilde{f}$ agree on the overlap $S(F) = \partial E_F \cap \partial \nu(F)$. By the same argument as in the proof of the $G$-Surface Theorem~\ref{theorem:G-surface} (where the orientable case is treated), the restriction of $f$ to $S(F)$ factors, up to homotopy, through projection onto the meridian; the $2$-cell attached along $s$ in $KD_n$ fills in each meridional circle, so the two maps are compatible. This yields a map $B^4 \to KD_n$ extending $\theta \colon S^3 \to KD_n$, hence $\Theta_{\bar{\rho}}(K) = 0 \in \pi_3(KD_n)$.

\ref{part:dihedral-invariant}$\implies$\ref{part:nonor-surf}: By the equivalence \ref{part:dihedral-invariant}$\iff$\ref{part:dih-surf}, there exists an orientable $D_n$-surface $F \subseteq B^4$ with $\partial F = K$ over which $\bar{\rho}$ extends. Choose a small $4$-ball $B \subseteq \mathrm{int}(B^4)$ meeting $F$ in a trivially embedded disk $D \subseteq F^\circ$. Replace $D$ with an unknotted M\"obius band $M \subseteq B$ having $\partial M = \partial D$, and set $F' = (F \smallsetminus \mathrm{int}(D)) \cup M$. Then $F'$ is non-orientable, locally flat, and $\partial F' = K$. It remains to extend $\bar{\rho}$ over $E_{F'}$. Outside $B$, the surfaces $F$ and $F'$ coincide, so the existing extension $\bar{\rho} \colon \pi_1(E_F) \to D_n$ restricts to the complement of $B$. Inside $B$, the M\"obius band $M$ is unknotted, so $\pi_1(B \smallsetminus \nu(M)) \cong \Z/2$, generated by a meridian $\mu_M$ satisfying $\mu_M^2 = 1$. Sending $\mu_M \mapsto s$ defines a homomorphism $\Z/2 \to D_n$, since $s^2 = 1$. This is compatible with the restriction of $\bar{\rho}$ to $\partial B \smallsetminus \nu(\partial D)$, where both maps send the meridian to $s$. The two maps therefore glue to give a homomorphism $\pi_1(E_{F'}) \to D_n$ extending $\bar{\rho}$.

\ref{part:nonor-surf}$\implies$\ref{part:nonor-in-W}: Take $W = B^4$.

\ref{part:nonor-in-W}$\implies$\ref{part:dihedral-invariant}: The argument is the same as for \ref{part:nonor-surf}$\implies$\ref{part:dihedral-invariant}, with $W$ in place of $B^4$. We note that the proof of Lemma~\ref{lemma:nonor-extension} (and Lemma~\ref{lemma:nonor-section} on which it relies) applies in any orientable ambient $4$-manifold: the existence of a global section uses $H^2(F; \Z^{w_1}) = 0$, which holds because $\partial F \neq \emptyset$; and the bundle type of $\nu(F)$ is determined by $w_1(F)$ in any orientable $W$, since $w_1(\nu_F) = w_1(TF)$ when $TW$ is orientable. Neither ingredient depends on the ambient manifold. The resulting map $\tilde{f} \colon \nu(F) \to KD_n$ glues with $f$ on $E_F$ to yield a map $W \to KD_n$ extending $\theta \colon S^3 \to KD_n$. Since $KD_n$ is $2$-connected, a null-bordism over $KD_n$ suffices to kill $\Theta_{\bar{\rho}}(K)$ (see the proof of Theorem~\ref{cor:disk-in-W}), giving $\Theta_{\bar{\rho}}(K) = 0$.
\end{proof}


  \section{Constructing dihedral surfaces by hand}\label{sec:n-bottle} 

We have seen that Corollary~\ref{thm:main} guarantees the existence of a dihedral surface whenever $\Theta_\rho(K)=0$, via a transversality argument. In this section, we give an independent, constructive proof of the implication \ref{thmpart:beta}~$\implies$~\ref{thmpart:extension} of Theorem~\ref{thm:equivalence}: given a dihedral knot satisfying the Seifert matrix criterion, we build a dihedral surface explicitly. This explicit construction is of interest in its own right, as it produces a concrete surface that can be used, for instance, in computing the $\Xi_n$ invariant~\cite{kjuchukova2018dihedral} or in future applications requiring control over the topology of the branching set.

We use the notation for $D_n$ and $\mathbb{D}_n$ given in Section~\ref{section:linking}. As always, $n$ is odd.

We assume throughout this section that $K\subseteq S^3$ is a knot equipped with a surjection $\rho\colon\pi_1(E_K)\twoheadrightarrow D_n$ satisfying the equivalent conditions of Theorem~\ref{thm:equivalence}. In particular, for any Seifert surface $S$ for $K$ with Seifert matrix $V$ and any mod~$n$ characteristic knot $\beta\subseteq\mathring{S}$ corresponding to $\rho$, condition~\ref{thmpart:beta} gives 
\[
[\beta]^T(V+V^T)[\beta]\equiv 0\pmod{n^2}.
\]
Our construction proceeds in two stages. First, we show that (after stabilization) $\beta$ may be chosen to be $0$-framed (Definition~\ref{def:0-framed-beta}). Second, we use the $0$-framed characteristic knot to build a mod~$n$ characteristic surface $G\subseteq U$ (Definition~\ref{def:char-surface}) inside a Seifert solid $U$, and verify that our construction produces a dihedral surface (Definition~\ref{def:G-surface}).

\subsection{Existence of a 0-framed characteristic knot}

\begin{definition}\label{def:0-framed-beta}
Let $K\subseteq S^3$ be a knot. A knot  $\beta$ is a {\em $0$-framed mod~$n$ characteristic knot} for $K$ if $K$ admits a Seifert surface $S$ with Seifert matrix $V$ and $\beta\subseteq S^\circ$ is such that
\[
(V + V^{T})\cdot[\beta] \equiv 0\mod{n}
\]
and 
\[
[\beta]^T \cdot(V + V^{T})\cdot[\beta] = 0.
\]
\end{definition}

Since $\mathbb{D}_n$ is a quotient of $\pi_1(E_K)$, a group of weight one, we know that $n$ is odd. As seen in Equation~\ref{eq:self-lk-even},
\[
[\beta]^T\cdot(V  +V^{T})\cdot[\beta]=2[\beta]^T\cdot V \cdot[\beta],
\]
which implies that
\[
[\beta]^T\cdot(V  +V^{T})\cdot[\beta] \equiv 0\mod n \iff [\beta]^T\cdot V  \cdot[\beta]\equiv 0\mod n. 
\]
That is, the second condition in Definition~\ref{def:0-framed-beta} is equivalent to the condition that the Seifert surface containing $\beta$ determines the 0-framing on~$\beta$.
\begin{lemma}\label{lem:0-framed-beta}
Let $K\subseteq S^3$ be a knot with Seifert surface $S_1$ and corresponding Seifert matrix $V_1$. Assume that for some positive odd integer $n$, the knot $K$  admits a$\mod{n}$ characteristic knot $\beta_1\subseteq \overset{\circ}{S_1}$ such that $[\beta_1]^T\cdot (V_1 + V_1^{T})\cdot[\beta_1] \equiv 0\mod{n^2}$. Then, $K$ admits a (possibly different) Seifert surface $S$ with Seifert matrix $V$ and a 0-framed mod~$n$ characteristic knot $\beta$ which is equivalent to $\beta_1$ in the sense of ~\cite[Proposition~1.2]{CS1984linking}, that is, $\beta$ and $\beta_1$ determine the same representation $\pi_1(E_K)\twoheadrightarrow D_n$. 
\end{lemma} 

\begin{proof}
By assumption, there exists a non-negative integer $a$ such that 
\[
[\beta_1]^T\cdot (V_1 + V_1^{T})\cdot[\beta_1] = \pm an^2  \stackrel{\text{WLOG}}{=}an^2.
\]\footnote{If the self-linking number of $\beta$ is negative, the argument proceeds identically with the opposite choice of orientation for the stabilizing tori. This has the effect of replacing each instance of a $-1$ in the stabilized Seifert matrix by a 1, leading to the same conclusion.}

Since $([\beta_1]^T\cdot V_1\cdot[\beta_1])^T=[\beta_1]^T \cdot V_1^{T}\cdot[\beta_1],$
we have
\begin{equation}\label{eq:self-lk-even}
an^2 
= 2[\beta_1]^T \cdot V_1^{T}\cdot[\beta_1],
\end{equation}
so, since $n$ is odd, $a$ is even.
 
By Lagrange's four-square theorem, we can write $\frac{a}{2}=a^2_1+ a^2_2+a_3^2+a_4^2$ for some nonnegative integers $a_1, \dots, a_4$. Let $S$ denote the surface obtained by stabilizing $S_1$ with four trivial tori, so $S\cong S_1\#4 T^2.$ We perform the stabilization by taking the connected sum along 4 disks in the interior of $S_1$ which are disjoint from a neighborhood of $\beta_1.$ With respect to an appropriately chosen basis for $H_1(S;\mathbb{Z}),$ the Seifert matrix for $S$ is a block diagonal matrix $V_1\oplus 4\begin{bmatrix*}[r]
    0&-1\\
    0&0
\end{bmatrix*}=:V,$ and the symmetrized Seifert form on $S$, denoted $L_V:=V+V^T$, is given by $L_V=(V_1+V_1^T)\oplus 4\begin{bmatrix*}[r]
    0&- 1\\
    - 1&0
\end{bmatrix*}.$ 
 
Let $[\beta_1]\in H_1(S_1; \mathbb{Z})$ be represented by $(b_1, \dots, b_{2g_{1}})$, with respect to the basis used to obtain $V_1$. We may also regard $\beta_1$ as an embedded curve in $S$ representing the primitive class $(b_1, \dots, b_{2g_{1}}, 0, \dots 0)$. Set 
\[
\beta:=(b_1, \dots, b_{2g_{1}}, a_1n, a_1n, a_2n,a_2n,a_3n,a_3n,a_4n,a_4n).
\]
By assumption, the class $(b_1, \dots, b_{2g_{1}})$ is primitive, and therefore, so is $[\beta]$. We also have $[\beta]-[\beta_1]\equiv 0\mod n,$ which implies that $\beta$ is also a $\mod{n}$ characteristic knot for $K$ (this can also be checked from the definition, using that $[\beta_1]^T(V_1+V_1^T)[\beta]\equiv 0\mod{n}$) and that $\beta$ and $\beta_1$ belong to the same equivalence class of mod~$n$ characteristic knots. That is, $\beta$ and $\beta_1$ define the same dihedral quotient $\pi_1(E_K)\to D_n$. Lastly, we have:
\[
[\beta]^T \cdot(V + V^{T})\cdot[\beta] = an^2 -2\Sigma_{i=1}^4 a_i^2n^2 = n^2(a-2\cdot\frac{a}{2})= 0, 
\]
as claimed.
\end{proof}

\begin{remark}
    In practice, fewer than four -- and as few as zero -- stabilizations may be necessary in order to embed a zero-framed characteristic knot in a Seifert surface. For instance, as seen in Example~\ref{ex:twist}, for all 3-colorable twist knots, a 0-framed characteristic knot can be chosen to be an unknot in a minimal Seifert surface.
\end{remark}

\begin{corollary} \label{cor:0-framed-iff-0-mod-n2}
    Let $K$ be a knot and $\rho:\pi_1(E_K)\to D_n$ a dihedral quotient determined by a characteristic knot $\beta$ embedded in a Seifert surface $S$ for $K$ with Seifert matrix $V$. There is a  {\it 0-framed} characteristic knot which determines $\rho$ if and only if the symmetrized self-linking of $\beta$ satisfies:
    \[
    [\beta]^T(V+V^T)[\beta]\equiv 0 \mod n^2.
    \]
\end{corollary}

\begin{proof}
    $\impliedby$ Lemma~\ref{lem:0-framed-beta} shows that if any characteristic knot which determines $\rho$ has framing $0\mod n^2$, there is also a 0-framed characteristic knot in the equivalence class of $\beta$. 

    $\implies$ assume there exists a 0-framed characteristic knot, $\beta'$, in the equivalence class of characteristic knots which determine $\rho$. By Theorem~\ref{thm:equivalence}, (specifically,~\ref{thmpart:beta} $\iff$ \ref{part:dih-surf}), $\rho$ extends over a dihedral surface in $B^4$. But then, condition~\ref{thmpart:beta} must be satisfied by any characteristic knot $\beta$ corresponding to $\rho$. 
\end{proof}

\subsection{Building a dihedral surface}

\subsubsection{Set-up}
Let $n$ be an odd integer and $K$ a knot equipped with a quotient $\rho:\pi_1(E_K)\twoheadrightarrow D_n$. We assume that $K$ satisfies one (and therefore all) of the equivalent conditions given in Theorem~\ref{thm:equivalence}. By Lemma~\ref{lem:0-framed-beta}, there exists a Seifert surface $S$ for $K$ and a {\it 0-framed} mod~$n$ characteristic knot $\beta\subseteq S^\circ$ such that $\beta$ induces the given homomorphism $\rho$. Let $M\subseteq S^3$ be an oriented Seifert surface for $\beta$.

Fix a regular neighborhood of $\beta$ in $S$ and choose an embedding
\[
\phi:S^1\times[-1,1]\hookrightarrow S^\circ,
\]
with $\phi(S^1\times\{0\})=\beta$ and boundary circles
\[
b^\pm := \phi(S^1\times\{\pm 1\}).
\]
Let $N(\beta):=\phi(S^1\times[-1,1])\subseteq S$.

Next, fix a collar of the boundary of $B^4$ and choose an embedding
\[
S^3\times[0,1]\hookrightarrow B^4,\qquad S^3\times\{0\}=\partial B^4.
\]
We will write $S^3_t:=S^3\times\{t\}$ and refer to the $t$ parameter as \emph{depth}.  For any subset $R\subseteq S^3$ we will denote by $R_t$ the pushoff of $R$ lying at depth $t$, that is, $R_t\subseteq S^3_t$.

The constructions below (of a Seifert solid, dihedral surface and characteristic surface) take place inside this collar of the boundary. We will specify, for each depth $t$, the relevant cross-section in $S^3_t$. Note that we will construct the Seifert solid first; then we will show that  its boundary is  a dihedral surface  (more precisely, the boundary minus the interior of the given Seifert surface $S$ in $S^3$).

\subsubsection{The Seifert solid $U$ and dihedral surface $F$}

We will construct a Seifert solid $U$ with boundary $\partial U= F\cup_K(-S)$. We will show that $F$ is a dihedral surface with mod~$n$ characteristic surface (Definition~\ref{def:char-surface}) properly embedded in $U$.

The 3-manifold $U$ will be built from three pieces: a thickening of $S$; a thickening of the regular neighborhood of $\beta$ in $S$ (appropriately pushed into $B^4$); and a thickening of $M$ (also pushed in). In order to show how the pieces glue together, we must parametrize everything carefully.

{\it First piece.} Set
\[
U_0 := S\times\Big[0,\tfrac18\Big]\subseteq S^3\times\Big[0,\tfrac18\Big],
\]
oriented as a product so that the induced boundary orientation on $S\times\{0\}$ equals $-S\subseteq S^3_0$.
Then
\[
\partial U_0 = \bigl(S\times\{0\}\bigr)\ \cup\ \bigl(S\times\{\tfrac18\}\bigr)\ \cup\ \bigl(K\times[0,\tfrac18]\bigr).
\]

{\it Second piece.} 
Recall that $N(\beta):=\phi(S^1\times[-1,1])\subseteq S$.
Define a linear function
\[
u:[\tfrac14,\tfrac34]\to[-1,1],\qquad u(t):=2-4t,
\]
so $u(\tfrac14)=1$, $u(\tfrac12)=0$, $u(\tfrac34)=-1$.
For $t\in[\tfrac14,\tfrac34]$ define the \emph{sliding circle}
\[
c_t := \phi\big(S^1\times\{u(t)\}\big)\subseteq S^3_t.
\]
Thus $c_{1/4}=b^+_{1/4}$, $c_{1/2}=\beta_{1/2}$, and $c_{3/4}=b^-_{3/4}$.

We define $U_1\subseteq S^3\times[\tfrac18,\tfrac34]$ by specifying its cross-section at depth $t$:
\begin{itemize}[leftmargin=2em]
\item for $t\in[\tfrac18,\tfrac14]$, set $U_1\cap S^3_t = (N(\beta))_t$ (the full annulus, lying at depth $t$);
\item for $t\in[\tfrac14,\tfrac34]$, set $U_1\cap S^3_t = \phi\bigl(S^1\times[-1,u(t)]\bigr)$,
a subannulus at depth $t$ whose ``outer'' boundary is always $b^-_t$ and whose ``inner'' boundary is $c_t$.
\end{itemize}

The boundary of the second piece, $\partial U_1$, is a union of the following pieces:
\begin{itemize}[leftmargin=2em]
\item  $(N(\beta))_{1/8}$;
\item the ``vertical'' annulus $b^+\times[\tfrac18,\tfrac14]$ ;
\item the ``vertical'' annulus $b^-\times[\tfrac18,\tfrac34]$;
\item the ``sliding'' annulus $A:=\bigcup_{t\in[\tfrac14,\tfrac34]} c_t\ \subset\ \bigcup_{t\in[\tfrac14,\tfrac34]} S^3\times[\tfrac14,\tfrac34]$, which is exactly the trace of the inner boundary $c_t$.
\end{itemize}
(The face at depth $\tfrac34$ degenerates because $c_{3/4}=b^-_{3/4}$.)

To glue $U_0$ to $U_1$, we identify $S\times\{\tfrac18\}$ with $S_{1/8}\subseteq S^3_{1/8}$ and glue
along  $(N(\beta))_{1/8}$. We note that the annulus $N(\beta)_{1/8}$ is no longer part of $\partial(U_0\cup U_1)$.

{\it Third piece.}
We construct the last piece of our Seifert solid, $U_2$, as a sweep of Seifert surfaces for the circles $c_t$ which make up the sliding annulus.

Since $\beta$ is a 0-framed characteristic knot, the surface framing on $\beta$ induced by $S$ agrees with the Seifert framing induced by $M$. Hence, after choosing a product neighborhood $M\times[-1,1]$ of $M$ in $S^3$, we may arrange that for each $s\in[-1,1]$ the pushed-off surface $M\times\{s\}$ is an embedded Seifert surface whose boundary is the parallel copy $\phi(S^1\times\{s\})\subseteq N(\beta)$. 

For \(t\in[\tfrac14,\tfrac34]\), set  $s=2-4t$ and $M_t:=(M\times(2-4t))_t\subseteq S^3_t.$ In other words, each $M_t$ is a particular parallel pushoff of $M,$ lying at depth $t$.  By construction, $\partial M_t=c_t.$ We set
\[
U_2:=\bigcup_{t\in[\tfrac14,\tfrac34]} M_t \subseteq S^3\times[\tfrac14,\tfrac34].
\]

Note that 
\[
\partial U_2 = M_{1/4}\ \cup\ M_{3/4}\ \cup\ A,
\]
where $A$ is the annulus traced by $\partial M_t=c_t$ which we defined previously and which is contained in $\partial(U_0\cup U_1)$. We glue on $U_2$ along $A$ and set
\[
U := U_0\cup U_1\cup U_2 \subseteq B^4.
\]

The boundary of $U$ consists of those parts of the boundaries of $U_0, U_1$ and $U_2$ which were not used for gluing. In other words, $\partial U=-S\cup_K F,$ where
\begin{equation}\label{eq:dih surface}
    F:=(S\setminus (N(\beta))^\circ)_\frac{1}{8}\cup \bigl(K\times[0,\tfrac18]\bigr)\ \cup\ \bigl(b^+\times[\tfrac18,\tfrac14]\bigr)\ \cup\
\bigl(b^-\times[\tfrac18,\tfrac34]\bigr)\ \cup\ M_{1/4}\ \cup\ M_{3/4}.
\end{equation}
One may of course round corners, producing a smooth surface.

To show that $\pi_1(E_F)$ admits a surjective homomorphism to $D_n,$  it suffices to prove that $U,$ a Seifert solid for $F,$ contains a mod~$n$ characteristic surface, $G$. We will construct $G$ with $\partial G=\beta$, so that the dihedral quotient induced by $G$ (which is defined on the complement of $U$ by linking with $G$) clearly extends that induced by $\beta$ (which is defined on the complement of $S$ by linking with~$\beta$).

\subsubsection{Constructing the characteristic surface $G$}

Define
\begin{equation}\label{eq:def-G}
    G := \bigl(\beta\times[0,\tfrac12]\bigr)\ \cup\ M_{1/2}\ \subset\ B^4.
\end{equation}
Recall that $M_{1/2}\subseteq S^3_{1/2}$ denotes the copy of $M$ with $\partial M_{1/2}=\beta_{1/2}$. Then $G$ is oriented and, as we will now see, properly embedded in $U$. We have $\partial G=\beta\subseteq -S\subseteq \partial U$. We will next show that $G^\circ\subseteq U^\circ$.

For $t\in(0,\tfrac18)$, $\beta_t\subseteq S^\circ_t\subseteq U_0^\circ$. For $t\in(\tfrac18,\tfrac14)$, the cross-section $U_1\cap S^3_t$ is the full annulus $(N(\beta))_t$, and $\beta_t\subseteq (N(\beta))_t^\circ$ since $\beta=\phi(S^1\times\{0\})$ lies in the interior of $\phi(S^1\times[-1,1])$. For $t\in[\tfrac14,\tfrac12)$, we have $u(t)>0$, hence $\beta_t$ lies in the interior of the annulus $\phi(S^1\times[-1,u(t)])=U_1\cap S^3_t$.

Lastly, $M_{1/2}\subseteq U_2$ and $M_{1/2}^\circ\subseteq U^\circ$. Thus, every interior point of $G$ lies in the interior of $U$, as required.

\begin{figure}[htb!]
\begin{center}
\begin{tikzpicture}[scale=0.75]
\draw[black, thick, <->](-6.75, -5) --(6.75, -5);

\draw[draw opacity = 40, shade]
(-5,-5) -- (-5,-3.3) -- (-3,-3.3) -- (-3,.80) -- (3,-.10) -- (3,-3.3) -- (5,-3.3) -- (5,-5) -- (-5,-5);

\fill (-5, -5) circle(3pt);
\fill (5, -5) circle(3pt);

\path (-5, -1.65) node{$\mathbf{B^4}$};
\path (5, -1.65) node{$\mathbf{B^4}$};

\draw[fill=violet] (2,-5) circle(3pt);
\draw[fill=violet] (-2,-5) circle(3pt);
\draw[violet,thick] (-2,-5) -- (-2,-0.325) -- (2,-0.925) -- (2,-5);

\path (-5,-5.5) node{$\mathbf K$};
\path (5, -5.5) node{$\mathbf K$};
\draw[red,thick](-5,-5) -- (5, -5); 
\path (0,-5.3) node{\color{red} $\mathbf{S}$};

\filldraw[ultra thin, white] (-1,-3.3) -- (-1,-1.45) -- (1,-1.75) -- (1,-3.3) -- cycle;
\draw[blue, ultra thick] (-1,-3.3) -- (-1,-1.45) -- (1,-1.75) -- (1,-3.3) -- (-1,-3.3);

\draw[blue,thick, even odd rule]
(-5,-5) -- (-5,-3.3) -- (-3,-3.3) -- (-3,.80) -- (3,-.10) -- (3,-3.3) -- (5,-3.3) -- (5,-5);

\path (-6.1, -4.65) node{$\mathbf{S^3}$};
\path (6.1, -4.65) node{$\mathbf{S^3}$};

\path (-2, -5.5) node{\color{blue} $\boldsymbol{\beta}$};
\path ( 2, -5.5) node{\color{blue} $\boldsymbol{\beta}$};

\path (-3.5, -4) node{\color{black} $\mathbf{N(\beta)}$};

\path (0,-1.88) node[rotate=-9] {$\color{blue}\mathbf{M_{1/4}}$};
\path (0,0.62) node[rotate=-9] {$\color{blue}\mathbf{M_{3/4}}$};
\path (2.35,-0.62) node {\color{violet} $\mathbf{G}$};

\end{tikzpicture}
\end{center}
    \caption{Constructing a characteristic surface. }\label{Fc}
    \label{diagramN}
\end{figure}
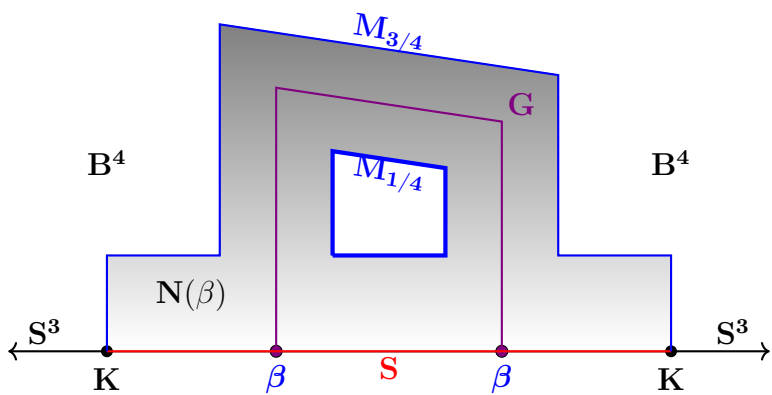

\begin{remark}\label{rem:why-slanted}
One can modify the construction by replacing the slanted annulus
$N(\beta)\times[0,\tfrac12]$ with a straight product annulus sitting in a
single $S^3$-slice, so that the $M$-block and the gluing cylinder occur at the
same depth.  This produces a surface $F$ isotopic to the one above and does
not change the topology of $U$ or $G$. However, with that choice the forthcoming proof of
Theorem~\ref{thm:G_is_char_surface} becomes more involved at the step where we verify the
mod–$n$ linking condition. In the slanted annulus
model this vanishing follows from a simple depth-separation argument,
whereas in the straight model the vanishing of $\Lambda_G$ on the
$H_1(M)$–summand requires an additional linking-number computation.
To avoid this computation, we use the slanted annulus model above.
\end{remark}

\subsubsection{Proof that $G$ is a characteristic surface for $F$}

\begin{theorem}\label{thm:G_is_char_surface}
In the above construction, $G\subseteq U$ is a mod $n$ characteristic surface (Definition~\ref{def:char-surface}) for $F$.
\end{theorem}

\begin{corollary}
    $F$ is a dihedral surface and the homomorphism $\pi_1(E_F)\twoheadrightarrow D_n$ is an extension of the homomorphism $\rho: \pi_1(E_K)\twoheadrightarrow D_n$ determined by the characteristic knot $\beta=\partial G$.
\end{corollary}
\begin{proof}
    The fact that $G$ determines a dihedral quotient $\hat{\rho}: \pi_1(E_F)\twoheadrightarrow D_n$ follows from Theorem~\ref{thm:Dn-iff-char-surface}. The fact that $F$ is in fact a dihedral surface (Definition~\ref{def:G-surface}) follows from Proposition~\ref{prop:dihedral-surface-section}. The fact that $\hat{\rho}$ extends $\rho$ is immediate. As explained in Section~\ref{sec:beta}, $\rho$ sends a fixed meridian $m$ of $K$ to a generating reflection in $D_n$; and the image of curves in the complement of a Seifert surface $S$ is determined by their linking with $\beta$. Similarly, as seen in Section~\ref{sec:characteristic-surfaces}, we let $\hat{\rho}$ send $m$ to the same reflection; and the image of a curve in the complement of $U$ is determined by linking with $G$. Since linking with $G$ restricts (in $S^3$) to linking with $\beta,$ the result follows.
\end{proof}

We will shortly make use of the following lemma. 
\begin{lemma}\label{lem:retract}
$U$ deformation retracts onto the $2$-complex $S\cup_\beta M$ obtained by gluing $M$ to $S$ along $\partial M=\beta\subseteq S^\circ$.
\end{lemma}
\begin{proof}
Collapse the interval direction in $U_0=S\times[0,\tfrac18]$ to $S$, collapse each annular cross-section of $U_1$
to its core circle at depth $\tfrac12$ (which is $\beta_\frac{1}{2}$), and collapse the sweep $U_2=\bigcup_{t}M_t$ along the depth direction to $M_{1/2}=M$.
These collapses are compatible with the gluings (since the gluing loci collapse to $\beta$), and give a deformation retraction onto $S\cup_\beta M$.
\end{proof}

\begin{proof}[Proof of Theorem~\ref{thm:G_is_char_surface}]
We have already checked that $G$ is properly embedded in $U$. We verify the remaining two conditions of Definition~\ref{def:char-surface}.

{\it Proof that the class $[G, \partial G]\in H_2(U,\partial U;\Z)$ is primitive.} 
We will produce a dual curve to $G$ in the interior of $U$.

Because $[\beta]\in H_1(S)$ is primitive, there exists a class $a\in H_1(S)$ whose algebraic intersection number with $\beta$ in $S$ equals 1. Choose an embedded oriented curve $\alpha\subseteq S^\circ$ representing $a$
and transverse to $\beta$ with $\alpha\cap\beta$ consisting of a single point of positive sign.

Fix a depth $0<\varepsilon<\frac{1}{8}$ and consider the embedded curve
\[
\alpha_\varepsilon= \alpha\times\{\varepsilon\}\ \subset\ U^\circ.
\] 
Recall that $G=\beta\times[0,\tfrac12]\cup M_{\frac{1}{2}}$. By construction,
 $\alpha_\varepsilon$ meets the cylinder $\beta\times[0,\tfrac12]\subseteq G$ transversely in exactly one point at depth $\varepsilon$; 
and $\alpha_\varepsilon$ is disjoint from $M_{1/2}$ since $\varepsilon<\frac{1}{2}$.
Therefore the algebraic intersection number satisfies
\[
\alpha_\varepsilon\cdot G = 1.
\]
This tells us that the value of the Poincar\'e-Lefschetz dual of $[G,\partial G]$ on $[\alpha_\varepsilon]$ equals $\alpha_\varepsilon\cdot G=1$.
Hence, the class $[G,\partial G]$ is primitive.

{\it Proof that $G$ satisfies the mod~$n$ linking condition.} 

Define a homomorphism $\Lambda_G:\ H_1(U;\mathbb Z)\longrightarrow \mathbb Z/n$ by
\[
\Lambda_G([\gamma]) := \lk_{B^4}(\gamma^+,G)+\lk_{B^4}(\gamma^-,G)\ \ (\mathrm{mod}\ n),
\]
where $\gamma\subseteq U^\circ$ is any oriented representative of the class $[\gamma]$. Our aim is to show $\Lambda_G\equiv 0\mod n$.

By Lemma~\ref{lem:retract}, $U$ deformation retracts onto $S\cup_\beta M$. Hence
\[
H_1(U)\cong H_1(S\cup_\beta M).
\]
Since $S\cup_\beta M$ is obtained by attaching $M$ to $S$ along the boundary circle
$\partial M=\beta\subseteq S$, every class in $H_1(S\cup_\beta M)$ may be represented by a cycle
lying in $S\cup M$, and therefore by a sum of classes represented by curves in $S$ and curves in $M$.
Because $\Lambda_G$ is a homomorphism, it suffices to show that $\Lambda_G$ vanishes on classes represented by curves in $S$
and on classes represented by curves in $M$.

First we establish the vanishing of $\Lambda_G$ on curves coming from $M$. 
Let $x\in H_1(U)$ be represented by an oriented curve
\[
\gamma\subseteq M_{3/4}^\circ\subseteq S^3_{3/4},
\]
pushed slightly into $U^\circ$.
Then $\gamma^\pm\subseteq B^4\setminus U$ lie in a small neighborhood of the slice $S^3_{3/4}$ in the collar $S^3\times[0,1]$. Without loss of generality, we may assume that each curve $\gamma^\pm$ lies in a slice $S^3_{3/4\pm\varepsilon}$. Hence each $\gamma^\pm$ bounds a surface $A^\pm\subseteq B^4\setminus U$
contained in an arbitrarily small neighborhood of $S^3_{3/4}$.
Since
\[
G\subseteq S^3\times[0,\tfrac12]\cup S^3_{1/2},
\]
we have $A^\pm\cap G=\varnothing$, and therefore
\[
\lk_{B^4}(\gamma^\pm,G)=A^\pm\cdot G=0.
\]
and therefore $\Lambda_G(x)=0$.

Now we establish the vanishing of $\Lambda_G$ on curves coming from $S$. Let $x\in H_1(U)$ be represented by an oriented curve $\gamma\subseteq S^\circ\subseteq S^3_0$.
Push $\gamma$ slightly into $U^\circ$ at some small depth $\varepsilon<\frac{1}{8}$.
Its normal push-offs $\gamma^\pm$ from $U$ lie in $B^4\setminus U$.
In the product region near depth $0$, the complement $B^4\setminus U$ can be identified with
\[
(S^3_0\setminus S)\times[0,\varepsilon]
\]
(up to corner-rounding), so each of $\gamma^\pm$ is isotopic in $B^4\setminus U$ to a push-off of $\gamma$
lying in $S^3_0\setminus S$.
Up to swapping the labels $+$ and $-$, the unordered pair
$\{\gamma^+,\gamma^-\}$ corresponds to the usual Seifert push-offs $\{\gamma_+,\gamma_-\}$ of $\gamma$ from $S$ in $S^3_0$.
Since we only use the sum $\lk(\gamma^+,G)+\lk(\gamma^-,G)$, we need not agonize over which sign corresponds to which pushoff.

Because $G$ is properly embedded in $U$ with $\partial G=\beta\subseteq S^3_0$, and because $\gamma_\pm$ are disjoint from $\beta$,
we may compute the linking numbers between $\gamma_\pm$ and $G$ using spanning surfaces in $S^3_0$:
\[
\lk_{B^4}(\gamma^\pm,G)=\lk_{S^3_0}(\gamma_\pm,\beta).
\]
Hence,
\[
\Lambda_G(x)\equiv \lk_{S^3_0}(\gamma_+,\beta)+\lk_{S^3_0}(\gamma_-,\beta)\mod n.
\]

Let $L_S$ be the Seifert pairing on $H_1(S)$, that is, $L_S(a,b)=\lk_{S^3_0}(a_+,b)$.
Since the unordered pair $\{\gamma^+,\gamma^-\}$ corresponds to the unordered pair of Seifert push-offs
$\{\gamma_+,\gamma_-\}$ in $S^3_0$ determined by $S$, we have
\[
\lk_{S^3_0}(\gamma^+,\beta)+\lk_{S^3_0}(\gamma^-,\beta)
=
\lk_{S^3_0}(\gamma_+,\beta)+\lk_{S^3_0}(\gamma_-,\beta).
\]
In other words,
\[
\Lambda_G(x)\equiv L_S([\gamma],[\beta]) + L_S([\beta],[\gamma])\mod n.
\]
Since $\beta$ is mod-$n$ characteristic with respect to $S$ (Definition~\ref{def:chark}),
the right-hand side is $0$ mod $n$ for all $[\gamma]\in H_1(S)$.
Hence $\Lambda_G$ vanishes on classes represented by curves in $S$.

\medskip
It follows that $\Lambda_G\equiv 0$ on $H_1(U)$, i.e. $G$ satisfies the mod~$n$ linking number condition. 
This completes the proof that $G$ is mod $n$ characteristic for $F$.

\end{proof}

\begin{corollary}
   The surface $F$ thus constructed admits a surjection
    $\pi_1(E_F)\twoheadrightarrow D_n$ extending $\rho$, and is 
    a $D_n$-surface in the sense of Definition~\ref{def:G-surface}.
\end{corollary}

\begin{proof}
    This follows by combining the results of this section: Theorem~\ref{thm:G_is_char_surface} 
shows that $G$ is a mod~$n$ characteristic surface for $F$; by 
Theorem~\ref{thm:Dn-iff-char-surface}, $F$ therefore admits a surjection 
$\pi_1(E_F)\twoheadrightarrow D_n$; and by 
Proposition~\ref{prop:dihedral-surface-section}, $F$ is a $D_n$-surface 
in the sense of Definition~\ref{def:G-surface}.
\end{proof}

This completes the 
constructive proof of the implication 
\ref{thmpart:beta}~$\implies$~\ref{thmpart:extension} of 
Theorem~\ref{thm:equivalence}.


\section{Applications and computations}\label{sec:proof-of-conjecture}

This section provides both background and applications of the dihedral obstruction, Theorem~\ref{thm:equivalence}. First, in~\ref{sec:xi} we briefly recall the $\Xi_n$ invariant of dihedral covers of knots and its applications. Then, in~\ref{thm:n=3}, we provide an independent proof, in the case $n=3$, of the linking number criterion for the existence of a $D_n$-surface given in Theorem~\ref{thm:intro-dihedral}. In fact, Theorem~\ref{thm:n=3} shows the linking number criterion detects the existence of {\it any} locally flat surface extending the given quotient, be it orientable or not, embedded in {\it any} orientable four-manifold with $S^3$ boundary. This argument, which uses a formula for $\Xi_n$ from~\cite{kjuchukova2018dihedral}, has not been previously written anywhere; it was known to the first named author and provided the initial impetus for the present work. Next, in~\ref{section:examples}, we give several illustrations of how to evaluate the obstruction for dihedral groups in practice. For knots whose double branched cover have cyclic first homology we give an even simpler criterion than those previously presented: see Theorem~\ref{thm:2-bridge}. For twist knots whose determinant is $0\mod 3$ we show that, using only the self-linking number of a characteristic knot, we can simultaneously detect whether they bound a $D_3$ surface and obstruct ribbonness (except for $6_1$). Lastly, in~\ref{sec:QZcomp}, we describe in detail a completely general computational procedure for evaluating the dihedral extension criterion, using the $\Q/\Z$ linking form (\ref{sec:QZcomp}) on the double branched cover.

\subsection{The $\Xi_n$ invariant and signature applications}\label{sec:xi}
We review the $\Xi_n$ invariant of irregular dihedral branched covers which was introduced in~\cite{kjuchukova2018dihedral} and which gave the original motivation for this work.

Let $X^4$, $Y^4$ be two closed, oriented smooth (resp. topological) 4-manifolds. Consider an irregular dihedral $n$-fold cover $f: Y^4\to X^4$ whose branching set $B^2\subseteq X^4$ is smoothly (resp. locally flatly) embedded away from finitely many cone singularities $x_i, i=1,\dots, r$. Let $N(x_i)\cong B^4$ denote a small neighborhood of $x_i$ in $X$ with the property that $B\cap N(x_i)$ is  the cone of a knot, $K_i$. When $f^{-1}(\partial N(x_i))$ is connected, $K_i$ is a dihedral knot. Assuming this is the case for all $i,$ we have~\cite{kjuchukova2018dihedral, geske2018signatures}
\begin{equation}\label{eq:sign-dih-cover}
\sigma(Y)=n\sigma(X)-\frac{n-1}{4}e(B)-\sum_{i=1}^r\Xi_n(K_i),
\end{equation}
where $\sigma$ denotes the signature of a 4-manifold and $e$ the normal Euler number of a submanifold. In this setting, $\Xi_n(K_i)$ is the contribution to the signature of $Y$ resulting from the presence of a singularity of type\footnote{We are blissfully ignoring an ambiguity about the sign of $\Xi_n(K_i)$. The indeterminacy can be resolved by specifying -- which we do not do -- whether $K_i$ denotes the boundary of the cone $B\cap N(K_i)$ or the boundary of the punctured surface $B\backslash(B\cap\mathring{N}(K_i))$.} $K_i$ on the branching set. Indeed, if the branching set were smooth (resp. locally flat), the signature of an $n$-fold irregular dihedral cover would satisfy $\sigma(Y)=n\sigma(X)-\frac{n-1}{4}e(B)$, by a standard formula for the signature of a branched covers~\cite{viro1984signature} (resp.~\cite{geske2018signatures}).

As shown in~\cite{kjuchukova2018dihedral}, for a knot $K$ equipped with a dihedral quotient $\rho:\pi_1(E_K)\to D_n$, the quantity $\Xi_n(K)$, an invariant of the pair $(K,\rho)$, satisfies the following equation: 
    \begin{equation}\label{eq:Xi}
       	\Xi_n(K,\rho)=\dfrac{n^2-1}{6n}L_S(\beta,\beta)+\sigma(W(K,\beta))+\sum_{i=1}^{n-1}\sigma_{\zeta^i}(\beta).
    \end{equation}
   In the above formula, $S$ denotes a Seifert surface for $K$ and  $\beta\subseteq \mathring{S}$ is a mod~$n$ characteristic knot which determines the representation $\rho$ in the sense described in Section~\ref{sec:beta}. Furthermore, $L_S$ a symmetrized Seifert form associated to $S$ (that is, if $V$ is a Seifert matrix for $S,$ then $L_S(\beta,\beta)=[\beta]^T(V+V^T)[\beta]$). As before, $\sigma$ denotes the signature of a 4-manifold and $\sigma_{w}$ the Tristram-Levine $w$-signature of a knot for $w$ a complex number with $|w|=1$. By $\zeta$ we denote a primitive $n$-th root of unity. Lastly, the manifold $W(K,c)$ is the Cappell-Shaneson cobordism, constructed in~\cite{CS1984linking}, between the $n$-fold irregular dihedral cover of $K$ determined by $\rho$ and the cyclic $n$-fold cover of $c$. Techniques to compute the $\Xi_n$ invariant using the above expression are developed in~\cite{cahn2021computing, cahn2021linking}. When $K$ bounds a dihedral surface in $B^4$ and this surface is presented in a triplane diagram in the sense of~\cite{meier2017bridge}, a computer program for computing  $\Xi_n(K)$ can be found in~\cite{cahn2023algorithms}. In particular, finding a way to present the dihedral surfaces constructed in Section~\ref{sec:n-bottle} via triplane diagrams would enable an automated computation of the invariant.

One of the applications of the $\Xi_n$ invariant is to test if a $D_n$-knot is ribbon. In~\cite{geske2018signatures} it is shown that if a $D_n$-knot $K$ is ribbon, then there is a representation $\rho:\pi_1(E_K)\twoheadrightarrow D_n$ -- namely any representation that extends over a ribbon surface for $K$ -- for which the $\Xi_n$ invariant satisfies the equation.

\begin{theorem}~\cite[Theorem~3]{geske2018signatures}\label{thm:GKS}
    Let $n>1$ be odd and square-free and $K\subseteq S^3$ a $D_n$-knot. If $K$ is (homotopy) ribbon, then there is a representation $\rho:\pi_1(E_K)\twoheadrightarrow D_n$ whose associated $\Xi_n$ invariant satisfies the inequality
    \begin{equation}\label{eq:ribbon-obstruction}
    |\Xi_n(K, \rho)| \leq \text{rk } H_1(M_n) + \frac{n-1}{2},
\end{equation}
where $M_n$ is the irregular dihedral $n$-fold branched cover of $K$ induced by $\rho$.
\end{theorem}
We remind the reader that the definition of $M_n$ was reviewed in the Introduction to this paper. The above inequality applies for any representation $\rho$ which extends over a homotopy ribbon surface for $K$. For knots with a single dihedral quotient, there is only one $\Xi$ invariant to check.  This is the case in Example~\ref{ex:twist}, where we see that Equation~\ref{eq:ribbon-obstruction} can be applied to detect, {\it using only the self-linking number of a characteristic knot}, the non-ribbonness of all twist knots with quotients to $D_3$ except of course the knot $6_1$.

More generally, the $\Xi_n(K)$ invariant can be used to give a lower bound on the genus of a ribbon dihedral surface in $B^4$ bounded by $K$. This bound is shown to be sharp on a family of knots whose ribbon dihedral genus is unbounded~\cite{cahnkju2018genus}.

\subsection{The $n=3$ theorem}\label{sec:n3-thm}

     We conclude with a surprising application of Equation~\ref{eq:Xi} and~\ref{eq:sign-dih-cover} to detect the non-existence of dihedral surfaces {\it in the broadest sense possible}, in the case $n=3$.

One compelling feature of Equation~\ref{eq:sign-dih-cover} is that it does not require orientability of the $D_n$-surface, nor is it restricted to dihedral surfaces in $B^4$. In combination with Equation~\ref{eq:Xi},
this allows us to give a condition, for $n\equiv0\mod{3}$, which guarantees the non-existence of dihedral surfaces {\it in the broadest sense possible}. This is a considerable strengthening of the implication~\ref{thmpart:extension}$\implies$\ref{thmpart:beta} of Theorem~\ref{thm:equivalence} in the case where 3 divides the order of the dihedral group.

    \begin{theorem}\label{thm:n=3}
        Let $K\subseteq S^3$ be a knot equipped with a quotient $\rho:\pi_1(E_K)\twoheadrightarrow D_n$ where $n\equiv0\mod{3}$. Let $S$ be any Seifert surface for $K$, $V$ a Seifert matrix determined by $S$, and $\beta\subseteq\mathring{S}$ a mod~$n$ characteristic knot corresponding to $\rho$.

        If there exists an orientable (smooth or topological) 4-manifold $X$ with $\partial(X)=S^3$ and a properly embedded (either smooth or locally flat; not necessarily orientable) surface $F\subseteq X$ with $\partial F=K$ such that the homomorphism $\rho:\pi_1(E_K)\twoheadrightarrow D_n$ factors through the inclusion map $i_\ast:\pi_1(E_K)\to \pi_1(X\backslash N(F))$, then
   \begin{equation}\label{eq:linking-condition}
        [\beta]^T(V+V^T)[\beta]\equiv 0\mod{9}.
    \end{equation}
    When $n=3$, the converse holds as well.
    \end{theorem}

     \begin{proof}
Given $K, \rho, F$ and $X$ as above, let $E_F$ denote, as usual, the exterior of $F$ in $X$. By assumption, there exists a homomorphism $\bar{\rho}$ which fits into the commutative diagram:
\begin{equation}
\begin{tikzcd}
\pi_1(E_K) \arrow[r, "\rho"] \ar[d, "i_*", labels=left] & G.\\
\pi_1(E_F) \arrow[ru, dashed, "\overline{\rho}", labels=below right] &
\end{tikzcd}
\end{equation}
Let $f:Y\to X$ be the irregular dihedral $n$-fold cover branched along $F$ and induced by $\bar{\rho}$. Coning off the boundaries of $X$ and $Y$, we obtain a branched cover\footnote{In this construction, the total space $Y\cup_\partial c(\partial(Y))$ is not necessarily a manifold, which is fine.} whose branching set $F\cup_K c(K)$ has an isolated cone singularity of type $K$. Thus, by~\cite[Theorem~2]{geske2018signatures}, we have
\[
\Xi_n(K)=n\sigma(X)-\frac{n-1}{4}e(F) -\sigma(Y),
\]
where, as before, $\sigma$ denotes the (Novikov) signature of a 4-manifold and $e$ the normal Euler number of a submanifold. 
This equation implies that $\Xi_n(K)$ is an integer, since $e(F)$ is even (the branching surface has pre-images of branching index 2) and $n$ is odd (the group $D_n$ is a quotient of a knot group).

We now rewrite Equation~\eqref{eq:Xi} as:
    
    \begin{equation}\label{eq:self-linking}
       \dfrac{n^2-1}{6n}L_V(\beta,\beta)=\Xi_n(K)-\sigma(W(K,\beta))-\sum_{i=1}^{n-1}\sigma_{\zeta^i}(\beta). 
    \end{equation}
    
    Since each terms on the right-hand side takes integer values, we conclude that $6n$ divides $(n^2-1)L_V(\beta,\beta)$. By assumption, $n\equiv0\mod{3}$ which means that $6n\equiv 0\mod{9}$ and moreover that $(3, n^2-1)=1$. Thus, $L_V(\beta,\beta)=[\beta]^T(V+V^T)[\beta]\equiv 0\mod{9},$ as claimed.

     The final statement is a direct application of Theorem~\ref{thm:equivalence}.  More specifically the statement that~\ref{thmpart:beta}$\implies$~\ref{thmpart:extension}.
     \end{proof}

Theorem~\ref{thm:n=3}, which was known to the first author since~\cite{kjuchukova2018dihedral} (though not to be found in the literature), gave the initial impetus for the present paper. Our results grew out of the belief that the hypothesis $n=3$ in Theorem~\ref{thm:n=3} is not needed and the analogous statement should hold for all $n$. This is indeed the case, as proved in Theorem~\ref{thm:equivalence}, though proving it required developing the machinery presented in Part~I.

\subsection{Two-bridge knots, twist knots, and other examples}
~\label{section:examples} 

In this section we offer a few explicit computations of our obstruction, using Theorem~\ref{thm:equivalence}. Specifically, we show that, in the case of two-bridge knots, there is an especially easy way to detect the existence of a dihedral surface. We also give an explicit computation using characteristic knots and the self-linking number condition, Theorem~\ref{thm:equivalence}~\ref{thmpart:beta}. See Section~\ref{sec:QZcomp} for a detailed example of the general computational procedure.

\subsubsection{Knots $K$ with $H_1(\Sigma_2K;\Z)$ cyclic}

\begin{theorem}\label{thm:2-bridge}
Let $K\subseteq S^3$ be a knot equipped with a dihedral quotient 
\[
\rho \colon \pi_1(E_K) \twoheadrightarrow D_n 
\]
and assume $K$ has the property that $H_1(\Sigma_2K;\Z)$ is cyclic (e.g. any two-bridge knot). Then $\rho$ extends over a smooth or locally flat orientable surface $F \subseteq D^4$ if  and only if $n^2$ divides $\vert H_1(\Sigma_2K) \vert$.
\end{theorem}

Recall $\vert H_1(\Sigma_2K) \vert = \Delta(-1) \defeq det(K)$, where $\Delta(t)$ is the Alexander polynomial of $K$. Recall that a (regular) $2n$-fold dihedral branched cover of $K$ is the same as a $n$-fold unbranched cover of $\Sigma_2K$. The former exists if and only if $\pi_1(E_K)$ admits a quotient onto $D_n$, and the latter exists if and only if $n$ divides $\vert H_1(\Sigma_2K) \vert$, since $H_1(\Sigma_2K)$ is cyclic. The above theorem tells us that such a $\rho$ extends over a surface if and only if $\vert H_1(\Sigma_2K) \vert$ is also divisible by $n^2$.

\begin{proof}
$(\implies)$ Let $F\subseteq D^4$ be as in the theorem statement and $\Sigma_2F$ its two-fold branched cover. In other words, $(\Sigma_2F, \Sigma_2K)$ is the two-fold branched cover of the pair, $(F, K) \subseteq (D^4, S^3)$, and we have the following commutative diagram:
\begin{equation}\label{Zk}
\begin{tikzcd}
\Sigma_2 K \ar[r,two heads,"\tilde{\rho}"]\ar[d,"i_\ast"]  
		& B\Z_n.\\
\Sigma_2F \ar[ur,two heads]
\end{tikzcd}
\end{equation}
Here, $\tilde{\rho}$ is the lift of $\rho$ to the $2$-fold branched cover of $K \subseteq S^3$. Since $H_2(\Sigma_2K) \cong 0$ by duality and universal coefficients, the long exact sequence of the pair  $(\Sigma_2F, \Sigma_2K)$ is as shown below. 
\begin{equation}\label{duality}
\begin{tikzcd}
\{0\} \ar{r} & H_2(\Sigma_2F)  \ar{r} & H_2(\Sigma_2F, \Sigma_2K) \ar{r} & H_1(\Sigma_2K)\\
\ar{r} & H_1(\Sigma_2F) \ar{r} &  H_1(\Sigma_2F, \Sigma_2K) \ar{r} & \{0\}.
\end{tikzcd} 
\end{equation}
Obtaining the desired conclusion from the above exact sequence is an exercise. We provide the details. 

Let
\[
I:=\Im\big(H_1(\Sigma_2K)\to H_1(\Sigma_2F)\big).
\]
Then \(|I|\) is divisible by \(n\), since the composition
\[
H_1(\Sigma_2K)\xrightarrow{i_\ast} H_1(\Sigma_2F)\xrightarrow{\tilde{\rho}_\ast} H_1(B\Z_n)\cong \Z_n
\]
is surjective.

Also, by exactness,
\[
0\to I \to H_1(\Sigma_2F)\to H_1(\Sigma_2F,\Sigma_2K)\to 0,
\]
so

\begin{equation}\label{eq:size-I}
|I|=\frac{|H_1(\Sigma_2F)|}{|H_1(\Sigma_2F,\Sigma_2K)|}.    
\end{equation}

Since the groups below are finite, duality and the Universal Coefficient Theorem imply
{\small
\begin{equation}\label{dduality}
|H_2(\Sigma_2F)|=|H_1(\Sigma_2F,\Sigma_2K)|
\qquad\text{and}\qquad
|H_2(\Sigma_2F,\Sigma_2K)|=|H_1(\Sigma_2F)|.
\end{equation}
}

Hence, Equation~\eqref{eq:size-I} can also be written as
\[
|I|=\frac{|H_2(\Sigma_2F,\Sigma_2K)|}{|H_2(\Sigma_2F)|}.
\]
Thus, exactness of \eqref{duality} gives
\[
|H_1(\Sigma_2K)|
=|I|\cdot \frac{|H_2(\Sigma_2F,\Sigma_2K)|}{|H_2(\Sigma_2F)|}
=|I|^2.
\]
Therefore, $n^2$ divides $|H_1(\Sigma_2K)|$.

$(\impliedby)$ We now assume that $(K, \rho)$ is a ${D}_n$-knot with $H_1(\Sigma_2K)$ cyclic and such that $n^2$ divides $m:=\vert H_1(\Sigma_2K) \vert$. 
As recalled earlier, we may regard $\rho$ as induced by a map $\tilde{\rho}: H_1(\Sigma_2 K)\twoheadrightarrow \Z_n,$ which in this case factors as follows:
\begin{equation}
\label{eq:factorization}
H_1(\Sigma_2 K)\cong \Z_m \twoheadrightarrow \Z_{n^2} \twoheadrightarrow \Z_n
\end{equation}
We also have that the map $H_1(\Sigma_2 K)\twoheadrightarrow \Z_{n^2},$ is given by linking with an element, $[c']$, determined as in Corollary~\ref{cor:correspondence}.  
That is, by Lemma~\ref{lemma:linking-correspondence}, we have that $[c'] \in H_1(\Sigma_2 K)$ is the unique element with the property that, for all $[x] \in H_1(\Sigma_2 K)$, the map $\Z_m\to \Z_{n^2}$ above is given by $[x]\mapsto \lk([c'], [x])$. Regarding both groups as embedded in $\Z_{(2)}/\Z$, the map 
\[
\Z_{n^2}=\langle \frac{1}{n^2} \rangle\longrightarrow \Z_n=\langle \frac{1}{n}\rangle
\]
sends $\frac{1}{n^2}\in \Z_{(2)}/\Z$ to $\frac{1}{n}\in \Z_{(2)}/\Z$. Thus, we have, denoting inclusion into $\Z_{(2)}/\Z$ by $\iota$, 
\[
\iota(\tilde{\rho}_*([x])) = n \cdot \lk([c'], [x])=\lk(n\cdot [c'], [x])=\lk([n\cdot c'], [x]).
\]
Hence, the class -- denote it $[c]$ -- corresponding to $\tilde{\rho}_*([x]): H_1(\Sigma_2 K)\to \Z_n$ under Corollary~\ref{cor:correspondence} and Lemma~\ref{lemma:linking-correspondence} satisfies $[c]=[nc']$. Thus,
\begin{equation}\label{eq:self-linking-is-right}
\lk([c],[c])=n^2\lk([c'],[c'])=0\in \Z_{(2)}/\Z,
\end{equation}
since the map  $[x]\mapsto \lk([c'], [x])$ has image $\Z_{n^2}$. 
By Theorem~\ref{thm:equivalence}, Equation~\eqref{eq:self-linking-is-right} implies that the homomorphism $\rho$ extends over a smooth or locally flat orientable $D_n$-surface. 
\end{proof}

\begin{remark}
    We may replace the dihedral group $D_n$ in the above theorem by its $n$-bottle counterpart, $\mathbb{D}_n$, thanks to Corollary~\ref{cor:char-knots-induce-both}, which gives a correspondence between dihedral and $n$-bottle quotients of knot groups; and Theorem~\ref{cor:Dn-Dn-bottle}, which tells us that one extends if and only if the other does.
\end{remark}

\begin{example}  Consider the knot $10_{87}$. The first homology of its double branched cover has order 81. It follows that $10_{87}$ 
is a $D_9$-knot {\em and} bounds a $D_9$-surface. Additionally, $10_{87}$ is a $D_3$-knot bounding a $D_3$-surface. It is also a $D_{27}$- and $D_{81}$-knot, but these quotients do not extend over (locally flat) surfaces.
\end{example}

\subsubsection{Applying the self-linking number condition} 
We illustrate how to apply Theorem~\ref{thm:equivalence}~\ref{thmpart:beta}. Since the knots in the next example are two-bridge, whether an extension exists can also be concluded using Theorem~\ref{thm:2-bridge}. That said, for these examples, computing the self-linking number of the characteristic knot has the additional consequence of obstructing an infinite family of topologically slice knots from being ribbon; we include this illustration of the usefulness of dihedral invariants. 

\begin{example}\label{ex:twist} Let $K_m$ denote the twist knot with $m$ half-twists (so $K_{-1}$ is the trefoil, $K_0$ the unknot and $K_1$ the stevedore knot). A Seifert matrix for $K_m$, obtained from a surface bounded by the standard diagram of $K_m$, is $\begin{bmatrix*}[r]  -1 & 1 \\ 0 & m\end{bmatrix*}$ and $|\Delta_{K_m}(-1)| = |4m+1|$. In particular, $K_m$ admits a quotient to $D_3$ if and only if $4m+1\equiv 0 \mod{3}$ if and only if $m \equiv 2 \mod{3}$. (Since twist knots have two bridges, if a $D_3$ quotient exists, it is unique up to conjugation in $D_3$. In particular, there is a unique class of mod~$3$ characteristic knots for each $K_m$ with $m\equiv 2\mod{3}$.)

We identify those characteristic knots next. The primitive homology class $\begin{bmatrix*}[r] -1 \\ 1    \end{bmatrix*}$ satisfies the condition of Definition~\ref{def:chark}:
\[
\begin{bmatrix*}[r]
    	-2&1 \\ 1&2m    \end{bmatrix*} \begin{bmatrix*}[r] -1 \\ 1    \end{bmatrix*} = \begin{bmatrix*}[r] 3 \\ 2m-1    \end{bmatrix*} \equiv \begin{bmatrix*}[r] 0 \\ 0    \end{bmatrix*} \mod{3}.
\]
Moreover, the class $\begin{bmatrix*}[r] -1 \\ 1    \end{bmatrix*}$ is represented by an unknot, denote it by $\beta_m$, in the Seifert surface for $K_m$. This observation will be useful later. 

To see for which $m$ the $D_3$-knot $K_m$ admits a $D_3$ surface, we may apply Theorem~\ref{thm:equivalence}~\ref{thmpart:beta}:
\begin{equation}\label{eq:twist-beta}
L_V(\beta, \beta)=
 \begin{bmatrix*}[r] -1&1 \end{bmatrix*} \begin{bmatrix*}[r] -2&1 \\ 1&2m    \end{bmatrix*} \begin{bmatrix*}[r] -1 \\ 1    \end{bmatrix*} = 2m-4  
\end{equation}
and $2m-4\equiv 0\mod{9} \iff m\equiv 2\mod{9}$.  Note that this agrees with the conclusion of Theorem~\ref{thm:2-bridge} since $|\Delta_{K_m}(-1)|= |4m+1|\equiv 0\mod{9} \iff m\equiv2\mod{9}$.

It follows for example that for the trefoil knot, which is the twist knot with parameter $m=-1$, this condition is not satisfied. Therefore, by Theorem~\ref{thm:n=3}, {\it the trefoil does not bound any} {\it dihedral surface}, locally flat or smooth; orientable or not; be it in $B^4$ or in another oriented $4$-manifold with $S^3$ boundary. We were somewhat astonished that we did not find this basic fact anywhere in the literature. 

Since Theorem~\ref{thm:2-bridge} is so easy to apply, it may seem that in this case we did a lot of unnecessary work evaluating the self-linking numbers of characteristic knots. However, this computation can also be used to conclude that 3-colorable twist knots (with the exception of $K_2=6_1$) are not ribbon. This conclusion falls out readily from the ribbon obstruction from~\cite{geske2018signatures}, Theorem~\ref{thm:GKS}. Since twist knots are 2-bridge, their Fox colorings are unique and the associated irregular dihedral covers are always $S^3$. (The covers admit Heegaard splittings of genus 0, as seen by lifting a bridge sphere for the knot.) Thus, whenever $K_m$ admits a Fox 3-coloring, Equation~\eqref{eq:ribbon-obstruction} reduces to 
\[
|\Xi_3(K_m)|\leq 1
\]
for any $K_m$ that is ribbon. On the other hand, Equation~\ref{eq:Xi} becomes
\[
\Xi_3(K_m)= \frac{4}{9}(2m-4)\pm 1.
\]
Here we use~(\ref{eq:twist-beta}), the fact that $\beta_m$ is unknotted, and\cite[Equation~2.20]{kjuchukova2018dihedral}, which gives that $\text{rank}(H_2(W(K_m,\beta_m))=1$ (since $n=3$ and we are working with a Seifert surface of genus one). Putting this together, when $m\equiv 2\mod{3}$ and $m\neq 2$, we have $|\Xi_3(K_m)|>1$, hence $K_m$ is not ribbon by~\cite{geske2018signatures}. We found some humor in the fact that this once groundbreaking realization can also be detected by examining the self-linking number of a characteristic knot.
\end{example}

\begin{remark} \label{rem:connected-sum} 
    We recall a standard fact, namely that quotients of knot groups to a fixed group $G$ are compatible with knot connected sum. Let $K_1, K_2$ be two knots in $S^3$ equipped with dihedral quotients $\rho_i: \pi_1(E_{K_i})\to G$ for $i=1, 2$. Assume that meridians of $K_1$ are mapped to the same conjugacy class as meridians of $K_2$. 
    The connected sum $K_1\# K_2$ admits a natural quotient $\rho_1\#\rho_2: \pi_1(E_{K_1\#K_2})\to G,$ defined as follows. From van Kampen's theorem, $\pi_1(E_{K_1\#K_2})$ is the amalgamated product $\pi_1(E_{K_1})\ast\pi_1(E_{K_2})/\langle \mu_1=\mu_2\rangle,$ where $\mu_i$ is the (Wirtinger) meridian of the strand of $K_i$ along which the connected sum is taken. Without loss of generality (by performing the connected sum along judiciously chosen meridians; and potentially composing $\rho_2$ with an inner automorphism of $G$) we may assume that $\rho_1(\mu_1)=\rho_2(\mu_2)$ so that the free product $\rho_1\ast\rho_2: \pi_1(E_{K_1})\ast\pi_1(E_{K_2})\to G$ descends to a map  $\pi_1(E_{K_1\#K_2})\to G$. 
    
    This basic fact can be illustrated diagrammatically via Fox colorings and their natural generalization to quotients to other groups. Recall that the homomorphism $\rho_i$ can be represented in any diagram $D(K_i)$ by assigning to each strand of $s$ of $D(K_i)$ the image of the Wirtinger meridian of $s$ under $\rho_i$. (These assignments are referred to as ``coloring'' the strands, after Fox. The strand colors at each crossings satisfy the Wirtinger relations at crossings, of course -- see~\cite{blair2025coxeter} for a review of this basic subject.) Performing the connected sum operation along strands labeled by the same element of $G$ -- again, two such strands always exist, possibly after composing with conjugation in $G$ (so-called permuting the colors) -- yields a valid coloring of the knot $K_1\# K_2$, which is a visual representation of the quotient $\rho_1\#\rho_2$ defined above. 
\end{remark}

\begin{example} Let $K$ be any knot equipped with a dihedral quotient $\rho: \pi_1(E_K)\to D_n$. The $n$-fold connected sum of $K$ with itself, $\#_nK$, admits a natural quotient $\rho^n:=\pi_1(E_{\#_nK})\to D_n$, defined by iterating the process described in Remark~\ref{rem:connected-sum}.  
The invariant $\Theta_{\rho^n}(\#_nK)$ always vanishes. That is, the $n$-fold connected sum $\#_nK$ bounds a dihedral surface. To see this, we will use Theorem~\ref{thm:equivalence}~\ref{thmpart:beta}. Fix a diagram $D$ of $K$ bounding a Seifert surface $S$ and let $\beta\subseteq S^\circ$ be a characteristic knot which determines the given quotient map $\rho$. The $n$-fold connected sum $\#_nK$ bounds $\natural_nS$, and the characteristic knot corresponding to $\rho^n$ can be chosen to be $\#_n\beta$. If $V$ is a Seifert matrix for $S,$ then a Seifert matrix for $\natural_nS$ is a block diagonal matrix $W$ whose blocks are $n$ copies of $V$. We thus see that 
\[
[\#_n\beta]^T (W+W^T)[\#_n\beta]=n([\beta]^T(V+V^T)[\beta])\equiv 0\mod n^2,
\]
where in the last step we use the fact that $[\beta]^T(V+V^T)[\beta]\equiv 0 \mod n$ since $\beta$ is a mod~$n$ characteristic knot. Thus, $\#_nK$ bounds a dihedral surface, as claimed.
\end{example}

\subsection{Computing the dihedral obstruction in general}\label{sec:QZcomp}
   We provide a general (detailed) procedure for evaluating our obstruction using the $\Q/\Z$ linking form on $\Sigma_2K$. As with much of Section~\ref{linkingpairing}, this will surely be familiar to experts, but we include it so that the computation of our obstruction is laid out in a self-contained manner.

\textit{Step 1.}
Given a Seifert matrix $V$ for $K$, set $A:=V+V^T$. Recall that $A$ is a presentation matrix for the homology of the double branched cover of $K$, that is, $H_1(\Sigma_2K)\cong \Z^m/A\Z^m$.

\textit{Step 2.}
Find the Smith Normal Form of $A$. That is, write 
\[
PAQ = D=\operatorname{diag}(1,\dots,1,d_1,\dots,d_r),
\]
where $P, Q \in GL_m(\Z)$ and $d_i|d_{i+1}$. We have $H_1(\Sigma_2K)\cong \Z_{d_1}\oplus\cdots\oplus \Z_{d_r}$.

\textit{Step 3.}
Let $k:=m-r$ and define torsion generators in $\Z^m/A\Z^m$ by
\[
\ell_i := Q e_{k+i}\qquad (i=1,\dots,r).
\]
Note that $[\ell_i]$ has order $d_i$.

\textit{Step 4.}
For $u,v\in \Z^m$, set
\[
\operatorname{lk}([u],[v]) \equiv u^T A^{-1} v \pmod{\Z}\in \Q/\Z.
\]
In the basis $(\ell_1,\dots,\ell_r)$, the $\Q/\Z$ linking form on $\Sigma_2K$ is presented by the matrix
\[
\Lambda=(\Lambda_{ij})\in(\Q/\Z)^{r\times r},\qquad 
\Lambda_{ij}:=\operatorname{lk}(\ell_i,\ell_j)\equiv \ell_i^T A^{-1}\ell_j \pmod{\Z}.
\]

\textit{Step 5.}
Write $c=\sum_{i=1}^r x_i\ell_i$ with $x_i\in \Z/d_i$.  Solve
\[
\operatorname{lk}(c,c)=\sum_{i,j=1}^r x_i x_j\,\Lambda_{ij}=0\in \Q/\Z.
\]

\textit{Step 6.}
For each solution $c$ found above, define the corresponding character $H_1(\Sigma_2K)\to \Z_n =\langle \frac{1}{n}\rangle\subseteq \Q/\Z$ by setting 
\[
\rho_c(\,\cdot\,):=\operatorname{lk}(c,\,\cdot\,)\in \Q/\Z.
\]

\textit{Step 7.}
To find the equivalence classes of characters (that is, the equivalence classes of epimorphisms $\pi_1(E_K)\twoheadrightarrow D_n$), quotient the set of solutions above by the action of $\operatorname{Aut}(\Z_n)\cong(\Z_n)^\times$.

\begin{example}\label{Ex-comp}
We illustrate the above computation on the knot $9_{37}$. 

\begin{figure}[H]
\centering
\includegraphics[width=.3\textwidth]{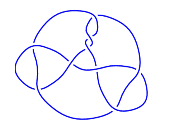}
\caption{The Poindexter Knot, $9_{37}$}
\end{figure}

 Denote its Seifert matrix by $V$. From Knotinfo~\cite{knotinfo}, we obtain:
\[
V=\begin{pmatrix*}[r]
1&0&0&0\\
0&1&0&0\\
1&0&-2&1\\
-1&-1&0&-1
\end{pmatrix*},
\qquad
H_1(\Sigma_2K)\cong \Z_3\oplus \Z_{15}.
\]

\textit{Step 1.} $A=V+V^T$.
\begin{comment}
\[
A=S+S^T=
\begin{pmatrix}
2&0&1&-1\\
0&2&0&-1\\
1&0&-4&1\\
-1&-1&1&-2
\end{pmatrix}.
\]
\end{comment}

\textit{Step 2.}
The Smith Normal form of the matrix $A$ is $\operatorname{diag}(1,1,3,15)$.

\textit{Step 3.}
We obtain the following generators for $H_1(\Sigma_2K)$: 
\[
\ell_1=(-1,-1,0,0)^T ~~\text{ and }~~ \ell_2=(-1,0,0,0)^T.
\]
Note that $\operatorname{ord}([\ell_1])=3$ and $\operatorname{ord}([\ell_2])=15$.

\textit{Step 4.}
\[
A^{-1}=
\begin{pmatrix*}[r]
\frac25&-\frac1{15}&\frac1{15}&-\frac2{15}\\
-\frac1{15}&\frac25&-\frac1{15}&-\frac15\\
\frac1{15}&-\frac1{15}&-\frac4{15}&-\frac2{15}\\
-\frac2{15}&-\frac15&-\frac2{15}&-\frac25
\end{pmatrix*},
\qquad
\Lambda=
\begin{pmatrix*}[r]
\frac23 & \frac13\\[4pt]
\frac13 & \frac25
\end{pmatrix*}\in (\Q/\Z)^{2\times 2}.
\]

\textit{Step 5.}
For $c=x\ell_1+y\ell_2$ with $x\in \Z/3$, $y\in \Z/15$,
\[
\operatorname{lk}(c,c)=\frac23x^2+\frac23xy+\frac25y^2\in \Q/\Z.
\]
The solutions to $\operatorname{lk}(c,c)=0$ are
\[
(0,0),\ (0,5),\ (0,10),\ (1,5),\ (2,10)\in (\Z/3)\oplus(\Z/15).
\]

\textit{Step 6.}
For $c=(a,b)$,
\[
\rho_{(a,b)}(x,y)=\begin{pmatrix}a&b\end{pmatrix}\Lambda\binom{x}{y}\in \Q/\Z.
\]
For the four nonzero solutions in the previous step, we get:
\[
\rho_{(0,5)}(x,y)\equiv \frac{2x}{3},\qquad
\rho_{(0,10)}(x,y)\equiv \frac{x}{3},
\]
\[
\rho_{(1,5)}(x,y)\equiv \frac{x+y}{3},\qquad
\rho_{(2,10)}(x,y)\equiv \frac{2(x+y)}{3}.
\]
Each $\rho$ has image $\Z_3$.

\textit{Step 7.}
Since $(\Z_3)^\times=\{1,2\}$, we have $\rho\sim 2\rho$, hence
\[
\frac{x}{3}\sim \frac{2x}{3},\qquad \frac{x+y}{3}\sim \frac{2(x+y)}{3},
\]
so there are exactly \emph{two} equivalence classes of nontrivial $\Z_3$-valued characters with $\operatorname{lk}(c,c)=0$. By Theorem~\ref{thm:equivalence}, both of these representations extend.
\end{example}


\section{Acknowledgements}
During the course of this work, AK was partially supported by NSF DMS 2204349, NSF DMS 1821257 and a Simons Foundation Travel grant for mathematicians. KO was partially supported by Simons Foundation funds. Some of this work was completed while the authors were at the Max Planck Institute in Bonn. We thank the MPIM for its support and hospitality.

We are especially grateful to Julius L. Shaneson for sharing his many insights. AK is especially indebted to him for encouraging her to pursue her linking number mod~$n^2$ conjecture, which, many years ago, gave this work its initial impetus.

We thank Jonathan Hillmann for a helpful discussion concerning our notion of a {\em taut} group, and for directing us to Theorem~1 in J. P. Levine's seminal paper, {\em Knot Modules I}.

\section*{The Red Asterisk}

\textit{Which dihedral homoms have the flair}\\
\textit{A surface in the four-ball to ensnare?}\\
\textit{Since the group is taut}\\
\textit{The characteristic knot}\\
\textit{Must self-link zero, mod $n$'s square.}

\bibliographystyle{amsalpha}
\renewcommand{\MR}[1]{}
\bibliography{res.bib}
\end{document}